%% file: llg_bdf2_a_priori_estimates.tex
\documentclass[11pt,reqno]{NumPDEsArticle}

\usepackage{NumPDEsMacros}

\usepackage{csquotes}
\usepackage{enumitem}
\usepackage{amsmath,amssymb}
\usepackage{empheq}
\usepackage{graphicx}
\usepackage{tikz}
\usepackage{caption}
\usepackage{subcaption}
\usepackage[capitalize,nameinlink]{cleveref}
\usepackage{hyperref}
\usepackage{cleveref}

\DeclareMathAlphabet\mathbfcal{OMS}{cmsy}{b}{n}

\title[BD2-type integrator for LLG, part II]{BDF2-type integrator for Landau-Lifshitz-Gilbert equation in micromagnetics: a-priori error estimates}

\author{Michele Ald\'e}
\email{Michele.Alde@asc.tuwien.ac.at \quad\rm (corresponding author)}

\author{Michael Feischl}
\email{Michael.Feischl@asc.tuwien.ac.at}

\author{Dirk Praetorius}
\email{Dirk.Praetorius@asc.tuwien.ac.at}

\address{TU Wien, Institute of Analysis and Scientific Computing, Wiedner Hauptstra\ss{}e 8--10, 1040 Wien, Austria}

\thanks{This research was funded by the Austrian Science Fund (FWF) projects \href{https://www.doi.org/10.55776/F65}{10.55776/F65} (SFB F65 ``Taming complexity in PDE systems''), and \href{https://www.doi.org/10.55776/P33216}{10.55776/P33216} (stand-alone project P33216 ``Computational nonlinear PDEs'') and through the ERC grant \href{https://www.doi.org/10.3030/101125225}{10.3030/101125225} ``New frontiers in optimal adaptivity'' of Michael Feischl. Additionally, the \href{https://www.vsmath.at}{Vienna School of Mathematics} supports Michele Ald\'e.}

\newtheorem*{proposition*}{Proposition}

\newcommand{\mm}{\boldsymbol{m}}
\newcommand{\ee}{\boldsymbol{e}}
\newcommand{\heff}{\boldsymbol{h}_{\rm eff}}
\newcommand{\ff}{\boldsymbol{f}}
\newcommand{\vv}{\boldsymbol{v}}
\newcommand{\zz}{\boldsymbol{z}}
\newcommand{\dd}{\boldsymbol{d}}
\newcommand{\qq}{\boldsymbol{q}}
\newcommand{\rr}{\boldsymbol{r}}
\newcommand{\sss}{\boldsymbol{s}}
\newcommand{\bg}{\boldsymbol{g}}
\newcommand{\oomega}{\boldsymbol{\omega}}
\newcommand{\bD}{\boldsymbol{D}}
\newcommand{\pphi}{\vvarphi}
\newcommand{\bphi}{\pphi}
\newcommand{\bpsi}{\boldsymbol{\psi}}
\newcommand{\lamex}{\lambda_{\rm ex}}
\newcommand{\bT}{\boldsymbol{T}}
\newcommand{\bTh}{\boldsymbol{T}_{\!h}}
\newcommand{\bL}{\mathbf{L}}
\newcommand{\bH}{\mathbf{H}}
\newcommand{\bP}{\mathbf{P}}
\newcommand{\bPh}{\mathbf{P}_{\!h}}
\newcommand{\bR}{\mathbf{R}}
\newcommand{\bW}{\mathbf{W}}
\newcommand{\bI}{\mathbf{I}}
\newcommand{\OT}{\Omega_T}
\newcommand{\mhat}{\widehat{\mm}}
\newcommand{\mhatp}{\widehat{\mm}_{h\tau}^+}
\newcommand{\mstarh}{{\mm}_{\star,h}}
\newcommand{\mstarhhat}{\widehat{\mm}_{\star,h}}
\newcommand{\vstarh}{\vv_{\star,h}}

\newcommand{\scalarproductO}[2]{\langle #1, #2 \rangle_{\Omega}}

\newcommand{\normO}[2][]{\norm[#1]{#2}_{\Omega}}
\newcommand{\normOT}[2][]{\norm[#1]{#2}_{\Omega_T}}
\newcommand{\bN}{\mathbf{N}}

\renewcommand{\set}[3][]{#1\{#2 \,:\, #3#1\}}
\newcommand{\xx}{\boldsymbol{x}}
\newcommand{\uu}{\boldsymbol{u}}
\def\III{\boldsymbol{\II}}
\def\SSS{\boldsymbol{\SS}}
\def\pphi{\boldsymbol{\phi}}
\def\vvarphi{\boldsymbol{\varphi}}

\newcounter{constantcnt}
\makeatletter

\newcommand{\newconst}[1]{%
    \@ifundefined{cst@#1}{%
        \@ifundefined{ifmeasuring@}{%
            \stepcounter{constantcnt}%
            \expandafter\xdef\csname cst@#1\endcsname{\theconstantcnt}%
            C_{\theconstantcnt}%
        }{%
            \ifmeasuring@
                C_{0}
            \else
                \stepcounter{constantcnt}%
                \expandafter\xdef\csname cst@#1\endcsname{\theconstantcnt}%
                C_{\theconstantcnt}%
            \fi
        }%
    }{%
        C_{\@nameuse{cst@#1}}%
    }%
}

\newcommand{\const}[1]{%
    \@ifundefined{cst@#1}{%
        \mathbf{C_{??}}
    }{%
        C_{\@nameuse{cst@#1}}%
    }%
}
\makeatother

\begin{document}
	
\maketitle

\begin{abstract}
We consider the Landau--Lifshitz--Gilbert equation (LLG), which models time-dependent micromagnetic phenomena. 
We analyze a fully discrete scheme that combines first-order finite elements in space with a BDF2 method in time. 
The method requires the solution of only one linear system of equations per time step and does not enforce the pointwise unit-length constraint of the magnetization.
While unconditional weak convergence has been analyzed in an earlier work, we now prove optimal-order convergence rates under sufficient regularity assumptions on the exact solution and the external field. 
In combination with our previous work, this establishes the first linear higher-order-in-time integrator that converges both to weak and strong solutions of LLG. Numerical experiments confirm first-order convergence in space and second-order convergence in time.
\end{abstract}

\section{Introduction}
Time-dependent micromagnetic phenomena are typically modeled by the phenomenological Landau--Lifshitz--Gilbert equation (LLG), a nonlinear time-dependent PDE that describes the evolution of the vector field magnetization $\mm(t,\xx)\in \R^3$ in a ferromagnetic body $\Omega \subset \R^d$, $d=2,3$.
The magnetization represents the magnetic moment per unit volume and can be interpreted as the continuum counterpart of discrete magnetic spins.
It is a fundamental physical property of the model that, for constant temperatures below the Curie temperature, the modulus of the magnetization remains constant, i.e., $|\mm(t,\xx)| = M_s$, where $M_s>0$ denotes the saturation magnetization.
The equation plays a central role in micromagnetics and underlies many areas such as sensors, actuators, memory storage devices such as hard disk drives and magnetic recording~\cite{LandauA1935,BrownB1962,BrownB1963,hubert1998magnetic}.

The analytical theory of LLG is by very well-developed.
We refer to~\cite{Alouges_Soyeur} for the global-in-time existence of weak solutions, to
\cite{csg1998,cf2001,cimrak2005,cimrak2007,ft2017b,MR4646547} for local-in-time existence of smooth strong solutions, and to
\cite{ds2014,dip2020} for weak-strong uniqueness, i.e., if a strong solution exists until a certain time, then every weak solution with the same initial data also coincides with it until that time. In particular, strong solutions are unique if they exist.

From the point of view of numerical analysis, discretization schemes of LLG may be grouped into schemes that converge, at a given rate, to sufficiently smooth strong solutions under sufficient regularity assumptions on the exact solution~\cite{prohl2001,bs2006,MR3273326,MR3542789,feischl,MR4454924,MR4362832,MR4673464,MR4617914} and schemes that only guarantee weak convergence, along a subsequence, to weak solutions of LLG~\cite{aj2006,bkp2008,bp2006,Alouges_2008,ruggeri2022}.
Some works also turn attention to coupled LLG systems~\cite{bppr2014,bpp2015,dpprs2020,akt2012,bsffgppr2014,mrs2018,hpprss2019,dppr2023,guo2026NSLLG}, employing an implicit-explicit (IMEX) discretization to reduce the computational cost.

One of the first works that proved unconditional weak convergence to weak solutions of LLG is \cite{Alouges_2008}, where a projection-based first-order in time tangent-plane integrator is proposed and analyzed. The main idea of the scheme is to exploit inherited properties of the equation to turn the non-linear PDE in the variable $\mm$ into a linear variational formulation for the variable $\vv \coloneqq \partial_t\mm$, so that only one linear system is solved at every time-step of the algorithm. This integrator relies on first-order finite elements in space and weak convergence to weak solutions of LLG was proved. Being a linear scheme for a non-linear equation, this method has therefore inspired a large number of subsequent works.
In particular, \cite{akst2014,dpprs2020} propose a variant of the algorithm which is still unconditionally weakly convergent to weak solutions and empirically of second order in time, but the corresponding \textit{a-priori} error analysis for smooth strong solutions remains open. More recently, the scheme was even extended to higher-order finite element space discretization and higher-order BDF$\ell$ ($\ell=1,\dots,5$) discretization in time in \cite{feischl}, where \textit{a-priori} error estimates in the presence of sufficiently smooth strong solutions were proved, while unconditional weak convergence was not addressed there and indeed remains open. Other higher-order discretizations in time were recently proposed in \cite{XieWang2025, Xie2025}.

So far, the first-order tangent-plane scheme is the only integrator which converges both weakly to a weak solution of LLG \cite{Alouges_2008} and, under suitable regularity assumptions on the exact solution, with optimal convergence rates to the exact solution \cite{AnLiSun2024}. A similar weak-and-strong convergence picture is available also to a projection-free variant of the scheme, which adapts ideas for projection-free integrators for harmonic maps \cite{bartels2016} to LLG; see \cite{Spin-polarized} for the weak convergence and an IMEX time-stepping and \cite{ft2017} for the \textit{a-priori} analysis. In both cases, the schemes are first-order in both time and space. Related work includes \cite{Li2026}.

This work is the second of a series of three papers, where we aim to formulate and analyze these properties for an integrator which has higher-order accuracy in time, employing the BDF2 time-stepping discretization.  Based on the connection between LLG and the harmonic map heat flow, our previous work~\cite{part1} adapted~\cite{abp24}, which proved unconditional weak convergence for a projection-free BDF2 scheme applied to the harmonic map heat flow. This paper is complementary to~\cite{part1} in the following sense:~\cite{part1} proves unconditional weak convergence to weak solutions, while the present paper proves optimal-order \textit{a-priori} error estimates for a sufficiently regular strong solution of LLG, if the latter exists and is sufficiently smooth. To the best of our knowledge, this is the first second-order in time linear integrator for LLG for which this full picture is available. While these two papers focus on the mathematical core difficulties and thus restrict to a simplified effective field consisting only of exchange and exterior field, they will be completed by an upcoming preprint~\cite{part3}, where the integrator is extended via an IMEX strategy from the mathematical core model to the more physically relevant fully effective field which includes further energy contributions.

The analysis in the present paper combines ideas from~\cite{feischl} and~\cite{bkw2024}.
A key difference to~\cite{feischl} lies in the construction of the discrete tangent space.
In~\cite{feischl}, the fully discrete scheme employs a variational definition of the discrete tangent space $\widetilde{\bT}_h(\mm)$ in which the orthogonality constraint is enforced only weakly, namely in the averaged sense
\(
\Pi_h(\mm \cdot \bphi_h) = 0
\)
with the $L^2$-orthogonal projection $\Pi_h$ onto the finite element tangent space $\widetilde{\bT}_h(\mm)$.
In contrast, our algorithm \cite{part1} follows the seminal works \cite{bp2006,Alouges_2008} and employs a more efficient node-wise variant of the discrete tangent space $\bT_h(\mm)$ in which the orthogonality constraint is enforced pointwise at every node of the triangulation. This choice is more efficient, but also tailored to lowest-order finite elements in space.
In Section~\ref{sec: preliminary results}, we show that the associated orthogonal projection operator onto $\bT_h(\mm)$ satisfies suitable approximation and stability properties that are crucial for the convergence analysis.
This node-wise construction of the discrete tangent space has also been used in~\cite{bkw2024} for the numerical analysis of the harmonic map heat flow into spheres, a problem that is closely related to the LLG equation.
However, the scheme in~\cite{bkw2024} is only first order in time, whereas in the present work we extend this approach to a second-order accurate BDF2 time discretization. 

The resulting scheme of the present paper uses a predictor-corrector structure with a BDF2 corrector step and a node-wise orthogonal discrete tangent space.
As in~\cite{Alouges_2008} and the subsequent literature, each step of the algorithm leads to a sequence of approximations $\vv_h^j\approx \partial_t\mm(t_{j+1})$ at discrete times $t_{j+1}$, which are then employed to construct $\mm_h^{j+1}\approx \mm(t_{j+1})$.
The analysis follows a classical argument, where consistency and stability imply convergence with optimal rates. However, ensuring that the intrinsic physical properties of LLG are preserved by the proposed structure-preserving scheme makes the analysis highly non-trivial.

\textbf{Outline.}
%
The outline of the present work reads as follows: 
In Section~\ref{section:model}, we formulate the LLG equation \eqref{eq: LLG system} in the presence of a simplified effective field consisting only of an applied exterior field and the leading-order exchange contribution. Sections~\ref{section:discretization}--\ref{section:numerical integrator} present the proposed numerical integrator in detail (Algorithm~\ref{alg: full discr}). In Section~\ref{section: weak convergence}, we recall the definition of weak solutions to the LLG equation from \cite{Alouges_Soyeur}, along with Theorem~\ref{thm: main result part1} from \cite{part1}, which establishes weak convergence of the scheme to weak solutions of LLG \eqref{eq: LLG system}. The main result of this paper is stated in Section~\ref{sec: main result} (Theorem~\ref{thm: main result}), where we prove optimal convergence rates for the proposed integrator under suitable regularity assumptions on the exact solution. After recalling some preliminary lemmas in Section~\ref{sec: preliminary results}, the convergence proof of Theorem~\ref{thm: main result} is divided into two parts: consistency in Section~\ref{sec: consistency} and stability together with the main argument in Section~\ref{sec: stability and main proof}. Finally, Section~\ref{section:numerics} provides numerical experiments that support the theoretical findings.

\textbf{General notation.}
%
Throughout this paper, let $\Omega\subset \R^d$ with $d=2,3$ be a bounded domain. We adhere to the standard notation for Lebesgue, Sobolev, and Bochner spaces, using bold letters to indicate spaces of vector valued functions, e.g., $\bL^2(\Omega) \coloneqq L^2(\Omega,\R^3)$ and $\bH^1(\Omega)\coloneqq H^1(\Omega,\R^3)$.
Given $\uu, \vv \in \bL^2(\Omega)$, we write 
\begin{align*}
 \langle\uu,\vv\rangle_\Omega 
 \coloneqq 
 \langle\uu,\vv\rangle _{\bL^2(\Omega)}\coloneqq \int_\Omega\uu(\xx)\cdot\vv(\xx)\,\d\xx
\end{align*}
for the $\bL^2$-scalar product in $\Omega$ and $\normO{\,\cdot\,}$ for the induced norm, i.e., $\normO{\uu} \coloneqq \langle\uu,\uu\rangle_\Omega^{1/2}$.
Similar notation is applied for the space-time domain $\OT\coloneqq \Omega \times [0,T]$ and the corresponding scalar product $\langle\cdot,\cdot\rangle_{\OT}$ and norm $\normOT{\,\cdot\,}$. To abbreviate the notation for norms in space-time Bochner spaces, we write, e.g., $\norm{\cdot}_{L^\infty(\bH^1)} \coloneqq \norm{\cdot}_{L^\infty(0,T;\bH^1(\Omega))}$ and $\norm{\cdot}_{L^2(\bL^2)} \coloneqq \norm{\cdot}_{L^2(0,T;\bL^2(\Omega))}$.
 All other norms will be explicitly stated.
We write $A\lesssim B$ to indicate that $A\leq C \, B$, and the dependency of the hidden constant $C > 0$ is either clear from the context or carefully specified. All constant dependencies are tracked in the proofs, and the final dependencies of the constants in the main results are explicitly stated.
Finally, $A \simeq B$ abbreviates $A \lesssim B \lesssim A$.
\section{Numerical BDF2 integrator and main result}\label{sec: main result}

\subsection{Mathematical model}\label{section:model}

Let $\Omega\subset \R^d$, $d=2,3$ be a bounded Lipschitz domain describing a ferromagnetic body. We consider the (simplified) effective field $\heff=\lamex^2\Delta\mm + \ff$, corresponding to the total Gibbs free energy 
\begin{equation}\label{eq: Gibbs free energy}
	\mathcal{E}(\mm,\ff) = \frac{\lamex^2}{2}\int_\Omega |\nabla \mm|^2 \, \d \xx - \int_\Omega \ff \cdot \mm \, \d \xx,	
\end{equation}
where $\ff\colon \Omega_T\to\R^3$ is a given external field and $\lamex>0$ is the so-called exchange constant. While the effective field is simplified with respect to the general physical model (see, e.g., \cite{mrs2018,hpprss2019,davoli2020}), we stress that, mathematically, the leading-order term and hence the main mathematical challenge is indeed included. The evolution of the normalized magnetization $\mm\colon \Omega_T\to\mathbb{S}^2\coloneqq\{\xx\in\R^3 \colon |\xx|=1\}$ is governed by the Landau--Lifshitz--Gilbert equation (LLG): Given a final time $T>0$, we seek $\mm$ satisfying
\begin{subequations}\label{eq: LLG system}
	\begin{empheq}[left=\empheqlbrace]{align}
	 &\partial_t\mm = -\mm\times\heff + \alpha \mm\times\partial_t\mm 
	  &&\hspace{-2em} \text{in} \quad \Omega \times [0,T] \label{eq: LLG}, \\
	 &\partial_{\vec{n}}\mm = 0 
	  &&\hspace{-2em} \text{in} \quad \partial\Omega \times [0,T] \label{eq:subeq_b}, \\
	 &\mm(0) = \mm^0 
	  &&\hspace{-2em} \text{in} \quad \Omega \label{eq:subeq_c},
	\end{empheq}
\end{subequations}
where $\alpha>0$ is the so-called Gilbert damping parameter. 

Multiplying	\eqref{eq: LLG} by $\mm$, we note that $\partial_t|\mm|^2 = 2\mm\cdot\partial_t\mm = 0,$ in $\OT$. If the initial condition \eqref{eq:subeq_c} satisfies $|\mm^0|=1$, then it is PDE-inherent that the solution $\mm$ satisfies $|\mm|=1$ in $\OT$. Furthermore, the time-derivative $\partial_t\mm$ is pointwise orthogonal to $\mm$ in $\OT$. This built-in orthogonality suggests introducing the tangent space 
\begin{equation}\label{eq: tangent space}
\bT(\mm(t))\coloneqq\{\vvarphi\in\bH^1(\Omega)\colon \mm(t)\cdot\vvarphi=0 \text{ a.e.}\}\ni \partial_t\mm(t) \quad \text{at each time} \quad t\in~[0,T].
\end{equation}

\subsection{Discretization}\label{section:discretization}

We partition the time domain $[0,T]$ into $N$ intervals of length $\tau \coloneqq T/N$ with time steps $t_j \coloneqq j\tau$ for all $j = 0, \ldots, N$.
Given a sequence of functions $\{\mm^j\}_{j=0}^N$ associated with these time steps, we define the discrete time derivative
\begin{align*}
 \d _t{\mm}^{j+1} 
 \coloneqq 
 \frac{\mm^{j+1}-\mm^j}{\tau}
\end{align*}
and the second-order discrete time derivative
\begin{equation}\label{eq: second order discrete time derivative}
 \d_t^2{\mm}^{j+1}
 \coloneqq 
 \frac{\d _t \mm^{j+1}-\d _t \mm^j}{\tau}
 =
 \frac{\mm^{j+1}-2\mm^j+\mm^{j-1}}{\tau^2}.
\end{equation}
For the spatial discretization, we employ a $\gamma$-quasi-uniform family of conforming simplicial meshes $\mathcal{T}_h$ of $\Omega$ with mesh size $h > 0$, i.e., there exists a constant $\gamma > 0$ such that 
\begin{align*}
 \gamma^{-1} h \le |K|^{1/d} \le \text{diam}(K) \le h
 \quad \text{for all } K \in \mathcal{T}_h \text{ and all } h > 0.
\end{align*}
We use the finite element space of continuous, $\mathcal{T}_h$-piecewise affine vector-valued functions
\begin{align*}
 \boldsymbol{\mathcal{S}}^1(\mathcal{T}_h) \coloneqq \set{\vv_h \in C(\Omega; \R^3)}{\forall K \in \mathcal{T}_h: \vv_h|_{K} \text{ is affine}},
\end{align*}
and denote as $\mathcal{N}_h$ the set of nodes of the triangulation $\mathcal{T}_h$.

To mimic the unit-length constraint at a discrete level, we introduce
\begin{align*}
 \boldsymbol{\mathcal{M}}_h 
 \coloneqq 
 \set{\bphi_h\in\SSS^1(\TT_h) }{|\bphi_h(\vec{z})|=1 \text{ for all } \vec{z}\in\NN_h}
\end{align*}
as the set of discrete magnetizations satisfying the unit-length constraint at all nodes of $\mathcal{T}_h$.
Analogously, also the discrete orthogonality constraint is imposed at every node $\zz\in \mathcal{N}_h$ via the discrete tangent space, defined at $\pphi_h\in\SSS^1(\TT_h)$ as 
\begin{equation*}
 \bT_h(\pphi_h) 
 \coloneqq 
 \set[\big]{\bpsi_h\in\SSS^1(\TT_h)}{\bphi_h(\vec{z}) \cdot \bpsi_h(\vec{z}) = 0 \text{ for all } \vec{z} \in \NN_h}.
\end{equation*}
\subsection{Numerical integrator}\label{section:numerical integrator}
Under the constraint $|\mm|=1$, it is possible to prove that~\eqref{eq: LLG} can be rewritten as 
\begin{equation}\label{eq: alternative LLG}
 \alpha\partial_t\mm + \mm\times\partial_t\mm = \heff - (\mm\cdot\heff)\mm,
\end{equation}
see, e.g., \cite{goldenits2012konvergente} for the elementary calculation. To derive the discrete variational formulation for the numerical integration of~\eqref{eq: LLG system}, we note that the time derivative $\partial_t\mm(t)\in\bT(\mm(t))$ satisfies, for all test functions $\vvarphi\in\bT(\mm(t))$,
\begin{equation}\label{eq: weak formulation}
	\alpha\scalarproductO{\partial_t\mm(t)}{\vvarphi} + \scalarproductO{\mm(t)\times\partial_t\mm(t)}{\vvarphi} + \lamex^2\scalarproductO{\nabla\mm(t)}{\nabla \vvarphi} = \scalarproductO{\ff(t)}{\vvarphi},
\end{equation}
which follows from \eqref{eq: alternative LLG} by multiplying LLG with $\vvarphi$ and integrating over $\Omega$.

The following Algorithm~\ref{alg: full discr} states the time-marching scheme from \cite{part1}, which exploits a finite element discretization of~\eqref{eq: weak formulation} for the time-derivative $\vv \coloneqq \partial_t\mm\in\bT(\mm)$ and uses second-order backward differences (BDF2) in time. We use discrete initial data $\mm_h^0\approx\mm^0$ and assume that 
\begin{equation}\label{eq: assumption m0}
	\norm{\mm_h^0-\mm^0}_{\bH^1(\Omega)}\le C_0 h
\end{equation}
for a constant $C_0>0$ independent of $h$.
Since BDF2 requires an additional initial value $\mm_h^1 \approx \mm(t_1)$ beyond $\mm_h^0 \approx \mm^0$, we employ one step of the projection-free first-order tangent-plane integrator from \cite{Spin-polarized}. We recall from \cite{ft2017} that the error in the first time step satisfies
\begin{equation}\label{eq: assumption m1}
	\norm{\mm(t_1)-\mm_h^1}_{\bH^1(\Omega)}\le \newconst{feischl_tran} (h+\tau^2).
\end{equation}

Successively, for $j=1,\dots,N-1$, we define a predictor term $\mhat_h^{j+1}\coloneqq 2\mm_h^{j}-\mm_h^{j-1}$, solve a linear system to approximate $\partial_t\mm_h(t_{j+1})\approx\vv_h^{j}\in\bTh(\mhat_h^{j+1})$, and use $\vv_h^{j}$ to approximate $\mm(t_{j+1})\approx \mm_h^{j+1}\in \SSS^1(\mathcal{T}_h)$.
\begin{algorithm}[\texorpdfstring{\cite[Algorithm A]{part1}}{[Algorithm A, part1]}]\label{alg: full discr}
	\textbf{Input:} Conforming mesh $\TT_h$ of $\Omega$, $\mm(0) \approx \mm_h^0\in \boldsymbol{\MM}_h$, $T>0$, $N\in\N$, $\tau\coloneqq T/N$, $t_j\coloneqq j\tau$, and $\ff_h^j\approx \ff(t_j)$ for all $j=0,\dots,N$.
	\begin{enumerate}[label=\rm (\roman*), ref=\rm \roman*]
	\item\label{eq: first step full discr item} 
	For $j=0,$ 
	compute $\vv_h^0\in \bTh(\mm_h^0)$ such that, for all $\pphi_h\in \bTh(\mm_h^0)$,
	\begin{equation}\label{eq: first step full discr}
	 \alpha\scalarproductO{\vv_h^0}{\pphi_h} + \scalarproductO{\mm_h^0\times\vv_h^0}{\pphi_h} + \lamex^2\tau\scalarproductO{\nabla\vv_h^0}{\nabla\pphi_h}
	 =
	 \scalarproductO{\ff_h^1}{\pphi_h} - \lamex^2\scalarproductO{\nabla\mm_h^0}{\nabla\pphi_h}.
	\end{equation}
	\item\label{item: def mm1}
	Set $\mm_h^1\coloneqq \mm_h^0+\tau\vv_h^0$.
	\item\label{item: next steps}  
	For $j=1,\dots,N-1$, repeat the following steps~\eqref{item: def mhat}--\eqref{item: def mmhj}:
	\begin{enumerate}[label=\rm(\alph*), ref=\rm \alph*]
	 \item\label{item: def mhat}
	 Set $\widehat{\mm}_h^{j+1}=2\mm_h^{j}-\mm_h^{j-1}$.
	 \item\label{item: var form}
	 Compute $\vv_h^j\in \bTh(\widehat{\mm}_h^{j+1})$ such that, for all $\pphi_h\in \bTh(\widehat{\mm}_h^{j+1})$,
	 \begin{equation}\label{eq: full discr alg B}
	 \begin{split}
	 \alpha\scalarproductO{\vv_h^j}{\pphi_h} &+ \scalarproductO{\widehat{\mm}_h^{j+1}\times\vv_h^j}{\pphi_h} + \frac{2}{3}\lamex^2\tau\scalarproductO{\nabla\vv_h^j}{\nabla\pphi_h}
	  \\&
	  =
	  \scalarproductO{\ff_h^{j+1}}{\pphi_h} - \frac{1}{3}\lamex^2\scalarproductO{\nabla[4\mm_h^{j}-\mm_h^{j-1}]}{\nabla\pphi_h}.
	 \end{split}
	 \end{equation}
	 \item\label{item: def mmhj}
	 Set $\mm_h^{j+1}\coloneqq \dfrac{4}{3}\mm_h^{j}-\dfrac{1}{3}\mm_h^{j-1}+\dfrac{2}{3}\tau \vv_h^j$.
	\end{enumerate}
	\end{enumerate}
	\textbf{Output:} Sequences $\vv_h^j\approx\partial_t\mm(t_{j+1})$ and $\mm_h^{j+1}\approx\mm(t_{j+1})$ for all $j=0,\dots,N-1$.
\end{algorithm}
Thanks to the Lax--Milgram lemma, the linear systems in \eqref{eq: first step full discr}--\eqref{eq: full discr alg B} admit unique solutions $\vv_h^j$ so that Algorithm~\ref{alg: full discr} is indeed unconditionally well-defined.

\subsection{Unconditional weak convergence to weak solutions}\label{section: weak convergence}

In this subsection, we briefly recall the main convergence result of \cite{part1} for Algorithm~\ref{alg: full discr}, which states unconditional weak convergence to weak solutions of LLG~\eqref{eq: LLG system}. To this end, we recall the definition of weak solutions in the sense of \cite[Definition~1.2]{Alouges_Soyeur}.
{\begin{definition}[Weak solution]\label{def: weak solution}
Let $\mm^0\in\bH^1(\Omega)$ satisfy $|\mm^0|=1$ a.e.\ in $\Omega$.
A function $\mm: \OT \rightarrow \mathbb{R}^3$ is a \text{\rm weak solution} of~\eqref{eq: LLG system} if the following properties {\rm (i)--(iv)} are satisfied:
\begin{enumerate}[label={\rm(\roman*)}, ref={\rm\roman*}]
\item\label{def: weak solution i} 
$\mm \in  \bH^1(\OT) \cap L^{\infty}(0, T ; \bH^1(\Omega))$ and $|\mm|=1$ a.e.\ in $\OT$;
\item\label{def: weak solution ii}
$\mm(0)=\mm^0$ in the sense of traces;
\item\label{def: weak solution iii}
for all $\vvarphi \in \bH^1(\OT)$, it holds that 
$$
\begin{aligned}
& \int_0^T\scalarproductO{\partial_t \mm(t)}{\vvarphi(t)} \d t-\alpha \int_0^T\scalarproductO{\mm(t) \times \partial_t \mm(t)}{\vvarphi(t)} \d t \\
&= \lambda_{\mathrm{ex}}^2 \int_0^T\scalarproductO{\mm(t) \times {\nabla} \mm(t)}{{\nabla} \vvarphi(t)} \d t-\int_0^T\scalarproductO{\mm(t) \times \boldsymbol{f}(t)}{\vvarphi(t)} \d t;
\end{aligned}
$$
\item\label{def: weak solution iv} 
for a.e.\ $t' \in(0, T)$, $\mm$ satisfies the energy inequality 
$$
\begin{aligned}\label{eq: energy inequality}
\EE(\mm(t'), \boldsymbol{f}(t'))+\alpha \int_0^{t'}\normO[\big]{\partial_t \mm(t)}^2 \mathrm{~d} t+\int_0^{t'}\langle\partial_t \boldsymbol{f}(t), \mm(t)\rangle _\Omega \d t\leq \EE(\mm^0, \boldsymbol{f}(0)).
\end{aligned}
$$
\end{enumerate}
\end{definition}}
Using the approximations $\{\mm_h^j\}_{0\le j\le N}$ obtained by Algorithm~\ref{alg: full discr}, it is possible to define the space-time interpolands $\mm_{h\tau}\in\bH^1(\OT)$, $\mm_{h\tau}^{\pm},\mhatp\in L^2(0,T;\bH^1(\Omega))$, and $\ff_{h\tau}^+\in L^2(0,T; \bL^2(\Omega))=\bL^2(\OT)$ as follows:
For all $0\le j\le N-1$ and all $t\in[t_j,t_{j+1}),$ define
	\begin{align*}
		\mm_{h\tau}(t) & \coloneqq \frac{t-t_j}{\tau}\mm_h^{j+1} + \frac{t_{j+1}-t}{\tau}\mm_h^{j},\\
		\mm_{h\tau}^-(t) & \coloneqq \mm_h^j, \\
		\mm_{h\tau}^+(t) & \coloneqq \mm_h^{j+1},\\
		\widehat{\mm}_{h\tau}^+(t) & \coloneqq
			\begin{cases}
				\mhat_h^1\coloneqq\mm_h^0 & \text{for } j=0,\\
				\widehat{\mm}_h^{j+1} & \text{for } 1 \leq j \leq N-1,
			\end{cases} \\
		\ff_{h\tau}^{+}(t) & \coloneqq \ff_h^{j+1}.
	\end{align*}
The following theorem is the main result of \cite{part1} and states that the interpolands $\mm_{h\tau}$, $\mm_{h\tau}^\pm$ and $\mhatp$, converge weakly towards a weak solution of LLG~\eqref{eq: LLG system} as $(h,\tau)\to(0,0)$.
{\begin{theorem}[\texorpdfstring{\cite[Theorem~2]{part1}}{[Theorem~2, part1]}]\label{thm: main result part1}
For $h > 0$, let $\TT_h$ be a family of $\gamma$-quasi-uniform meshes of $\Omega$.
Suppose that $\mm_h^0$ satisfies \eqref{eq: assumption m0} and $\ff\in C^1([0,T], \bL^2(\Omega))$ such that $\norm{\ff-\ff_{h\tau}^+}_{\bL^2(\OT)}~\xrightarrow{h,\tau \to 0}~0$. Then, the sequences of discrete functions $\mm_{h\tau}$, $\mm_{h\tau}^\pm$ and $\mhatp$, admit subsequences (not relabeled) that unconditionally converge to a function $\mm$ which is a weak solution of LLG~\eqref{eq: LLG system} in the sense of Definition~\ref{def: weak solution}\eqref{def: weak solution i}--\eqref{def: weak solution iii}.
More precisely, it holds that $\mm_{h\tau}\rightharpoonup\mm$ weakly in~$\bH^1(\OT)$ (and hence $\mm_{h\tau} \to \mm$ strongly in $\bL^2(\OT)$) and $\mm_{h\tau},\mm_{h\tau}^\pm,\mhatp\stackrel{*}{\rightharpoonup}\mm$ in $L^\infty(0,T;\bH^1(\Omega))$ as $(h,\tau)\to(0,0)$,
where all convergences hold for the same subsequence.
Moreover, define 
\begin{equation}\label{eq: eta0n}
	\eta_0 \coloneqq \normO{\nabla\mm_h^1}^2-\normO{\nabla\mm_h^0}^2\qquad\text{and}\qquad\eta_n\coloneqq\normO{\nabla\mm_h^{n-1}}^2-\normO{\nabla\mm_h^{n}}^2,
\end{equation}
where $\N\ni n = \mathcal{O}(\tau^{-1})$ denotes the final time step.
Then, the limit $\mm$ satisfies also Definition~\ref{def: weak solution}\eqref{def: weak solution iv}, provided that 
\begin{equation}\label{eq: condition eta0n}
\eta_0+\eta_n\xrightarrow{(h,\tau)\to(0,0)} 0.
\end{equation}
In particular, \eqref{eq: condition eta0n} is guaranteed under the CFL-type condition $\tau = o(h^2)$, i.e., $\sqrt{\tau}h^{-1}\to0$ as $(h,\tau)\to(0,0)$, which is, however, only required for the first and last time step.
\end{theorem}
}

\subsection{Optimal \emph{a-priori} error estimate in the presence of a strong solution}\label{section:strong solutions}

Throughout this work, we suppose that the applied external field satisfies
\begin{subequations}\label{eq: regularity assumptions}
\begin{align}\label{eq: regularity f}
	\ff\in C^1([0,T],\bL^\infty(\Omega))
\end{align}
\text{
and that LLG \eqref{eq: LLG system} admits a strong solution satisfying} 
\begin{align}\label{eq: regularity}
	\mm\in C^3\big([0,T\,],\bH^1(\Omega)\big)\,\cap\, C^1\big([0,T\,],\bH^2(\Omega)\cap\bW^{1,\infty}(\Omega)\big)\,\cap \,C\big([0,T\,],\bW^{2,\infty}(\Omega)\big)
\end{align}
\end{subequations}
Note that LLG \eqref{eq: LLG system} satisfies a weak-strong uniqueness principle, i.e., any weak solution to LLG coincides with $\mm$ until time $T$; see \cite{ds2014, dip2020}. Moreover, the chain rule and integration in time prove that $\mm$ satisfies the energy equality, for all times $t'\in[0,T]$,
	\begin{equation*}\label{eq: energy equality}
		\begin{aligned}
			\mathcal{E}(\mm(0), \ff(0)) = \mathcal{E}(\mm(t'), \ff(t')) + \alpha \int_0^{t'}\normO[\big]{\partial_t \mm(t)}^2 \mathrm{~d} t + \int_0^{t'}\langle\partial_t \ff(t), \mm(t)\rangle _\Omega \d t. 
		\end{aligned}
	\end{equation*}

The following theorem is the main result of this work. It states optimal-order error bounds for the approximations $\mm_h^n$ derived by Algorithm \ref{alg: full discr}.
\begin{theorem}\label{thm: main result}
	Consider Algorithm \ref{alg: full discr} for a $\gamma$-quasi-uniform family of conforming simplicial meshes $\mathcal{T}_h$ of $\Omega$. Suppose the regularity assumptions \eqref{eq: regularity assumptions}. Suppose that the initial value $\mm_h^0$ satisfies the assumption \eqref{eq: assumption m0}
	and that $\ff$ satisfies $\max\limits_{n=0,\dots,N}\norm{\ff(t_n)-\ff_h^n}_{\bL^{\infty}(\Omega)}\le C_{\ff}h$ for some constant $C_{\ff}>0$. Let $C_{\rm CFL}>0$ and $0<\varepsilon<1$.
	Then, there exist constants $C_{\rm conv}$, $\bar{h}$, $\bar{\tau}>0$ such that, if $h \le \bar{h}$ and $\tau \le \bar{\tau}$ satisfy the (very mild) CFL-condition 
	$\tau\le C_{\rm CFL} h^{(1+\varepsilon)/4}$, then the error is bounded by
	\begin{equation}\label{eq: optimal order error estimate}
	\norm{\mm(t_n)-\mm_h^n}_{\bH^1(\Omega)}\le C_{\rm conv}(h+\tau^2)\quad \text{ for all }\quad t_n = n\tau \le T \quad \text{with} \quad n\ge 1.
	\end{equation}
	The constant $C_{\rm conv}$ depends only on $\gamma$, $\alpha$, $\lamex$, $|\Omega|$, $T$, $C_0$, $C_{\ff}$, $C_{\rm CFL}$, $\varepsilon$, as well as $\bar{h}$ and $\bar{\tau}$ and on the regularity assumptions \eqref{eq: regularity assumptions} on $\ff$and $\mm$, but is independent of $h, \tau$, and $n$.
\end{theorem}
\begin{corollary}\label{cor: optimal order error estimate interpoland}
	Under the assumptions of Theorem~\ref{thm: main result}, the space-time interpoland $\mm_{h\tau}\in C([0,T];\\\bH^1(\Omega))$ satisfies that, for all $h\le \bar{h}$ and $\tau \le \bar{\tau}$ with $\tau\le C_{\rm CFL} h^{(1+\varepsilon)/4}$,
	\begin{equation}\label{eq: optimal order error interpoland}
		\norm{\mm-\mm_{h\tau}}_{L^\infty([0,T];\bH^1(\Omega))}\le C_{\rm conv}'(h+\tau^2),
	\end{equation}
	where the constant $C_{\rm conv}'>0$ depends only on $C_{\rm conv}$ from Theorem~\ref{thm: main result} and on the regularity assumptions \eqref{eq: regularity assumptions} on $\ff$ and $\mm$.
\end{corollary}

	\begin{remark}
		According to the optimal-order error estimate \eqref{eq: optimal order error estimate}, the natural choice is $\tau = \OO(h^{1/2})$, which also satisfies the CFL-condition $\tau = \OO(h^{(1+\varepsilon)/4})$ for any $0<\varepsilon\leq1$.
	\end{remark}

\section{Preliminaries}\label{sec: preliminary results}
\subsection{Tangential projection}\label{section:projection}
Recall the tangent space $\bT(\uu)$ from~\eqref{eq: tangent space}. We define
$$
\mathcal{M}^+\coloneqq \set{\uu\in \bH^1(\Omega)}{|\uu| > 0 \,\text{ a.e. in } \,\Omega}.
$$
For all $\uu\in \mathcal{M}^+$, we define the pointwise orthogonal projection onto $\bT(\uu)$ as
$$
\bP(\uu)\colon \mathcal{M}^+\to \bT(\uu), \qquad \bP(\uu)\coloneqq \mathbf{I}-\frac{\uu\otimes\uu}{|\uu|^2}.
$$
\begin{remark}
Thanks to the definition of the tangential projection and the inherent constraint $|\mm|=1$ for the solution of LLG \eqref{eq: LLG system}, it holds that
$
\bP(\mm)= \mathbf{I}-\mm\mm^T.
$
Therefore, equation \eqref{eq: LLG} can be equivalently rewritten as 
\begin{equation}\label{eq: alternative LLG with P}
	\alpha\partial_t\mm + \mm\times\partial_t \mm = \heff-(\mm\cdot \heff)\mm =  \bP(\mm)\heff= \bP(\mm)(\lamex^2\Delta\mm+\ff).
\end{equation}
\end{remark}

Firstly, we prove stability estimates for the tangential projection $\bP(\cdot)$. To this end, we define the normalization operator
$$
	\bN\colon \mathcal{M}^+\to \bH^1(\Omega), \qquad \bN(\uu) \coloneqq \frac{\uu}{|\uu|}.
$$
The following lemma from \cite{bkw2024} shows boundedness and Lipschitz continuity of $\bN(\cdot)$. To make this manuscript self-contained, the proofs are given in Appendix~\ref{section: appendix}.
\begin{lemma}[Normalization bounds \texorpdfstring{\cite[Lemma~2.1, Lemma~2.2]{bkw2024}}{[Lemma~2.1, Lemma~2.2, bkw2024]}]\label{lem: normalization bounds}
	Let $\uu\in \bW^{1,\infty}(\Omega)$ and $\uu_h\in \boldsymbol{\mathcal{S}}^1(\mathcal{T}_h)$. Let $c>0$ such that $0<c\leq|\uu|,|\uu_h|\leq c^{-1}$ a.e. in $\Omega$. Then, it holds that
	\begin{subequations}\label{eq: normalization bounds}
		\begin{align}
		\norm{\bN(\uu)}_{\bL^{\infty}(\Omega)}&\le C_\bN\norm{\uu}_{\bL^{\infty}(\Omega)},\label{eq: normalization bounds 0}\\
		\norm[\big]{\nabla \bN(\uu)}_{\bL^\infty(\Omega)} &\le C_\bN \norm{\nabla \uu}_{\bL^\infty(\Omega)},\label{eq: normalization bounds i}\\
		\norm[\big]{\mathrm{D}_h^2 \bN(\uu)}_{\bL^\infty(\Omega)} &\le C_\bN \big(\norm{\nabla \uu}^2_{\bL^\infty(\Omega)} + \norm{\mathrm{D}_h^2 \uu}_{\bL^\infty(\Omega)}\big),\label{eq: normalization bounds ii}\\
		\norm[\big]{\mathrm{D}_h^2 \bN(\uu_h)}_{\bL^\infty(\Omega)} &\le C_\bN \norm{\nabla \uu_h}^2_{\bL^\infty(\Omega)}.\label{eq: normalization bounds iii}
		\end{align}
	\end{subequations}
	Moreover, let $k\in\{0,1\}$, $1\le p \le \infty$, $c>0$, and $\uu,\widetilde{\uu}\in \bW^{k,p}(\Omega)$ such that $0 < c \le |\uu|,|\widetilde{\uu}|\le c^{-1}$ a.e. in $\Omega$. Then, it holds that
	\begin{equation}\label{eq: normalization bounds 2}
		\norm[\big]{\bN(\uu)-\bN(\widetilde{\uu})}_{\bW^{k,p}(\Omega)} \le C_\bN \norm{\uu-\widetilde{\uu}}_{\bW^{k,p}(\Omega)}.
	\end{equation}
	In either estimate \eqref{eq: normalization bounds}--\eqref{eq: normalization bounds 2}, the constant $C_\bN> 0$ depends only on $c$.
\end{lemma}
The following result provides stability results for the tangential projection $\bP(\cdot)$, generalizing \cite[Lemma~4.1]{feischl}. The proof is given in Appendix~\ref{section: appendix}.
\begin{lemma}\label{lem: projection}
Let $k\in\{0,1\}$, $p\in\{2,\infty\}$. Let $M,c>0$ and $\uu, \widetilde{\uu}\in \bW^{k,\infty}(\Omega)$ such that  
$$
0<c\le|\uu|,|\widetilde{\uu}|\le c^{-1} \,\text{ a.e. in }\Omega \quad \text{ and } \quad\norm{\uu}_{\bW^{k,\infty}(\Omega)},\norm{\widetilde{\uu}}_{\bW^{k,\infty}(\Omega)}\le M.
$$ 
Then, the tangential projection $\bP(\cdot)$ satisfies
	\begin{equation}\label{eq: stability projection}
		\norm{\bP(\uu)\vv}_{\bW^{k,p}(\Omega)}\le \newconst{stab_P} \norm{\vv}_{\bW^{k,p}(\Omega)}\norm{\uu}_{\bW^{k,\infty}(\Omega)}^2 \quad\text{for all} \quad \vv \in \bW^{k,\infty}(\Omega),
	\end{equation}
	where the constant $\const{stab_P}>0$ depends only on $c$.
Moreover, it holds that
	\begin{align}\label{eq: projection}
		\norm{(\bP(\uu)-\bP(\widetilde{\uu}))\vv}_{\bH^k(\Omega)} &\le \newconst{lipsch_P} \norm{\vv}_{\bW^{k,\infty}(\Omega)}\norm{\uu-\widetilde{\uu}}_{\bH^k(\Omega)},
	\end{align}
	and 
	\begin{align}\label{eq: projection L1}
		\norm{(\bP(\uu)-\bP(\widetilde{\uu}))\vv}_{\bL^1(\Omega)} &\le \newconst{lipsch_P} \normO{\vv}\normO{\uu-\widetilde{\uu}},
	\end{align}
where the constant $\const{lipsch_P}>0$ depends only on $M$ and $c$.
\end{lemma}

\subsection{Nodal interpolation}\label{section: nodal interpolation}
We use the vector-valued nodal interpolation operator
\begin{align}\label{eq: nodal interpolation}
 \III_h \colon C(\overline{\Omega}; \R^3) \to \boldsymbol{\mathcal{S}}^1(\mathcal{T}_h), \quad
 \III_h \vv \coloneqq \sum_{\zz \in \mathcal{N}_h} \vv(\zz) \zeta_{\zz},
\end{align}
where $\zeta_{\zz}\in C(\overline{\Omega})$ is the hat function satisfying that $\zeta_{\zz}|_K$ is affine for all $K \in \mathcal{T}_h$ and that $\zeta_{\zz}(\zz') = \delta_{\zz,\zz'}$ for all $\zz, \zz' \in \mathcal{N}_h$. Let $\mathrm{D}_h^2$ be the $\mathcal{T}_h$-elementwise defined Hessian operator. For all $\vv\in C(\overline{\Omega};\R^3)$ with $\vv|_K\in\bH^2(K)$ and all $k\in\{0,1\}$, \cite[Section~4.4]{brennerscott} proves 
\begin{equation}\label{eq: interpolation estimate}
	\norm{\vv-\III_h\vv}_{\bH^k(\Omega)} \le \newconst{nodal_interpol} h^{2-k}\norm{\mathrm{D}_h^2\vv}_{\bL^2(\Omega)}.
\end{equation}

The following lemma from \cite{bkw2024} shows stability of the nodal interpolation operator $\III_h$ on rational expressions. For completeness, the proof is given in Appendix~\ref{section: appendix}.
\begin{lemma}[Stability of $\III_h$ on rational expressions \texorpdfstring{\cite[Lemma~2.3]{bkw2024}}{[Lemma~2.3, bkw2024]}]\label{lem: stability I_h}
Let $r\in \N_0$. Let $q_h, r_h \in C(\overline{\Omega})$ be elementwise polynomial functions, i.e., $\left.q_h\right|_K, \left.r_h\right|_K \in \boldsymbol{\mathcal{P}}_r(K)$ for all $K \in \mathcal{T}_h$, and assume that $0<c\le\left|q_h\right|\le c^{-1}$. Then, for any $k \in \{0,1\}$ and $1 \leq p \leq \infty$, there holds
\begin{equation}\label{eq: stability I_h}
\norm[\Big]{\III_h\Big(\frac{r_h}{\left|q_h\right|}\Big)}_{\bW^{k, p}(\Omega)} \leq \newconst{stab_I}\,\norm[\Big]{\frac{r_h}{\left|q_h\right|}}_{\bW^{k, p}(\Omega)}.
\end{equation}
The constant $\const{stab_I} > 0$ depends only on $\gamma$, $r$, $k$, and $p$.
\end{lemma}

Finally, we will repeatedly use the following inverse estimate \cite[Section~4.5]{brennerscott}:
\begin{equation}\label{eq: inverse estimate L^p}
		\norm{\nabla \vv_h}_{\bL^p(\Omega)} \le \newconst{inverse_est} h^{-1}\norm{\vv_h}_{\bL^p(\Omega)} \quad \text{for all}\quad \vv_h\in \boldsymbol{\mathcal{S}}^1(\mathcal{T}_h).
\end{equation}
The constants in \eqref{eq: interpolation estimate}--\eqref{eq: inverse estimate L^p} depend only on $\gamma$.
\subsection{Discrete tangential projection}\label{section:discrete projection}
Recall $\III_h$ from \eqref{eq: nodal interpolation}. Define the discrete counterpart of $\mathcal{M}^+$, where the positivity is imposed at all mesh nodes, i.e.,
$$
\mathcal{M}_h^+\coloneqq \set{\uu\in C(\overline{\Omega},\R^3)}{|\uu(\zz)| > 0 \quad \text{for every}\quad \zz\in \mathcal{N}_h}.
$$
For all $\uu\in\mathcal{M}_h^+$, we consider the discrete tangential projection onto $\bTh(\uu)$
\begin{equation*}
	\begin{aligned}
	&\bPh(\uu) \colon  \boldsymbol{\mathcal{S}}^1(\mathcal{T}_h)  \to \bTh(\uu), \quad \bP_h(\uu)\vv_h \coloneqq \III_h[\bP(\uu)\vv_h] = \sum_{\zz\in\mathcal{N}_h} \left(\bI - \frac{\uu(\zz)\otimes\uu(\zz)}{|\uu(\zz)|^2} \right)\vv_h(\zz)\zeta_{\zz}.\end{aligned}
\end{equation*}

The following lemma proves that $\mhat_h^{j+1}\in \mathcal{M}_h^+$ for every $j=0,\dots,N-1$ and hence the well-posedness of the discrete projection $\bP_h(\mhat_h^{j+1})$.
\begin{lemma}\label{lem: well-posedness discrete projection}
	For every node $\zz\in\mathcal{N}_h$, Algorithm \ref{alg: full discr} guarantees
	\begin{enumerate}[label=\rm(\alph*)]
		\item $|\mm_h^{j+1}(\zz)|^2 = \frac{4}{3} |\mm_h^{j}(\zz)|^2 -\frac{1}{3} |\mm_h^{j-1}(\zz)|^2 + \tau^4 |\d_t^2\mm_h^{j+1}(\zz)|^2$ for all $j=2,\dots,N-~1$;
		\item $|\mm_h^{j+1}(\zz)|^2 \geq |\mm_h^{j}(\zz)|^2\geq 1$ for all $j=0,\dots,N-1$;
		\item\label{point c} $|\mhat_h^{j+1}(\zz)|^2\geq 1$ for all $j=0,\dots,N-1$.
	\end{enumerate}
	This implies that the discrete projection $\bPh(\mhat_h^{j+1})$ and the discrete tangent space $\bT_h(\mhat_h^{j+1})$ are well-defined for every $j=0,\dots,N-1$.
\end{lemma}
\begin{proof}
The statement of {\rm (a)} is found in \cite[Equation (40)]{part1}. It follows from {\rm (a)} that
\begin{align*}
	|\mm_h^{j+1}(\zz)|^2 &\stackrel{{\rm (a)}}{\geq} |\mm_h^{j}(\zz)|^2 + \frac{1}{3} \big(|\mm_h^{j}(\zz)|^2 -|\mm_h^{j-1}(\zz)|^2\big).
\end{align*}
Since $|\mm_h^1(\zz)|^2 = |\mm_h^0(\zz)|^2 + \tau^2 |\vv_h^0(\zz)|^2 \ge |\mm_h^0(\zz)|^2 = 1$, induction on $j$ thus proves~{\rm (b)}.
Moreover, elementary calculation yields
$$
2\mm_h^{j}(\zz)\cdot\mm_h^{j-1}(\zz) \leq |\mm_h^{j}(\zz)|^2 + |\mm_h^{j-1}(\zz)|^2,
$$
and hence
\begin{align*}
	|\mhat_h^{j+1}(\zz)|^2 &= \left|2\mm_h^{j}(\zz)-\mm_h^{j-1}(\zz)\right|^2 = 4|\mm_h^{j}(\zz)|^2 - 4\mm_h^{j}(\zz)\cdot\mm_h^{j-1}(\zz) + |\mm_h^{j-1}(\zz)|^2\\
	&\geq 4|\mm_h^{j}(\zz)|^2 - 2|\mm_h^{j}(\zz)|^2 - 2|\mm_h^{j-1}(\zz)|^2 + |\mm_h^{j-1}(\zz)|^2\\
	&= |\mm_h^{j}(\zz)|^2+ \left(|\mm_h^{j}(\zz)|^2 - |\mm_h^{j-1}(\zz)|^2\right)\stackrel{{\rm (b)}}{\geq} |\mm_h^{j}(\zz)|^2 \stackrel{{\rm (b)}}{\geq} 1.
\end{align*}
This concludes also the proof of \textit{{\rm (c)}}.
\end{proof}

The next lemma states that the discrete projection $\bPh(\cdot)$ approximates the orthogonal projection $\bP(\cdot)$ with optimal order.
\begin{lemma}\label{lem: approximation discrete projection}
Let $k\in\{0,1\}$, $\uu_h\in\boldsymbol{\mathcal{S}}^1(\mathcal{T}_h)$ and $c>0$ such that $0<c\le |\uu_h|\le c^{-1}$ a.e.~in~$\Omega$. 
Then, there holds
\begin{equation}\label{eq: approximation discrete projection}
	\norm{(\bP(\uu_h)-\bPh(\uu_h))\vv_h}_{\bH^k(\Omega)}\le \newconst{P_Ph_approx} h^{2-k}\norm{\vv_h}_{\bH^1(\Omega)}\norm{\uu_h}^2_{\bW^{1,\infty}(\Omega)}\quad\text{for all $\vv_h\in\boldsymbol{\mathcal{S}}^1(\mathcal{T}_h)$,}
\end{equation}
where the constant $\const{P_Ph_approx}>0$ depends only on $\gamma$ and $c$.
\end{lemma}
\begin{proof}
We prove \eqref{eq: approximation discrete projection} elementwise for all $K\in\mathcal{T}_h$ and $\vv_h\in \boldsymbol{\mathcal{S}}^1(\mathcal{T}_h)$. Thanks to the nodal interpolation estimate (see, e.g., \cite[Theorem 4.4.4, Theorem 4.4.20]{brennerscott}), we have
\begin{align*}
	\norm{(\bP(\uu_h)-\bPh(\uu_h))\vv_h}_{\bH^k(K)} &=\norm{\bP(\uu_h)\vv_h-\III_h(\bP(\uu_h)\vv_h)}_{\bH^k(K)}\\
	&\lesssim h^{2-k}\norm{\mathrm{D}_h^2(\bP(\uu_h)\vv_h)}_{\bL^2(K)},
\end{align*}
where the hidden constant depends only on $\gamma$.
Note that
$$
\bP(\uu_h)\vv_h=\vv_h-\dfrac{({\uu_h}\cdot{\vv_h})\uu_h}{|\uu_h|^2}=\vv_h - (\bN(\uu_h)\cdot\vv_h)\,\bN(\uu_h).
$$
Since $\vv_h$ is linear on $K$, it follows that
\begin{align*}
	\begin{split}
		\partial_i\partial_j&\big((\bN(\uu_h)\cdot \vv_h)\bN(\uu_h)\big)\\
		&\mkern-20mu=\partial_i\big[(\partial_j\bN(\uu_h)\cdot \vv_h)\bN(\uu_h) + (\bN(\uu_h)\cdot \partial_j\vv_h)\bN(\uu_h)+ (\bN(\uu_h)\cdot \vv_h)\partial_j\bN(\uu_h)\big]\\
		&\mkern-20mu= (\partial_i\partial_j\bN(\uu_h)\cdot \vv_h)\bN(\uu_h) + (\partial_j \bN(\uu_h) \cdot \partial_i\vv_h)\bN(\uu_h) + (\partial_j\bN(\uu_h)\cdot \vv_h)\partial_i\bN(\uu_h)\\
		 &+(\partial_i\bN(\uu_h)\cdot \partial_j\vv_h)\bN(\uu_h) + (\bN(\uu_h)\cdot \partial_j\vv_h)\partial_i\bN(\uu_h)\\
		&+ (\partial_i\bN(\uu_h)\cdot\vv_h)\partial_j\bN(\uu_h) + (\bN(\uu_h)\cdot \partial_i\vv_h)\partial_j\bN(\uu_h) + (\bN(\uu_h)\cdot \vv_h)\partial_i\partial_j\bN(\uu_h).
	\end{split}
\end{align*}
Noting that $\norm{\bN(\uu_h)}_{\bL^\infty(K)}=1$ and using Young's inequality, this implies that
\begin{align*}
	\norm{\mathrm{D}_h^2(\bP(\uu_h)\vv_h)}_{\bL^2(K)}&\lesssim \norm{\vv_h}_{\bH^1(K)}\big(\norm{\mathrm{D}_h^2\bN(\uu_h)}_{\bL^\infty(K)}\norm{\bN(\uu_h)}_{\bL^\infty(K)}\\
	&\mkern+117mu+\norm{\nabla\bN(\uu_h)}_{\bL^\infty(K)}\norm{\bN(\uu_h)}_{\bL^\infty(K)}+\norm{\nabla\bN(\uu_h)}^2_{\bL^\infty(K)}\big)\\
	&\lesssim \norm{\vv_h}_{\bH^1(K)}\big(\norm{\mathrm{D}_h^2\bN(\uu_h)}_{\bL^\infty(K)}+\norm{\bN(\uu_h)}^2_{\bL^\infty(K)}+\norm{\nabla\bN(\uu_h)}^2_{\bL^\infty(K)}\big)\\
	&\lesssim \norm{\vv_h}_{\bH^1(K)}\big(\norm{\mathrm{D}_h^2\bN(\uu_h)}_{\bL^\infty(K)}+\norm{\bN(\uu_h)}^2_{\bW^{1,\infty}(K)}\big).
\end{align*} Summing over all elements $K\in\mathcal{T}_h$ and applying \eqref{eq: normalization bounds}, we conclude the proof.
\end{proof}
Next, we provide stability results for $\bPh$, which generalize \cite[Lemma 3.4--3.5]{bkw2024}.
\begin{lemma}\label{lem: stability discrete projection}
	Let $k\in\{0,1\}, p\in\{2,\infty\}.$ Let $M>0, c>0$ and $\uu_h, \widetilde{\uu}_h\in \boldsymbol{\mathcal{S}}^1(\mathcal{T}_h)$ such that
	$$
	0<c\le |\uu_h|,|\widetilde{\uu}_h|\le c^{-1}\, \text{a.e. in }\, \Omega \qquad\text{and}\qquad\norm{\uu_h}_{\bW^{k,\infty}(\Omega)}, \norm{\widetilde{\uu}_h}_{\bW^{k,\infty}(\Omega)}\le M.
	$$	
	Then, the discrete projection $\bPh(\cdot)$ satisfies
	\begin{equation}\label{eq: stability discrete projection}
		\norm{\bPh(\uu_h)\vv_h}_{\bW^{k,p}(\Omega)}\le \newconst{stab_Ph} \norm{\vv_h}_{\bW^{k,p}(\Omega)}\norm{\uu_h}_{\bW^{k,\infty}(\Omega)}^2,\quad\text{for all}\quad \vv_h\in \boldsymbol{\mathcal{S}}^1(\mathcal{T}_h),
	\end{equation}
	where the constant $\const{stab_Ph}>0$ depends only on $\gamma$, $c$, $k$, and $p$.	Moreover, there holds
	\begin{equation}\label{eq: discrete projection lipschitz estimate}
		\norm{(\bPh(\uu_h)-\bPh(\widetilde{\uu}_h))\vv_h}_{\bH^k(\Omega)}\le \newconst{lipsch_Ph} \norm{\vv_h}_{\bW^{k,\infty}(\Omega)}\norm{\uu_h-\widetilde{\uu}_h}_{\bH^k(\Omega)}
	\end{equation}
	and
	\begin{equation}\label{eq: discrete projection lipschitz estimate L1}
		\norm{(\bPh(\uu_h)-\bPh(\widetilde{\uu}_h))\vv_h}_{\bL^1(\Omega)}\le \const{lipsch_Ph} \normO{\vv_h}\normO{\uu_h-\widetilde{\uu}_h}
	\end{equation}
where the constant $\const{lipsch_Ph}>0$ depends only on $\gamma$, $M$, $c$, and $k$.
\end{lemma}
\begin{proof}
	To prove \eqref{eq: stability discrete projection} it suffices to note that, by definition, it holds that 
	\begin{align*}
		\bPh(\uu_h)\vv_h = \III_h[\bP(\uu_h)\vv_h].
	\end{align*} The stability of the nodal interpolation operator from Lemma~\ref{lem: stability I_h} implies that
	\begin{align*}
		\norm{\bPh(\uu_h)\vv_h}_{\bW^{k,p}(\Omega)} \stackrel{\eqref{eq: stability I_h}}{\lesssim} \norm{\bP(\uu_h)\vv_h}_{\bW^{k,p}(\Omega)}\stackrel{\eqref{eq: stability projection}}{\lesssim}\norm{\vv_h}_{\bW^{k,p}(\Omega)}\norm{\uu_h}_{\bW^{k,\infty}(\Omega)}^2,
	\end{align*}
	where the hidden constant depends only on $\gamma$, $c$, $k$, and $p$. To prove \eqref{eq: discrete projection lipschitz estimate}, note that
	\begin{align*}
		\norm{(\bP_h(\uu_h) - \bP_h(\widetilde{\uu}_h))\vv_h}_{\bH^k(\Omega)} &= \norm{\III_h\big[(\bP(\uu_h)-\bP(\widetilde{\uu}_h))\vv_h\big]}_{\bH^k(\Omega)}\\
		&\mkern-90mu\stackrel{\eqref{eq: stability I_h}}{\lesssim} \norm{(\bP(\uu_h)-\bP(\widetilde{\uu}_h))\vv_h}_{\bH^k(\Omega)}
		\stackrel{\eqref{eq: projection}}{\lesssim}\norm{\vv_h}_{\bW^{k,\infty}(\Omega)}\norm{\uu_h - \widetilde{\uu}_h}_{\bH^k(\Omega)}.
	\end{align*}
	Similarly, it holds that
	\begin{align*}
		\norm{(\bP_h(\uu_h) - \bP_h(\widetilde{\uu}_h))\vv_h}_{\bL^1(\Omega)} &\stackrel{\eqref{eq: stability I_h}}{\lesssim} \norm{(\bP(\uu_h)-\bP(\widetilde{\uu}_h))\vv_h}_{\bL^1(\Omega)} \stackrel{\eqref{eq: projection L1}}{\lesssim} \normO{\vv_h}\normO{\uu_h - \widetilde{\uu}_h}.
	\end{align*}
	This concludes the proof.
	\end{proof}
\subsection{Ritz projection}\label{section: Ritz projection}
We recall the Ritz projection associated with the Poisson--Neumann problem and state some properties required for proving the main result. Specifically, we define $R_h\colon H^1(\Omega)\to \mathcal{S}^1(\mathcal{T}_h)$ by, for all $\varphi \in H^1(\Omega)$ and all $\psi_h\in \SS^1(\mathcal{T}_h)$,
\begin{equation}\label{eq: Ritz projection}
\scalarproductO{\nabla R_h \varphi}{\nabla\psi_h} + \scalarproductO{R_h\varphi}{1}\scalarproductO{\psi_h}{1} = \scalarproductO{\nabla\varphi}{\nabla\psi_h} + \scalarproductO{\varphi}{1}\scalarproductO{\psi_h}{1}.
\end{equation}
First, note that testing with $\psi\equiv1$ implies the mean-preserving property of the Ritz projection 
$$
\scalarproductO{R_h\varphi}{1}=\scalarproductO{\varphi}{1}.
$$
Second, considering test functions $\widetilde{\psi}\coloneqq \psi - |\Omega|^{-1}\scalarproductO{\psi}{1}$ satisfying $\scalarproductO{\widetilde{\psi}}{1}=0$ and $\nabla\widetilde{\psi}=\nabla\psi$, we see that \eqref{eq: Ritz projection} implies
\begin{equation}\label{eq: Ritz projection 2}
	\scalarproductO{\nabla R_h \varphi}{\nabla\psi_h} = \scalarproductO{\nabla\varphi}{\nabla\psi_h} \quad\text{ for all }\psi_h \in \mathcal{S}^1(\mathcal{T}_h).
\end{equation}
Third, if $\varphi\in H^2(\Omega)$ with $\nabla \varphi\cdot n = 0$ on $\partial\Omega$, then integration by parts yields
\begin{equation}\label{eq: Ritz projection 3}
	\scalarproductO{\nabla R_h \varphi}{\nabla\psi_h} = -\scalarproductO{\Delta \varphi}{\psi_h} \quad\text{ for all }\psi_h \in \mathcal{S}^1(\mathcal{T}_h).
\end{equation}
Finally, we collect some properties of the Ritz projection $R_h$, which are used later on; see~\cite[Theorem 8.1.11, Corollary 8.1.12]{brennerscott} or \cite[Section 2.1]{ciarlet} for more details. Specifically, we require that the standard approximation estimates hold:
\begin{align}
    \norm{v-R_hv}_{H^1(\Omega)} &\le C_{\rm R} h \norm{v}_{H^2(\Omega)} \quad \text{for all} \quad v\in H^2(\Omega),\label{eq: Ritz_error_1}\\
    \norm{v-R_hv}_{L^\infty(\Omega)} &\le C_{\rm R} h \norm{v}_{W^{2,\infty}(\Omega)} \quad \text{for all} \quad v\in W^{2,\infty}(\Omega).\label{eq: Ritz_error_2}
\end{align}
Moreover, we require the following stability estimate in the $W^{1,\infty}$-norm:
\begin{align}
    \norm{R_h v}_{W^{1,\infty}(\Omega)} &\le C_{\rm R} \norm{v}_{W^{1,\infty}(\Omega)} \quad \text{for all} \quad v\in W^{1,\infty}(\Omega).\label{eq: Ritz_stability}
\end{align}
\subsection{Notation}\label{section:notation}
Firstly, let $\mm_\star^n \coloneqq \mm(t_n)$ denote the exact solution of the LLG equation \eqref{eq: LLG system} at time $t_n$ and analogously 
\begin{subequations}\label{eq: exact solution quantities}
\begin{align}
	\widehat{\mm}_\star^n &\coloneqq 2\mm_\star^{n-1} - \mm_\star^{n-2},\label{eq: m_hat_star}\\
	\vv_\star^{n-1} &\coloneqq \bP(\widehat{\mm}_\star^{n})\frac{1}{\tau}\Big(\frac{3}{2}\mm_\star^{n} - 2\mm_\star^{n-1} + \frac{1}{2}\mm_\star^{n-2}\Big) \in \bT(\widehat{\mm}_\star^n).\label{eq: v_star}
\end{align}
\end{subequations}
Note that
$$
|\mhat_\star^n|=|2\mm_\star^{n-1}-\mm_\star^{n-2}|\begin{cases}
\ge 2 |\mm_\star^{n-1}| - |\mm_\star^{n-2}| = 1\\
\le 2 |\mm_\star^{n-1}| + |\mm_\star^{n-2}| = 3
  \end{cases}
$$ 
Therefore, the projection $\bP(\mhat_\star^n)$ and hence also $\vv_\star^{n-1}$ are well-defined.

With the vector-valued Ritz projection $\bR_h \coloneqq \mathbf{I}\otimes R_h\colon \bH^1(\Omega)\to \boldsymbol{\mathcal{S}}^1(\mathcal{T}_h)$, we will also consider the following quantities, related to the exact solution:
\begin{subequations}\label{eq: Ritz projections}
\begin{align}
\mm_{\star,h}^n 
&\coloneqq \bR_h\mm_\star^n,
\label{eq: m_star_ritz}
\\
\widehat{\mm}_{\star,h}^n 
&\coloneqq 2\mm_{\star,h}^{n-1} - \mm_{\star,h}^{n-2},
\label{eq: m_hat_star_ritz}
\\
\vv_{\star,h}^{n-1} 
&\coloneqq \bPh(\widehat{\mm}_{\star,h}^{n})\frac{1}{\tau}
\Big(\frac{3}{2}\mm_{\star,h}^{n}
- 2\mm_{\star,h}^{n-1}
+ \frac{1}{2}\mm_{\star,h}^{n-2}\Big)
\in \bTh(\widehat{\mm}_{\star,h}^n).
\label{eq: v_star_ritz}
\end{align}
\end{subequations}
The following lemma shows uniform bounds for the quantity $|\mstarhhat^n(\xx)|$ for all $\xx\in\Omega$ and sufficiently small $h$, which implies that $\bP_h(\mstarhhat^n)$ and hence $\vv_{\star,h}^{n-1}$ are well defined.
\begin{lemma}\label{lem: well-posedness mhat star h}
	There exists $h_{*}>0$ such that, for all $h\le h_{*},$
	\begin{equation*}\label{eq: bounds mhat star h}
		\frac{1}{4} \le |\mhat_{\star,h}^n(\xx)| \le 4 \quad \text{for all}\quad n=2,\ldots,N \,\text{ and  all }\, \xx\in \Omega.
	\end{equation*}
\end{lemma}
\begin{proof} Firstly, observe that for all $n=0,\ldots,N$, the Ritz projection $\mm_{\star,h}^n$ satisfies
$$
\norm{|\mm_{\star,h}^n| -1}_{\bL^\infty(\Omega)} \le \norm{\mm_{\star,h}^n - \mm_\star^n}_{\bL^\infty(\Omega)} \stackrel{\eqref{eq: Ritz_error_2}}{\le} C_{\rm R} h \norm{\mm_\star^n}_{\bW^{2,\infty}(\Omega)}\le C_{\rm R} h \norm{\mm}_{L^\infty(\bW^{2,\infty})}
$$
and thus
$$
1 - C_{\rm R} h \norm{\mm}_{L^\infty(\bW^{2,\infty})} \le |\mm_{\star,h}^n(\xx)| \le 1 + C_{\rm R} h \norm{\mm}_{L^\infty(\bW^{2,\infty})}\quad \text{for all $\xx\in \Omega$.}
$$
This allows to conclude that 
\begin{equation*}\label{eq: upperbound mhat}
|\mhat_{\star,h}^n(\xx)| \le 2 |\mm_{\star,h}^{n-1}(\xx)| + |\mm_{\star,h}^{n-2}(\xx)| \le 3 + 3\, C_{\rm R} h \norm{\mm}_{L^\infty(\bW^{2,\infty})}
\end{equation*}
and that
\begin{align*}
|\mhat_{\star,h}^n(\xx)| &\ge 2 |\mm_{\star,h}^{n-1}(\xx)| - |\mm_{\star,h}^{n-2}(\xx)| \ge  1 - 3\, C_{\rm R} h \norm{\mm}_{L^\infty(\bW^{2,\infty})}.
\end{align*}
Hence, there exist $h_{*}>0$ such that 
$C_{\rm R} h \norm{\mm}_{L^\infty(\bW^{2,\infty})} \le 1/4$ whenever $h\le h_{*}$. Therefore, 
\begin{equation*}\label{eq: small h}
1/4 \le |\mhat_{\star,h}^n(\xx)| \le 4  \quad \text{for all}\quad 0<h\le h_{*} \quad \text{and all}\quad\xx\in \Omega,
\end{equation*}
which guarantees the well-posedness of $\bP(\mhat_{\star,h}^n)$ and $\bP_h(\mhat_{\star,h}^n)$.
\end{proof}
The following lemma provides bounds that will be used in the subsequent proofs.
\begin{lemma}\label{lem: various v estimates}
	Suppose that the solution $\mm$ of LLG \eqref{eq: LLG system} satisfies the regularity assumptions~\eqref{eq: regularity}.
	Recall $h_*$ from Lemma~\ref{lem: well-posedness mhat star h} and suppose that $h\le h_*$. Let $M>0$ such that 
	\begin{equation}\label{eq: bounds lemma v}
	\max\big\{\norm{\mm}_{L^\infty(\bW^{1,\infty})},\norm{\partial_{t}\mm}_{L^\infty(\bW^{1,\infty})}, \norm{\mm}_{L^\infty(\bW^{2,\infty})}, \norm{\partial_t\mm}_{L^\infty(\bH^2)}\big\}\le M.
	\end{equation}
	Let $X\in\{\bH^1(\Omega),\bW^{1,\infty}(\Omega)\}$ and $Y\in\{\bH^2(\Omega),\bW^{1,\infty}(\Omega)\}$. 
	Then, there holds
	\begin{equation}\label{eq: bounds C_X C_Y}
		\norm{\vstarh^{n-1}}_X\le \newconst{C_X} \quad \text{and} \quad \norm{\vv_\star^{n-1}}_Y\le \newconst{C_Y}\quad \text{for all}\quad n=2,\ldots,N,
	\end{equation}
	where the constant $\const{C_X}>0$ depends only on $\gamma$, $|\Omega|$, $M$, and $h\le h_*$, whereas $\const{C_Y}>0$ depends only on $M$.
\end{lemma}
\begin{proof}
	Firstly, with $I_n=[t_{n-1},t_n]$, notice that 
	\begin{align*}
	\norm{\d _t\mm_\star^n}_{\bW^{1,\infty}(\Omega)}=\frac{1}{\tau}\norm[\Big]{\int_{t_{n-1}}^{t_n}\partial_t\mm(s)\d s}_{\bW^{1,\infty}(\Omega)}\le \norm{\partial_t\mm}_{L^\infty(I_n,\bW^{1,\infty}(\Omega))}.
	\end{align*}
	Recall that $h\le h_*$ allows to apply Lemma~\ref{lem: stability discrete projection} and recall definitions \eqref{eq: v_star} and \eqref{eq: v_star_ritz}. Together with $\bW^{1,\infty}$-stability of the Ritz projection~\eqref{eq: Ritz_stability}, the properties of the nodal interpolation $\III_h$, and \eqref{eq: bounds lemma v}, Lemma~\ref{lem: stability discrete projection} proves for $X\in\{\bH^1(\Omega),\bW^{1,\infty}(\Omega)\}$ that
\begin{align}
	\mkern-22mu\norm{\vstarh^{n-1}}_{X}&\lesssim\norm{\vstarh^{n-1}}_{\bW^{1,\infty}(\Omega)}\stackrel{\eqref{eq: stability discrete projection}}{\lesssim} \norm[\Big]{\frac{1}{\tau}\Big(\frac{3}{2}\mm_{\star,h}^{n} - 2\mm_{\star,h}^{n-1} + \frac{1}{2}\mm_{\star,h}^{n-2}\Big)}_{\bW^{1,\infty}(\Omega)}\mkern-4mu\norm{\mstarhhat^n}_{\bW^{1,\infty}(\Omega)}^2\nonumber\\
	&\mkern-60mu\stackrel{\eqref{eq: Ritz_stability}}{\lesssim} \norm[\Big]{\frac{1}{\tau}\Big(\frac{3}{2}\mm_{\star}^{n} - 2\mm_{\star}^{n-1} + \frac{1}{2}\mm_{\star}^{n-2}\Big)}_{\bW^{1,\infty}(\Omega)}\norm{\mhat_\star^n}_{\bW^{1,\infty}(\Omega)}^2 \label{eq: bound v star h}\\
	&\mkern-55mu\le \norm{\mhat_\star^n}_{\bW^{1,\infty}(\Omega)}^2\Big(\frac{3}{2}\norm{\d_t\mm_\star^n}_{\bW^{1,\infty}(\Omega)})+\frac{1}{2}\norm{\d_t\mm_\star^{n-1}}_{\bW^{1,\infty}(\Omega)}\Big)\stackrel{\eqref{eq: bounds lemma v}}{\lesssim} \norm{\partial_t\mm}_{L^\infty(\bW^{1,\infty})}\stackrel{\eqref{eq: bounds lemma v}}{\lesssim} 1.\nonumber
\end{align}
	Overall, the hidden constants depend only on $\gamma$, $|\Omega|$, $M$, and on $h \le h_*$. To bound $\norm{\vv_\star^{n-1}}_Y$, notice that the case $Y\equiv\bW^{1,\infty}(\Omega)$ can be handled similarly as 
	\begin{equation}
		\begin{split}
	\norm{\vv_\star^{n-1}}_{\bW^{1,\infty}(\Omega)}&\stackrel{\eqref{eq: stability projection}}{\le} \norm[\Big]{\frac{1}{\tau}\big(\frac{3}{2}\mm_\star^n-2\mm_\star^{n-1}+\frac{1}{2}\mm_\star^{n-2}\big)}_{\bW^{1,\infty}(\Omega)}\norm{\mhat_\star^{n}}_{\bW^{1,\infty}(\Omega)}^2\stackrel{\eqref{eq: bound v star h}}{\lesssim} 1.
		\end{split}
	\end{equation}
	If $Y\equiv\bH^2(\Omega)$, then Lemma~\ref{lem: stability discrete projection} does not apply. However, it holds that
	\begin{equation}\label{eq: bound v star H2}
		\begin{aligned}
			\mkern-20mu\norm{\vv_\star^{n-1}}_{\bH^2(\Omega)}&\mkern-7mu\stackrel{\eqref{eq: v_star}}{=} \norm[\Big]{\bP(\mhat_\star^n)\frac{1}{\tau}\big(\frac{3}{2}\mm_\star^n-2\mm_\star^{n-1}+\frac{1}{2}\mm_\star^{n-2}\big)}_{\bH^2(\Omega)}\\
					&= \norm[\Big]{\Big(\mathbf{I} - \bN(\mhat_\star^n)\bN(\mhat_\star^n)^T\Big)\frac{1}{\tau}\big(\frac{3}{2}\mm_\star^n-2\mm_\star^{n-1}+\frac{1}{2}\mm_\star^{n-2}\big)}_{\bH^2(\Omega)}\\
					&\lesssim \norm[\big]{\mathbf{I} - \bN(\mhat_\star^n)\bN(\mhat_\star^n)^T}_{\bW^{2,\infty}(\Omega)}\norm[\Big]{\frac{3}{2}\d_t\mm_\star^n-\frac{1}{2}\d_t\mm_\star^{n-1}}_{\bH^2(\Omega)}\\
					&\lesssim\big(1+\norm{\bN(\mhat_\star^n)}^2_{\bW^{2,\infty}(\Omega)}\big) \big(\norm{\d_t\mm_\star^n}_{\bH^2(\Omega)}+\norm{\d_t\mm_\star^{n-1}}_{\bH^2(\Omega)}\big)\\
					&\mkern-19mu\stackrel{\eqref{eq: regularity},\eqref{eq: normalization bounds}}{\lesssim} \norm{\mhat_\star^n}^2_{\bW^{2,\infty}(\Omega)}(\norm{\d_t\mm_\star^n}_{\bH^2(\Omega)}+\norm{\d_t\mm_\star^{n-1}}_{\bH^2(\Omega)})\stackrel{\eqref{eq: regularity}}{\le}\norm{\partial_t\mm}_{L^\infty(\bH^2)}\mkern-5mu\stackrel{\eqref{eq: bounds lemma v}}{\lesssim}\mkern-5mu 1
		\end{aligned}
	\end{equation}
	where the hidden constants depend only on $M$. This concludes the proof.
\end{proof}

In order to derive consistency and stability estimates for the full discretization \eqref{eq: full discr alg B}, it is convenient to equivalently rewrite \eqref{eq: full discr alg B} for all $n = 2,\dots,N$ and all $\pphi_h\in\bTh(\mhat_h^n)$ as
\begin{equation}\label{eq: discrete formulation}
	\alpha \scalarproductO{\vv_h^{n-1}}{\pphi_h} + \scalarproductO{\mhat_h^n\times\vv_h^{n-1}}{\pphi_h} + \lamex^2 \scalarproductO{\nabla \mm_h^n}{\nabla \pphi_h} = \scalarproductO{\ff_h^n}{\pphi_h}
\end{equation}
with 
\begin{equation}\label{eq: discrete derivative v_h}
	\mm_h^n = \frac{4}{3}\mm_h^{n-1} - \frac{1}{3}\mm_h^{n-2} + \frac{2}{3}\tau \vv_h^{n-1}\quad \text{and}\quad
	\vv_h^{n-1} = \frac{1}{\tau}\Big(\frac{3}{2}\mm_h^n - 2 \mm_h^{n-1} + \frac{1}{2}\mm_h^{n-2}\Big).
\end{equation}

Inserting \eqref{eq: Ritz projections} in the linearly implicit BDF2 method \eqref{eq: discrete formulation}, there exists a unique defect $\boldsymbol{d}_h^n\in \bTh(\mstarhhat^n)$ such that, for all $\pphi_h\in \bTh(\mstarhhat^n)$ and all $n=2,\dots,N$,
\begin{equation}\label{eq: defect}
	\scalarproductO{\boldsymbol{d}_h^n}{\pphi_h} \coloneqq\alpha \scalarproductO{\vv_{\star,h}^{n-1}}{\pphi_h} + \scalarproductO{\mstarhhat^n\times\vstarh^{n-1}}{\pphi_h} + \lamex^2\scalarproductO{\nabla\mstarh^{n}}{\nabla\pphi_h} - \scalarproductO{\ff_h^n}{\pphi_h}.
\end{equation}
In addition, there exists a unique residual $\rr_h^n\in \bT_h(\mhat_h^n)$ such that, for all $\pphi_h \in \bTh(\widehat{\mm}_h^n)$ and all $n=2,\ldots,N$,
\begin{equation}\label{eq: residual equation}
	\scalarproductO{\boldsymbol{r}_h^n}{\pphi_h} \coloneqq\alpha \scalarproductO{\vv_{\star,h}^{n-1}}{\pphi_h} + \scalarproductO{\mstarhhat^n\times\vstarh^{n-1}}{\pphi_h} + \lamex^2\scalarproductO{\nabla\mstarh^{n}}{\nabla\pphi_h} - \scalarproductO{\ff_h^n}{\pphi_h}.
\end{equation}
Note the important difference between $\dd_h^n$ and $\rr_h^n$: \eqref{eq: defect} relies on $\bT_h(\mstarhhat^n)$, while \eqref{eq: residual equation} relies on $\bT_h(\mhat_h^n)$, i.e., the ansatz and the test spaces in \eqref{eq: defect}--\eqref{eq: residual equation} differ. 

The residual $\rr_h^n$ quantifies how much the Ritz projections of the exact solution fail to satisfy the numerical scheme, with respect to the orthogonality constraint imposed by the numerical solution. In this respect, it will also be convenient to define 
\begin{equation}\label{eq: D^n}
\bD_\star^n \coloneqq\alpha \partial_t \mm_\star^n + \mm_\star^n \times \partial_t \mm_\star^n - \lamex^2 \Delta \mm_\star^n - \ff_\star^n \quad \text{with} \quad \ff_\star^n \coloneqq \ff(t_n).
\end{equation} 
and its discrete counterpart
\begin{equation}\label{eq: D_h^n}
\mkern-160mu\boldsymbol{D}_h^n\coloneqq\alpha\vstarh^{n-1} + \mstarhhat^n\times \vstarh^{n-1}-\lamex^2\Delta\mm_\star^n-\ff_h^n.
\end{equation} 
With the notation \eqref{eq: exact solution quantities}--\eqref{eq: Ritz projections}, the relation~\eqref{eq: Ritz projection 2} becomes
$$
\scalarproductO{\nabla\mstarh^n}{\nabla\pphi_h} \stackrel{\eqref{eq: Ritz projection 2}}{=} \scalarproductO{\nabla\mm(t_n)}{\nabla\pphi_h} \stackrel{\eqref{eq: Ritz projection 3}}{=} -\scalarproductO{\Delta\mm(t_n)}{\pphi_h}.
$$
This allows to rewrite \eqref{eq: defect} and \eqref{eq: residual equation} as
\begin{subequations}\label{eq: defect and residual}
	\begin{align}
		\scalarproductO{\boldsymbol{d}_h^n}{\pphi_h} &= \scalarproductO{\boldsymbol{D}_h^n}{\pphi_h}\quad \text{for all}\quad \pphi_h\in \bTh(\mstarhhat^n), \label{eq: defect 2}\\
		\intertext{and}
		\scalarproductO{\rr_h^n}{\pphi_h} &= \scalarproductO{\boldsymbol{D}_h^n}{\pphi_h} \quad \text{for all}\quad \pphi_h\in \bTh(\mhat_h^n).\label{eq: residual 2}
	\end{align}
\end{subequations}
In order to analyze the error equation (see \eqref{eq: error equation} below), it is convenient to introduce also the following notation for the errors:
\begin{subequations}\label{eq: errors}
\begin{align}
	\boldsymbol{e}_h^n &\coloneqq\mm_h^n - \mstarh^n,  \label{eq: e_h^n}\\
	\widehat{\ee}_h^n &\coloneqq\mhat_h^n - \mstarhhat^n,  \label{eq: hat_e_h^n}\\
	\oomega_h^{n-1} &\coloneqq\vv_h^{n-1} - \vv_{\star,h}^{n-1}\stackrel{\eqref{eq: discrete derivative v_h}}{=}\frac{1}{\tau}\Big(\frac{3}{2}\ee_h^n-2\ee_h^{n-1}+\frac{1}{2}\ee_h^{n-2}\Big)+\sss_h^n,  \label{eq: omega_h^n-1}
\intertext{where} 
	\sss_h^n &\coloneqq(\mathbf{I}-\bPh(\mstarhhat^n))\frac{1}{\tau}\Big(\frac{3}{2}\mm_{\star,h}^{n} - 2\mm_{\star,h}^{n-1} + \frac{1}{2}\mm_{\star,h}^{n-2}\Big).\label{eq: def s_h^n}
\end{align}
\end{subequations}
\section{Consistency estimates}\label{sec: consistency}
\subsection{Consistency estimates of the semi-discretization in time.}
In this section, we derive consistency error estimates for Algorithm~\ref{alg: full discr}. 
The first preliminary result adapts \cite[Lemma 6.1]{feischl} to the present BDF2 algorithm. The proof is postponed to Appendix~\ref{section: appendix}.
\begin{lemma}\label{lem: semi-discrete approximation}
	Suppose that the solution $\mm$ of LLG \eqref{eq: LLG system} satisfies the regularity assumptions~\eqref{eq: regularity}. Recall $h_*$ from Lemma~\ref{lem: well-posedness mhat star h} and suppose that $h\le h_*$. Let $M>0$ such that 
	$$
		\max\big\{\norm{\partial_t\mm}_{L^\infty(\bL^\infty)},\norm{\partial^2_{t}\mm}_{L^\infty(\bL^2)}, \norm{\partial_t^{3}\mm}_{L^\infty(\bL^2)}\big\}\le M.
	$$
	Then, there exists $\newconst{semi-consistency_1}>0$ depending only on $M$ such that, for all $n=2,\ldots,N$,
	\begin{align}
		\normO{\widehat{\mm}_\star^n-\mm_\star^n} &\le \const{semi-consistency_1}\tau^2,\label{eq: consistency semi predictor}\\
		\normO{\vv_\star^{n-1}-\partial_t\mm_\star^n} &\le \const{semi-consistency_1}\tau^2.\label{eq: consistency semi derivative}
	\end{align}
\end{lemma}
\subsection{Residual estimate.}
The following lemma provides a bound on the residual $\boldsymbol{r}_h^n$ defined in \eqref{eq: residual equation}. Note that its statement makes the additional assumption that $1/2 \le |\widehat{\mm}_h^n| \le 2$. Indeed, this assumption will be proved in Proposition~\ref{prop: stability full discretization} below, provided that $h$ and $\tau$ are sufficiently small.
\begin{lemma}\label{lem: bounds r}
Let $M>0$. Suppose that the solution of the LLG equation \eqref{eq: LLG system} satisfies the regularity assumption \eqref{eq: regularity} and that $1/2\le |\mhat_h^n|\le 2$ a.e. in $\Omega$. Suppose that $\ff$ satisfies \eqref{eq: regularity f} and that $\max\limits_{n=2,\dots,N}\norm{\ff_\star^n-\ff_h^n}_{\bL^\infty(\Omega)}\le C_{\ff} h$, as well as
	\begin{equation}\label{eq: hp regularity bounds residual}
		\norm{\mm}_{L^\infty(\bW^{1,\infty})}+\norm{\partial_t\mm}_{L^\infty(\bW^{1,\infty})}+\norm{\lamex^2\Delta\mm+\ff}_{L^\infty(\bL^\infty)}\le M.
	\end{equation}
Recall $h_*>0$ from Lemma~\ref{lem: well-posedness mhat star h} and suppose $h\le h_*$.
Then, it holds that
\begin{equation}\label{eq: bound r_h^n}
	\normO{\boldsymbol{r}_h^n}\le \newconst{residual}\big(\normO{\dd_h^n}+\normO{\widehat{\ee}_h^n}\big)\quad\text{ for every }\quad n=2,\dots,N,
\end{equation}
where the constant $\const{residual}>0$ depends only on $\gamma$, $M$, $C_{\ff}$, and on $h\le h_*$.
\end{lemma}
\begin{proof}
	For every $\pphi_h\in \bTh(\mhat_h^n)$, it holds $\bPh(\mstarhhat^n)\pphi_h\in\bTh(\mstarhhat^n)$ and $\bPh(\mhat_h^n)\pphi_h=\pphi_h$. Therefore, we can rewrite \eqref{eq: residual equation} as
	\begin{align*}
		\scalarproductO{\boldsymbol{r}_h^n}{\pphi_h} &\mkern-3mu\stackrel{\eqref{eq: residual 2}}{=} \scalarproductO{\bD_h^n}{\pphi_h}= \scalarproductO{\bD_h^n}{\bPh(\mhat_h^n)\pphi_h}\\
		&\mkern+5mu=\scalarproductO{\bD_h^n}{\bPh(\mstarhhat^n)\pphi_h}+\scalarproductO{\bD_h^n}{\big(\bPh(\mhat_h^n)-\bPh(\mstarhhat^n)\big)\pphi_h}\\
		&\mkern-3mu\stackrel{\eqref{eq: defect 2}}{=} \scalarproductO{\dd_h^n}{\bPh(\mstarhhat^n)\pphi_h}+\scalarproductO{\bD_h^n}{\big(\bPh(\mhat_h^n)-\bPh(\mstarhhat^n)\big)\pphi_h}.
	\end{align*}
	We choose $\pphi_h=\boldsymbol{r}_h^n$. Thanks to Lemma \ref{lem: stability discrete projection} and the Hölder inequality on the second term of the right-hand side above, we deduce that
	\begin{equation}\label{eq: residual error equation}
		\begin{aligned}
			\normO{\boldsymbol{r}_h^n}^2 &\le \normO{\boldsymbol{d}_h^n}\normO{\bPh(\mstarhhat^n)\rr_h^n} + \norm{\bD_h^n}_{\bL^\infty(\Omega)}\norm[\big]{\big(\bPh(\mhat_h^n)-\bPh(\mstarhhat^n)\big)\rr_h^n}_{\bL^1(\Omega)}\\
			&\mkern-19mu\stackrel{\eqref{eq: stability discrete projection},\eqref{eq: discrete projection lipschitz estimate L1}}{\lesssim} \normO{\boldsymbol{d}_h^n}\norm{\mstarhhat^n}^2_{\bL^{\infty}(\Omega)}\normO{\rr_h^n}+\norm{\bD_h^n}_{\bL^\infty(\Omega)}\normO{\mhat_h^n-\mstarhhat^n}\normO{\rr_h^n}\\
			&\mkern-8mu\stackrel{\eqref{eq: hat_e_h^n}}{=}\big(\normO{\boldsymbol{d}_h^n}\norm{\mstarhhat^n}^2_{\bL^{\infty}(\Omega)}+\norm{\bD_h^n}_{\bL^\infty(\Omega)}\normO{\widehat{\ee}_h^n}\big)\normO{\rr_h^n},
		\end{aligned}
	\end{equation}
	where the hidden constant depends only on $\gamma$, $M$, $|\Omega|$, and on $h\le h_*$. Moreover, by definition~\eqref{eq: D_h^n}, we have that
	$$
		\norm{\bD_h^n}_{\bL^\infty(\Omega)}\le \alpha\norm{\vv_{\star,h}^{n-1}}_{\bL^\infty(\Omega)}+\norm{\mstarhhat^n\times\vv_{\star,h}^{n-1}}_{\bL^\infty(\Omega)}+\norm{\lamex^2\Delta\mm_\star^n+\ff_\star^n}_{\bL^\infty(\Omega)} + \norm{\ff_\star^n-\ff_h^n}_{\bL^\infty(\Omega)}.
	$$
	Lemma~\ref{lem: various v estimates} guarantees that $\norm{\vstarh^{n-1}}_{\bL^\infty(\Omega)}\lesssim 1$ and moreover, note that
	$$
	\norm{\mstarhhat^n}_{\bL^\infty(\Omega)}\stackrel{\eqref{eq: Ritz_stability}}{\lesssim}\norm{\mhat_\star^n}_{\bW^{1,\infty}(\Omega)}\stackrel{\eqref{eq: hp regularity bounds residual}}{\lesssim}1,
	$$ 
	where the hidden constant depends only on $\gamma$, $M$, and $h\le h_*$. Thanks to \eqref{eq: hp regularity bounds residual}, it holds that
	$$
	\norm{\mstarhhat^n\times\vv_{\star,h}^{n-1}}_{\bL^\infty(\Omega)}\le \norm{\mstarhhat^n}_{\bL^\infty(\Omega)}\norm{\vstarh^{n-1}}_{\bL^\infty(\Omega)}\stackrel{\eqref{eq: Ritz_stability}}{\lesssim} \norm{\mhat_\star^n}_{\bW^{1,\infty}(\Omega)}\norm{\vstarh^{n-1}}_{\bL^\infty(\Omega)}\stackrel{\eqref{eq: hp regularity bounds residual},\eqref{eq: bounds C_X C_Y}}{\lesssim} 1,
	$$ where the hidden constant depends only on $\gamma$, $M$, and $h\le h_*$. Together with \eqref{eq: hp regularity bounds residual}, this gives $\norm{\bD_h^n}_{\bL^\infty(\Omega)}\lesssim 1$. Combining all derived bounds with \eqref{eq: residual error equation}, we conclude the proof of \eqref{eq: bound r_h^n}.
\end{proof}
\begin{remark}\label{rem: well-defined}
	Since the quantities $|\mhat_\star^n|$ and $|\mstarhhat^n|$ are uniformly bounded away from zero for all $\xx \in \Omega$, the projections $\bP(\mhat_\star^n)$, $\bP_h(\mhat_\star^n)$, $\bP(\mstarhhat^n)$, and $\bP_h(\mstarhhat^n)$ are all well-defined and Lemmas~\ref{lem: normalization bounds}--\ref{lem: stability I_h} and Lemmas \ref{lem: approximation discrete projection}--\ref{lem: stability discrete projection} apply. 
	On the other hand, Lemma~\ref{lem: well-posedness discrete projection}\ref{point c} guarantees that $|\mhat_h^n(\zz)| \geq 1$ only at the nodes $\zz \in \mathcal{N}_h$ of the triangulation. This is sufficient to ensure that both $\bP_h(\mhat_h^n)$ and $\bP(\mhat_h^n)$ are well-defined: $\bP_h(\mhat_h^n)$ since it depends only on the nodal values, and $\bP(\mhat_h^n)$ since the set where $|\mhat_h^n|$ can be zero has measure zero. However, the estimates from Lemmas~\ref{lem: normalization bounds}--\ref{lem: stability I_h} and Lemmas~\ref{lem: approximation discrete projection}--\ref{lem: stability discrete projection} require that $|\mhat_h^n|$ is uniformly bounded away from zero almost everywhere in $\Omega$, which cannot be guaranteed by Lemma~\ref{lem: well-posedness discrete projection}\ref{point c} alone. To overcome this limitation, we employ an induction argument in the proof of Proposition~\ref{prop: stability full discretization} to establish that also $|\mhat_h^n|$ is indeed uniformly bounded away from zero, provided that $h$ and $\tau$ are sufficiently small.
\end{remark}
\subsection{Error equation.}
To obtain the error equation satisfied by the error $\ee_h^n$ from~\eqref{eq: e_h^n}, we subtract the residual equation \eqref{eq: residual equation} from the fully discrete problem \eqref{eq: discrete formulation}. According to the definitions~\eqref{eq: e_h^n}--\eqref{eq: omega_h^n-1} this reads for every $\pphi_h \in \bTh(\mhat_h^n)$ and $n=2,\ldots,N$ as:
\begin{equation}\label{eq: error equation}
	\alpha \scalarproductO{\oomega_h^{n-1}}{\pphi_h} + \scalarproductO{\widehat{\ee}_h^n\times\vv_{\star,h}^{n-1}}{\pphi_h} + \scalarproductO{\widehat{\mm}_h^n\times \oomega_h^{n-1}}{\pphi_h} + \lamex^2\scalarproductO{\nabla\ee_h^n}{\nabla\pphi_h} = -\scalarproductO{\rr_h^n}{\pphi_h}.
\end{equation}
The following lemma provides a bound with optimal order for $\sss_h^n$ defined in \eqref{eq: def s_h^n}.
\begin{proposition}\label{prop: bound s_h^n}
	Let $\mm$ satisfy the regularity assumptions \eqref{eq: regularity} and suppose that $1/2 \le |\mhat_h^n| \le 2$ a.e. in $\Omega$. Let $M>0$ such that 
	\begin{equation}\label{eq: bounds prop s}
		\max\{\norm{\mm}_{L^\infty(\bW^{1,\infty})},\norm{\partial_t\mm}_{L^\infty(\bH^2)},\norm{\partial_t^3\mm}_{L^\infty(\bH^1)}\} \le M.
	\end{equation}
	Recall $h_*$ from Lemma~\ref{lem: well-posedness mhat star h} and suppose that $h\le h_*$.
	Then, it holds that
	\begin{equation}\label{eq: bound s_h^n}
		\norm{\sss_h^n}_{\bH^1(\Omega)}\le \newconst{s}(h+\tau^2)\quad\text{ for every }\quad n=2,\dots,N,
	\end{equation}
	where the constant $\const{s}>0$ depends only on $\gamma$, $|\Omega|$, $M$, and $h\le h_*$.
\end{proposition}
\begin{proof}
	We begin by subtracting $(\mathbf{I}-\bP(\mm_\star^n))\partial_t\mm_\star^n=0$ from $\sss_h^n$ and have
\begin{equation}\label{eq: s_h^n}
	\begin{aligned}
		\sss_h^n & \stackrel{\eqref{eq: def s_h^n}}{=} (\mathbf{I}-\bPh(\mstarhhat^n))\frac{1}{\tau}\Big(\frac{3}{2}\mm_{\star,h}^{n} - 2\mm_{\star,h}^{n-1} + \frac{1}{2}\mm_{\star,h}^{n-2}\Big) - (\mathbf{I}-\bP(\mm_\star^n))\partial_t\mm_\star^n\\
		& \mkern+8mu= \Big[\frac{1}{\tau}\Big(\frac{3}{2}\mm_{\star,h}^{n} - 2\mm_{\star,h}^{n-1} + \frac{1}{2}\mm_{\star,h}^{n-2}\Big) - \partial_t\mm_\star^n\Big]\\
		&\mkern+50mu- \Big[\bPh(\mstarhhat^n)\frac{1}{\tau}\Big(\frac{3}{2}\mm_{\star,h}^{n} - 2\mm_{\star,h}^{n-1} + \frac{1}{2}\mm_{\star,h}^{n-2}\Big)-\bP(\mm_\star^n)\partial_t\mm_\star^n\Big].
	\end{aligned}
\end{equation}
In the following, we bound the two brackets of the right-hand side of \eqref{eq: s_h^n} separately.

\textbf{Step~1.}
The first term on the right-hand side of \eqref{eq: s_h^n} is bounded via
\begin{align*}
	A_1 &\coloneqq \norm[\Big]{\frac{1}{\tau}\Big(\frac{3}{2}\mm_{\star,h}^{n} - 2\mm_{\star,h}^{n-1} + \frac{1}{2}\mm_{\star,h}^{n-2}\Big) - \partial_t\mm_\star^n}_{\bH^1(\Omega)}\nonumber \\
	& \mkern+2mu\le \norm[\Big]{\frac{1}{\tau} \Big(\frac{3}{2}\mm_{\star,h}^{n} - 2\mm_{\star,h}^{n-1} + \frac{1}{2}\mm_{\star,h}^{n-2} -\frac{3}{2}\mm_\star^n + 2\mm_\star^{n-1} - \frac{1}{2}\mm_\star^{n-2} \Big)}_{\bH^1(\Omega)}\\
	& \mkern+35mu + \norm[\Big]{\frac{1}{\tau} \Big(\frac{3}{2}\mm_\star^n - 2\mm_\star^{n-1} + \frac{1}{2}\mm_\star^{n-2}\Big) - \partial_t\mm_\star^n}_{\bH^1(\Omega)}.\nonumber
\end{align*}
Thanks to \eqref{eq: bounds prop s}, we can bound
\begin{align*}
	\norm[\Big]{\frac{1}{\tau} &\Big(\frac{3}{2}\mm_{\star,h}^{n} - 2\mm_{\star,h}^{n-1} + \frac{1}{2}\mm_{\star,h}^{n-2} -\frac{3}{2}\mm_\star^n + 2\mm_\star^{n-1} - \frac{1}{2}\mm_\star^{n-2} \Big)}_{\bH^1(\Omega)}\\
	&\stackrel{\eqref{eq: Ritz_error_1}}{\lesssim} h \norm[\Big]{\frac{1}{\tau}\Big(\frac{3}{2}\mm_\star^n - 2\mm_\star^{n-1} + \frac{1}{2}\mm_\star^{n-2}\Big)}_{\bH^2(\Omega)}\stackrel{\eqref{eq: bound v star H2}}{\lesssim} h \norm{\partial_t\mm}_{L^\infty(\bH^2)}\stackrel{\eqref{eq: bounds prop s}}{\lesssim} h.
\end{align*}
Moreover, a Taylor expansion of $\mm$ around $t_{n-2}$ shows that 
	\begin{align*}
		\mm_\star^{n-1}&=\mm_\star^{n-2}+\tau\partial_t\mm_\star^{n-2}+\frac{\tau^2}{2}\partial_t^2\mm_\star^{n-2}+\frac{1}{2}\int_{t_{n-2}}^{t_{n-1}}(t_{n-1}-s)^2\partial_s^3\mm(s)\d s,\\
		\mm_\star^{n}&=\mm_\star^{n-2}+2\tau\partial_t\mm_\star^{n-2}+2\tau^2\partial_t^2\mm_\star^{n-2}+\frac{1}{2}\int_{t_{n-2}}^{t_{n}}(t_{n}-s)^2\partial_s^3\mm(s)\d s,\\
		\partial_t\mm_\star^n&= \partial_t\mm_\star^{n-2}+2\tau\partial_t^2\mm_\star^{n-2}+\int_{t_{n-2}}^{t_n}(t_n-s)\partial_s^3\mm(s)\d s.
	\end{align*}
	Therefore, we have 
	\begin{align*}
		&\frac{1}{\tau}\Big(\frac{3}{2}\mm_\star^n-2\mm_\star^{n-1}+\frac{1}{2}\mm_\star^{n-2}\Big)-\partial_t\mm_\star^n\\&= 
		\frac{1}{\tau}\Big(\frac{3}{4}\int_{t_{n-2}}^{t_n}(t_n-s)^2\partial_s^3\mm(s)\d s-\int_{t_{n-2}}^{t_{n-1}}(t_{n-1}-s)^2\partial_s^3\mm(s)\d s\Big)-\int_{t_{n-2}}^{t_n}(t_n-s)\partial_s^3\mm(s)\d s.
	\end{align*}
	This implies 
	\begin{equation}\label{eq: time derivative error}
		\norm[\Big]{\frac{1}{\tau}\Big(\frac{3}{2}\mm_\star^n-2\mm_\star^{n-1}+\frac{1}{2}\mm_\star^{n-2}\Big)-\partial_t\mm_\star^n}_{\bH^1(\Omega)}\lesssim \tau^2,
	\end{equation}
	where the hidden constant depends only on 
	$$
	\max_{t\in[0,T]}\norm{\partial^{3}_t\mm(t)}_{\bH^1(\Omega)}=\norm{\partial^{3}_t\mm}_{L^\infty(\bH^1)}\le M <\infty.
	$$ 
With a hidden constant depending only on $\gamma$, $|\Omega|$, and $M$, this implies that
\begin{equation}\label{eq: A1}
A_1 \lesssim h+\tau^2.
\end{equation} 

\textbf{Step~2.}
The second term of the right-hand side of \eqref{eq: s_h^n} is bounded via 

\begin{equation}\label{eq: A_2}
\begin{aligned}
	A_2\coloneqq&\norm[\Big]{\bPh(\mstarhhat^n)\frac{1}{\tau}\Big(\frac{3}{2}\mm_{\star,h}^{n} - 2\mm_{\star,h}^{n-1} + \frac{1}{2}\mm_{\star,h}^{n-2}\Big)-\bP(\mm_\star^n)\partial_t\mm_\star^n}_{\bH^1(\Omega)}\\	
	&\le\norm[\Big]{\big(\bPh(\mstarhhat^n)-\bP(\mstarhhat^n)\big)\frac{1}{\tau}\Big(\frac{3}{2}\mm_{\star,h}^{n} - 2\mm_{\star,h}^{n-1} + \frac{1}{2}\mm_{\star,h}^{n-2}\Big)}_{\bH^1(\Omega)}\\
	&\quad+ \norm[\Big]{\bP(\mstarhhat^n)\Big[\frac{1}{\tau} \Big(\frac{3}{2}\mm_{\star,h}^{n} - 2\mm_{\star,h}^{n-1} + \frac{1}{2}\mm_{\star,h}^{n-2}\Big) - \partial_t\mm_\star^n\Big]}_{\bH^1(\Omega)}\\
	&\quad+ \norm[\Big]{[\bP(\mstarhhat^n)-\bP(\mm_\star^n)]\partial_t\mm_\star^n}_{\bH^1(\Omega)}\eqcolon A_{2,1} + A_{2,2} + A_{2,3}.
\end{aligned}
\end{equation}
The first term of \eqref{eq: A_2} can be bounded by Lemma \ref{lem: approximation discrete projection} and \eqref{eq: bounds prop s} as
\begin{align*}
	A_{2,1}&=\norm[\Big]{\big(\bPh(\mstarhhat^n)-\bP(\mstarhhat^n)\big)\frac{1}{\tau}\Big(\frac{3}{2}\mm_{\star,h}^{n} - 2\mm_{\star,h}^{n-1} + \frac{1}{2}\mm_{\star,h}^{n-2}\Big)}_{\bH^1(\Omega)}\\
	&\mkern-5mu\stackrel{\eqref{eq: approximation discrete projection}}{\lesssim} h\norm[\Big]{\frac{1}{\tau}\Big(\frac{3}{2}\mm_{\star,h}^{n} - 2\mm_{\star,h}^{n-1} + \frac{1}{2}\mm_{\star,h}^{n-2}\Big)}_{\bH^1(\Omega)}\norm{\mstarhhat^n}_{\bW^{1,\infty}(\Omega)}^2\\
	&\lesssim h\norm[\Big]{\frac{3}{2}\d_t\mm_{\star,h}^n-\frac{1}{2}\d _t\mm_{\star,h}^{n-1}}_{\bW^{1,\infty}(\Omega)}\norm{\mhat_{\star,h}^n}_{\bW^{1,\infty}(\Omega)}^2\\
	&\mkern-5mu\stackrel{\eqref{eq: Ritz_stability}}{\lesssim} h\big(\norm{\d _t \mm_\star^n}_{\bW^{1,\infty}(\Omega)} + \norm{\d _t \mm_\star^{n-1}}_{\bW^{1,\infty}(\Omega)}\big)\norm{\mhat_{\star}^n}_{\bW^{1,\infty}(\Omega)}^2\\
	& \lesssim h \norm{\partial_t\mm}_{L^\infty(\bW^{1,\infty})}\norm{\mm}_{L^\infty(\bW^{1,\infty})}^2\stackrel{\eqref{eq: bounds prop s}}{\lesssim} h,
\end{align*}
where the hidden constant depends only on $\gamma$, $M$, and $h\le h_*$.

The second term of \eqref{eq: A_2} is bounded through an $\bH^1$-stability bound for $\bP$, combined with the estimate for $A_1$ obtained in Step 1: Indeed, it holds that
\begin{align*}
	A_{2,2}&=\norm[\Big]{\bP(\mstarhhat^n)\Big[\frac{1}{\tau} \Big(\frac{3}{2}\mm_{\star,h}^{n} - 2\mm_{\star,h}^{n-1} + \frac{1}{2}\mm_{\star,h}^{n-2}\Big) - \partial_t\mm_\star^n\Big]}_{\bH^1(\Omega)}\\
	&\stackrel{\eqref{eq: stability projection}}{\lesssim} \norm{\mstarhhat^n}_{\bW^{1,\infty}(\Omega)}^2 \, A_1\stackrel{\eqref{eq: Ritz_stability}}{\lesssim} \norm{\mhat_{\star}^n}_{\bW^{1,\infty}(\Omega)}^2 \, A_1\stackrel{\eqref{eq: bounds prop s}, \eqref{eq: A1}}{\lesssim}h + \tau^2,
\end{align*}
where the hidden constant depends only on $\gamma$, $|\Omega|$, $M$, and $h\le h_*$.

Finally, thanks to the Lipschitz-type continuity bounds of $\bP$ from Lemma \ref{lem: projection} and \eqref{eq: bounds prop s}, the third term of \eqref{eq: A_2} is bounded as
\begin{align*}
	A_{2,3}&=\norm[\Big]{[\bP(\mstarhhat^n)-\bP(\mm_\star^n)]\partial_t\mm_\star^n}_{\bH^1(\Omega)}\stackrel{\eqref{eq: projection}}{\lesssim} \norm{\partial_t\mm_\star^n}_{\bW^{1,\infty}(\Omega)}\norm{\mstarhhat^n-\mm_\star^n}_{\bH^1(\Omega)}\\
	&\lesssim \norm{\partial_t\mm}_{L^\infty(\bW^{1,\infty})}\big(\norm{\mstarhhat^n-\mhat_\star^n}_{\bH^1(\Omega)}+\norm{\mhat_\star^n-\mm_\star^n}_{\bH^1(\Omega)}\big)\\
	&\mkern-17mu\stackrel{\eqref{eq: Ritz_error_1},\eqref{eq: consistency semi predictor}}{\lesssim} \norm{\partial_t\mm}_{L^\infty(\bW^{1,\infty})}(h\norm{\mhat_\star^n}_{\bH^2(\Omega)} + \tau^2)\stackrel{\eqref{eq: bounds prop s}}{\lesssim} h + \tau^2,
\end{align*}
where the hidden constant depends only on $\gamma$, $M$ and $h\le h_*$.

Therefore, also $A_2$ is bounded by $\OO(h+\tau^2)$, where the hidden constant depends only on $\gamma$, $|\Omega|$, $M$, and $h\le h_*$. This concludes the proof.
\end{proof}
\subsection{Consistency of the full discretization}\label{section:consistency}
The following proposition provides an optimal bound on the defect $\dd_h^n$ of the full discretization, which is defined through \eqref{eq: defect} above.
\begin{proposition}\label{prop: consistency error}
	Let $M>0$. Suppose that the solution $\mm$ of LLG \eqref{eq: LLG system} satisfies the regularity assumptions~\eqref{eq: regularity} and $\ff$ satisfies \eqref{eq: regularity f}. Suppose that $\norm{\ff}_{L^\infty(\bL^2)}\le M$, $\max\limits_{n=2,\dots,N}\normO{\ff_\star^n-\ff_h^n}\le~C_{\ff} h$, and
	\begin{equation}\label{eq: bounds}
	\max\Bigl\{
	\norm{\mm}_{L^\infty(\bW^{2,\infty})},
		\norm{\partial_t\mm}_{L^\infty(\bW^{1,\infty})},
		\norm{\partial_t\mm}_{L^\infty(\bH^2)},\norm{\partial_t^2\mm}_{L^\infty(\bL^2)},
		\norm{\partial_t^3\mm}_{L^\infty(\bH^1)}
	\Bigr\}
	\le M.
	\end{equation}
	Let $h\le h_*$, where $h_*$ is defined in Lemma~\ref{lem: well-posedness mhat star h}. Then, for all $n=2,\ldots,N$, it holds that
	\begin{equation}\label{eq: bound d_h^n}
		\normO{\boldsymbol{d}_h^n}\le \newconst{consistency}(h+\tau^2),
	\end{equation}
	where the constant $\const{consistency}>0$ depends only on $\gamma$, $|\Omega|$, $C_{\ff}$, $M$, and on $h\le h_*$.
\end{proposition}
\begin{proof}
	The proof is split into four steps.

	\textbf{Step~1.}
	Note that \(\partial_t\mm_\star^n\) and \(\mm_\star^n \times \partial_t \mm_\star^n\) belong to \(\bT(\mm_\star^n)\). Therefore, it holds that $\bP(\mm_\star^n)\partial_t\mm_\star^n = \partial_t\mm_\star^n$ and $\bP(\mm_\star^n)(\mm_\star^n \times \partial_t \mm_\star^n) = \mm_\star^n \times \partial_t \mm_\star^n$. Considering~\eqref{eq: D^n}, we can write equation \eqref{eq: alternative LLG with P} at time $t_n$ as
	$$
	\bP(\mm_\star^n) \bD_\star^n \stackrel{\eqref{eq: D^n}}{=}\alpha\partial_t \mm_\star^n + \mm_\star^n \times \partial_t \mm_\star^n - \bP(\mm_\star^n)\big(\lamex^2\Delta \mm_\star^n + \ff_\star^n\big) \stackrel{\eqref{eq: alternative LLG with P}}{=} 0.
	$$
	Using this and the fact that $\bP(\mm_\star^n)$ is symmetric, we can rewrite \eqref{eq: defect 2} for all $\pphi_h\in\bTh(\mstarhhat^n)$ and all $n\ge 2$ as
	\begin{align*}
		\scalarproductO{\boldsymbol{d}_h^n}{\pphi_h} &\stackrel{\eqref{eq: defect 2}}{=} \scalarproductO{\bD_h^n}{\pphi_h}=\scalarproductO{\bD_h^n}{\bPh(\mstarhhat^n)\pphi_h}-\scalarproductO{\bD_\star^n}{\bP(\mm_\star^n)\pphi_h}\\
		&\mkern+7mu= \scalarproductO{\boldsymbol{D}_h^n}{(\bPh(\mstarhhat^n)-\bP(\mstarhhat^n))\pphi_h}+\scalarproductO{\boldsymbol{D}_h^n-\bD_\star^n}{\bP(\mstarhhat^n)\pphi_h}	\\
		&\mkern+50mu+ \scalarproductO{\bD_\star^n}{{\big(}\bP(\mstarhhat^n)-\bP(\mm_\star^n){\big)}\pphi_h}.
	\end{align*}
	Choosing $\pphi_h=\boldsymbol{d}_h^n\in \bT_h(\mstarhhat^n)$, we are led to 
	\begin{equation} \label{eq: consistency error equation}
		\begin{aligned}
			\normO{\boldsymbol{d}_h^n}^2 &\le \normO{\boldsymbol{D}_h^n}\normO{\big(\bPh(\mstarhhat^n)-\bP(\mstarhhat^n)\big)\dd_h^n} 
			+ \normO{\boldsymbol{D}_h^n-\bD_\star^n}\normO{\dd_h^n} \\
			&\quad + \norm{\bD_\star^n}_{\bL^\infty(\Omega)}\normO{\bP(\mstarhhat^n)-\bP(\mm_\star^n)}\normO{\dd_h^n} 
			\eqqcolon A_1 + A_2 + A_3.
		\end{aligned}
	\end{equation}
	We estimate the terms $A_1$, $A_2$, and $A_3$ in three separate steps. More precisely, we will show that 
	$A_1\lesssim h\normO{\dd_h^n}$, $A_2\lesssim(h+\tau^2)\normO{\dd_h^n}$, and $A_3\lesssim (h+\tau^2)\normO{\dd_h^n}$ so that the claim~\eqref{eq: bound d_h^n} follows from \eqref{eq: consistency error equation}.

	\textbf{Step~2 (Bound on $\boldsymbol{A_1}$).} By definition \eqref{eq: D_h^n}, we have that
	\begin{equation}\label{eq: D_h^n triangular}
	\normO{\bD_h^n}\le \alpha\normO{\vstarh^{n-1}}+\normO{\mstarhhat^n\times\vstarh^{n-1}}+\normO{\lamex^2\Delta\mm_\star^n+\ff_h^n}.
	\end{equation}
	Moreover, we also have that
	\begin{align*}
		\normO{\mstarhhat^n\times\vstarh^{n-1}}&\le \norm{\mstarhhat^n}_{\bL^\infty(\Omega)}\normO{\vstarh^{n-1}}\stackrel{\eqref{eq: Ritz_stability}}{\lesssim}\norm{\mhat_\star^n}_{\bW^{1,\infty}(\Omega)}\normO{\vstarh^{n-1}}\stackrel{\eqref{eq: bounds},\eqref{eq: bounds C_X C_Y}}{\lesssim} 1.
	\end{align*}
	Together with $\normO{\lamex^2\Delta\mm_\star^n+\ff_\star^n}\lesssim\normO{\Delta\mm_\star^n}+\norm{\ff}_{L^\infty(\bL^2)}\le M$ from \eqref{eq: bounds} and $\normO{\ff_\star^n-\ff_h^n}\le C_{\ff} h$ by assumption, we obtain
	\begin{equation}\label{eq: bound D_h^n}
		\normO{\boldsymbol{D}_h^n} \lesssim 1.
	\end{equation}
	Thanks to Lemma	\ref{lem: approximation discrete projection}, the inverse estimate \eqref{eq: inverse estimate L^p}, and the $\bW^{1,\infty}$-stability of the Ritz projection~\eqref{eq: Ritz_stability}, we have
	\begin{align*}
		\normO{\big(\bP(\mstarhhat^n)&-\bPh(\mstarhhat^n)\big)\boldsymbol{d}_h^n} \stackrel{\eqref{eq: approximation discrete projection}}{\lesssim} h^2 \norm{\dd_h^n}_{\bH^1(\Omega)}\norm{\mstarhhat^n}_{\bW^{1,\infty}(\Omega)}^2\\
		&\stackrel{\eqref{eq: inverse estimate L^p}}{\lesssim} h \normO{\dd_h^n}\norm{\mstarhhat^n}_{\bW^{1,\infty}(\Omega)}^2\stackrel{\eqref{eq: Ritz_stability}}{\lesssim} h \normO{\dd_h^n}\norm{\widehat{\mm}_\star^n}_{\bW^{1,\infty}(\Omega)}^2\lesssim h \normO{\dd_h^n}\norm{\mm}^2_{L^\infty(\bW^{1,\infty})}.
	\end{align*}
	Together with \eqref{eq: bounds} and \eqref{eq: bound D_h^n}, this implies
	\begin{equation}\label{eq: bound A1}
		A_1\stackrel{\eqref{eq: bounds}}{\lesssim} h\normO{\boldsymbol{d}_h^n},
	\end{equation}
where the hidden constant depends only on $\gamma$, $|\Omega|$, $C_{\ff}$, $M$, and $h\le h_*$.

	\textbf{Step~3 (Bound on $\boldsymbol{A_2}$).} By definition \eqref{eq: D^n}--\eqref{eq: D_h^n}, we notice that
	\begin{align*}
		\normO{\boldsymbol{D}_h^n-\bD_\star^n}&\le \alpha \normO{\vstarh^{n-1}-\partial_t\mm_\star^n}+\normO{\mstarhhat^n\times\vstarh^{n-1}-\mm_\star^n\times\partial_t\mm_\star^n}+\normO{\ff_\star^n-\ff_h^n}.
	\end{align*}
	We note that $\partial_t\mm_\star^n=\bP(\mm_\star^n)\partial_t\mm_\star^n$, because of $\partial_t\mm_\star^n\in \bT(\mm_\star^n)$. With the term $A_2$ from \eqref{eq: A_2}, Step~2 of the proof of Proposition~\ref{prop: bound s_h^n} thus shows 
	\begin{align}\label{eq: bound v and partial}
		\normO{\vstarh^{n-1}-\partial_t\mm_\star^n} \stackrel{\eqref{eq: v_star_ritz}}{=} \normO[\Big]{\bPh(\mstarhhat^n)\frac{1}{\tau}\Big(\frac{3}{2}\mm_{\star,h}^{n} - 2\mm_{\star,h}^{n-1} + \frac{1}{2}\mm_{\star,h}^{n-2}\Big)-\bP(\mm_\star^n)\partial_t\mm_\star^n}\stackrel{\eqref{eq: A_2}}{\le} A_2\lesssim h+\tau^2,
	\end{align}
	where the hidden constant depends only on $\gamma$, $|\Omega|$, $M$, and on $h\le h_*$. Thanks to \eqref{eq: bounds}, we have
	\begin{align*}
		\normO{\mstarhhat^n\times \vstarh^{n-1}-\mhat_{\star}^{n}\times\vv_\star^{n-1}}&=\normO{\mstarhhat^n\times(\vstarh^{n-1}-\vv_\star^{n-1})+(\mstarhhat^n-\mhat_\star^n)\times\vv_\star^{n-1}}\nonumber\\
		&\mkern-200mu\le \norm{\mstarhhat^n}_{\bL^\infty(\Omega)}\big(\normO{\vv_{\star,h}^{n-1}-\partial_t\mm_{\star}^n}+\normO{\partial_t\mm_{\star}^n-\vv_\star^{n-1}} \big)+\norm{\vv_\star^{n-1}}_{\bL^\infty(\Omega)}\normO{\mstarhhat^n-\mhat_\star^n}\\
		&\mkern-233mu\stackrel{\eqref{eq: bound v and partial},\eqref{eq: consistency semi derivative},\eqref{eq: Ritz_error_1}}{\lesssim} \norm{\mstarhhat^n}_{\bL^\infty(\Omega)}(h+\tau^2)+\norm{\vv_\star^{n-1}}_{\bL^\infty(\Omega)}\norm{\mhat_\star^n}_{\bH^2(\Omega)}h\nonumber\\
		&\mkern-205mu\stackrel{\eqref{eq: Ritz_stability}}{\lesssim} \norm{\widehat{\mm}_\star^n}_{\bW^{1,\infty}(\Omega)}(h+\tau^2)+\norm{\vv_\star^{n-1}}_{\bL^\infty(\Omega)}\norm{\mhat_\star^n}_{\bH^2(\Omega)}h\nonumber\stackrel{\eqref{eq: bounds C_X C_Y}}{\lesssim} h+\tau^2.
	\end{align*}
	Thanks to \eqref{eq: bounds}, Lemma \ref{lem: semi-discrete approximation} proves
	\begin{align*}
		\normO{\mhat_\star^n\times\vv_\star^{n-1}-\mm_\star^n\times\partial_t\mm_\star^n}&=\normO{\mhat_\star^n\times(\vv_\star^{n-1}-\partial_t\mm_\star^n)+(\mhat_\star^n-\mm_\star^n)\times\partial_t\mm_\star^n}\nonumber\\
		&\mkern-200mu\lesssim \norm{\mhat_\star^n}_{\bL^\infty(\Omega)}\normO{\vv_\star^{n-1}-\partial_t\mm_\star^n}+\normO{\mhat_\star^n-\mm_\star^n}\norm{\partial_t\mm_\star^n}_{\bL^\infty(\Omega)}\stackrel{\eqref{eq: consistency semi predictor},\eqref{eq: consistency semi derivative},\eqref{eq: bounds}}{\lesssim}\tau^2.
	\end{align*}
	Combining the last two estimates, we have that 
	\begin{align}
		&\normO{\mstarhhat^n\times\vstarh^{n-1}-\mm_\star^n\times\partial_t\mm_\star^n}\nonumber\\
		&\le\normO{\mstarhhat^n\times \vstarh^{n-1}-\mhat_{\star}^{n}\times\vv_\star^{n-1}}+\normO{\mhat_\star^n\times\vv_\star^{n-1}-\mm_\star^n\times\partial_t\mm_\star^n}\lesssim h + \tau^2.\label{eq: bound two cross 3}
	\end{align}
	Finally, recall that $\normO{\ff_\star^n-\ff_h^n}\le C_{\ff} h$.
	Together with \eqref{eq: bound v and partial}--\eqref{eq: bound two cross 3}, this gives
	\begin{equation}\label{eq: bound D_h^n - D^n}
		\normO{\boldsymbol{D}_h^n-\bD_\star^n}\lesssim h+\tau^2
	\end{equation}
	and hence, with a hidden constant depending only on $\gamma$, $M$, and $h\le h_*$,
	\begin{equation}\label{eq: bound A2}
		A_2\lesssim (h+\tau^2)\normO{\boldsymbol{d}_h^n}.
	\end{equation}

	\textbf{Step~4 (Bound on $\boldsymbol{A_3}$).} Recall $\bD_\star^n$ from \eqref{eq: D^n}. The regularity assumptions \eqref{eq: bounds} guarantee that
	$
		\norm{\bD_\star^n}_{\bL^\infty(\Omega)}\lesssim 1.
	$
	Moreover, applying Lemma \ref{lem: projection} with $\vv \equiv \boldsymbol{1}$ yields
	\begin{equation*}
		\begin{split}
		\normO{\bP(\mstarhhat^n)-\bP(\mm_\star^n)}&\stackrel{\eqref{eq: projection}}{\lesssim}\normO{\mstarhhat^n-\mm_\star^n}\norm{\boldsymbol{1}}_{\bL^\infty(\Omega)}\\
		&\mkern+5mu\le \normO{(\bI-\bR_h)\mhat_\star^n}+\normO{\mhat_\star^n-\mm_\star^n}\stackrel{\eqref{eq: Ritz_error_1},\eqref{eq: consistency semi predictor}}{\lesssim}\norm{\mhat_\star^n}_{\bH^2(\Omega)}h+ \tau^2.
		\end{split}
	\end{equation*}
	Altogether, this implies
	\begin{equation}\label{eq: bound A3}
		A_3\lesssim (h+\tau^2)\normO{\boldsymbol{d}_h^n},
	\end{equation}
	where the hidden constant depends only on $\gamma$, $M$, and on $h\le h_*$.	As noted in Step 1, the estimates of Step 2--4 conclude the proof.
\end{proof}
\section{Stability estimates and proof of main result}\label{sec: stability and main proof}
\subsection{Stability of the full discretization.}
Next, we obtain an estimate on the error $\boldsymbol{e}_h^n$ from \eqref{eq: e_h^n} using the error equation \eqref{eq: error equation}. Notice that $\oomega_h^{n-1}$ is not an admissible test function for~\eqref{eq: error equation}, since $\oomega_h^{n-1}\notin\bTh(\mhat_h^n)$ in general. Following \cite{feischl,bkw2024}, we thus consider $\pphi_h\coloneqq\bPh(\mhat_h^n)\oomega_h^{n-1}\in\bTh(\mhat_h^n)$ as a test function in \eqref{eq: error equation}. The following lemma shows that $\pphi_h$ is a perturbation of $\oomega_h^{n-1}$ and provides an estimate on the correction error.

\begin{lemma}\label{lem: q_h^n}
	Suppose that $\mm$ satisfies the regularity assumption~\eqref{eq: regularity} and suppose that $1/2\le |\mhat_h^n| \le 2$. Let $M>0$ such that $\norm{\partial_t\mm}_{L^\infty(\bW^{1,\infty})}\le M$. Then, for all $n=2,\ldots,N$ there exists a function $\boldsymbol{q}_h^n\in \boldsymbol{\mathcal{S}}^1({\mathcal{T}}_h)$ such that
	\begin{equation}
		\pphi_h \coloneqq \bPh(\mhat_h^n)\oomega_h^{n-1}=\oomega_h^{n-1}+\boldsymbol{q}_h^n,
	\end{equation}
	and, for any $k\in\{0,1\}$, it holds that
	\begin{equation}\label{eq: bound q_h^n}
		\norm{\boldsymbol{q}_h^n}_{\bH^k(\Omega)}\le \newconst{const_q} \norm{\widehat{\ee}_h^{n}}_{\bH^k(\Omega)}.
	\end{equation} 
	The constant $\const{const_q}>0$ depends only on $\gamma$, $|\Omega|$, $M$, and on $h\le h_*$.
\end{lemma}
\begin{proof}
	We note that $\bPh(\mhat_h^n)\vv_h^{n-1}=\vv_h^{n-1}\in \bTh(\mhat_h^n)$ by construction of the method \eqref{eq: discrete formulation} and $\bPh(\mstarhhat^n)\vstarh^{n-1} = \vstarh^{n-1}\in \bTh(\mstarhhat^n)$ by definition \eqref{eq: v_star_ritz}. Therefore,
	$\pphi_h = \bPh(\mhat_h^n)\oomega_h^{n-1}$ can be rewritten as
	\begin{align*}
		\pphi_h = \bPh(\mhat_h^n)\oomega_h^{n-1} &\stackrel{\eqref{eq: omega_h^n-1}}{=} \bPh(\mhat_h^n)\vv_h^{n-1} - \bPh(\mhat_h^n)\vv_{\star,h}^{n-1} \\
		& \mkern+7mu= \bPh(\mhat_h^n)\vv_h^{n-1}-\bPh(\mstarhhat^n)\vstarh^{n-1} + \big(\bPh(\mstarhhat^n) - \bPh(\mhat_h^n)\big)\vv_{\star,h}^{n-1} \\
		& \mkern+7mu= \vv_h^{n-1} - \vstarh^{n-1} + \big(\bPh(\mstarhhat^n) - \bPh(\mhat_h^n)\big)\vv_{\star,h}^{n-1}\eqcolon \oomega_h^{n-1} + \boldsymbol{q}_h^n.
	\end{align*}
	To bound $\norm{\boldsymbol{q}_h^n}_{\bH^k(\Omega)}$, we apply Lemma \ref{lem: stability discrete projection} to obtain
	\begin{align*}
		\norm{\boldsymbol{q}_h^n}_{\bH^k(\Omega)} &\lesssim \norm{\vv_{\star,h}^{n-1}}_{\bW^{k,\infty}(\Omega)} \norm{\mstarhhat^n-\mhat_h^n}_{\bH^{k}(\Omega)}\stackrel{\eqref{eq: bounds C_X C_Y}}{\lesssim} \norm{\mstarhhat^n-\mhat_h^n}_{\bH^k(\Omega)}\stackrel{\eqref{eq: hat_e_h^n}}{=}\norm{\widehat{\ee}_h^n}_{\bH^k(\Omega)},
	\end{align*}
	where the hidden constants depend only on $\gamma$, $|\Omega|$, $k$, $M$, and $h\le h_*$.
\end{proof}

We can finally prove the following stability result, which turns out to be the key result besides Proposition~\ref{prop: consistency error} and Proposition \ref{prop: bound s_h^n}.
\begin{proposition}[Stability of the full discretization]\label{prop: stability full discretization}
	Suppose that $\mm$ satisfies the regularity assumption \eqref{eq: regularity}. Let $M>0$ such that
		\begin{equation}\label{eq: regularity last prop}
			\max\Bigl\{
			\norm{\mm}_{L^\infty(\bW^{2,\infty})},
			\norm{\partial_t\mm}_{L^\infty(\bW^{1,\infty})},
			\norm{\partial_t\mm}_{L^\infty(\bH^{2})},
			\norm{\partial_t^{2}\mm}_{L^\infty(\bL^{2})},
			\norm{\partial_t^{3}\mm}_{L^\infty(\bH^{1})}
			\Bigr\}\le M.
		\end{equation}
		and
		$
			\norm{\ff}_{L^\infty(\bL^\infty)}\le M.
		$
		Recall the error $\ee_h^n$ from \eqref{eq: e_h^n}, recall $h_*$ from Lemma~\ref{lem: well-posedness mhat star h} and the assumption on the initial value $\mm^0$ from \eqref{eq: assumption m0}. Let $C_{\rm CFL}>0$ and $0<\varepsilon<1$. Then, there exist constants $\bar{h}, \bar{\tau}>0$ such that for any $h\le \bar{h}\le h_*$ and $\tau \le \bar{\tau}$ satisfying the (mild) condition
		\begin{equation}\label{eq: CFL}
			\tau \le C_{\rm CFL} h^{(1+\varepsilon)/4}
		\end{equation}
		Algorithm~\ref{alg: full discr} guarantees the following bound: For any $n=0,\ldots,N$, it holds that
		\begin{equation}\label{eq: stability full discretization}
			\norm{\boldsymbol{e}_h^n}_{\bH^1(\Omega)}^2\le \newconst{C_stab}\Big(\norm{\ee_h^0}_{\bH^1(\Omega)}^2+\norm{\ee_h^1}_{\bH^1(\Omega)}^2+\tau\sum_{j=2}^{n}\norm{\sss_h^j}^2_{\bH^1(\Omega)}+\tau\sum_{j=2}^{n}\normO{\dd_h^j}^2\Big).
		\end{equation}
		The constant $\const{C_stab}>0$, depends only on $\gamma$, $\alpha$, $\lamex$, $M$, $|\Omega|$, $\mu^\pm$, $T$, $C_{\rm CFL}$, $C_0$, $h\le \bar{h}$, and $\tau \le \bar{\tau}$, but is independent of $h$, $\tau$, and $n$.
\end{proposition}
\begin{proof}
We firstly emphasize that, in order to ensure the stability result on $\III_h$ from Lemma~\ref{lem: stability I_h} and the stability of the
 normalization from $\bN$ from Lemma~\ref{lem: normalization bounds}, we need to ensure that $|\mhat_h^n(\xx)|$ remains uniformly bounded 
 away from zero for all $x\in \Omega$. To this end, we establish a $\bL^\infty$-bound on the errors $\ee_h^j = \mm_h^j-\mstarh^j$. Overall, our argument is split into five steps. In Step~1, we will state the necessary parameter constraints and the hypothesis. In 
  particular, given $n\leq N$, we will suppose, as an induction hypothesis, the $\bL^\infty$-bound on $\ee_h^j$ for $j=0,\dots,n-1$. In 
  Step~2, we show that these assumptions ensure that $|\mhat_h^j(x)|$ is well-defined and uniformly bounded away from zero for all 
  $x\in\Omega$ and $j=1,\dots,n$. In Step~3, we derive a bound for the quantity $\tau \sum_{j=2}^{n} \normO{\oomega_h^{j-1}}^2$, which is 
  then used in Step~4 to obtain a bound for $\norm{\ee_h^j}_{\bH^1(\Omega)}^2$. Finally, in Step~5, we will prove the claim of the 
  proposition and, with this, that $\norm{\ee_h^j}_{\bL^{\infty}(\Omega)}$ is bounded also for $j=n$, completing the induction step.

\textbf{Step~1 (Choice of $\bar{h}$ and preliminary assumptions on $\norm{\ee_h^j}_{\bL^{\infty}(\Omega)}$).}
Recall the constants $C_0$ from \eqref{eq: assumption m0}, $\const{feischl_tran}$ from \eqref{eq: assumption m1}, $C_{\rm R}$ from \eqref{eq: Ritz_error_1}--\eqref{eq: Ritz_stability}, $M$ from \eqref{eq: regularity last prop} and $C_{\rm CFL}$ from \eqref{eq: CFL}. Let $\newconst{final_const}\coloneqq \max\{3C_{\rm R}M,2M\}$ and note that $\const{final_const}$ depends only on $\gamma$ and $M$. Moreover, recall the following inverse and Sobolev inequalities for $q = 2d$:
\begin{equation}\label{eq: sobolev}
	C_{\rm inv}^{-1}\norm{\psi_h}_{\bL^\infty(\Omega)}\leq h^{-{d/q}}\norm{\psi_h}_{\bL^q(\Omega)}\leq C_{\rm Sob} h^{-1/2}\norm{\psi_h}_{\bH^1(\Omega)}\quad \text{for all}\quad\psi_h\in \boldsymbol{\mathcal{S}}^1(\mathcal{T}_h),
\end{equation}
where the overall constant $C_{\infty}\coloneq C_{\rm inv}C_{\rm Sob}>0$ depends only on $\gamma$ and $|\Omega|$. Recall $h_*>0$ from Lemma~\ref{lem: well-posedness mhat star h}. Given $\rho\coloneq 1/12$, we choose $0<\bar{h}\le h_*$ such that 
\begin{subequations}\label{eq: bar h}
	\begin{align}
	\bar{h}^{1/2} \le \frac{\rho}{C_{\infty}(C_0+C_{\rm R}M)} \label{eq: bar h 1}\\
	\bar{h}^{1/2} + C_{\rm CFL}^2 \bar{h}^{\varepsilon/2} \le \frac{\rho}{C_{\infty}\const{feischl_tran}} \label{eq: bar h 2}\\
	\bar{h} + C_{\rm CFL} \bar{h}^{{(1+\varepsilon)}/{4}} \le \frac{1}{4\newconst{final_const}} \label{eq: bar h 3}
	\end{align}
\end{subequations}
Finally, we suppose that there holds 
\begin{equation}\label{eq: induction claim}
	\norm{\ee_h^j}_{\bL^{\infty}(\Omega)} \le \rho\quad \text{for all}\quad j=0,\dots,n-1,
\end{equation}
but stress that \eqref{eq: induction claim} will indeed later be proven by induction; see Step~6 below. In the following Steps~2--5, let $0< h \leq \bar{h}$ and $\tau>0$ satisfying \eqref{eq: CFL}. 

\textbf{Step~2 (Well definedness of $\boldsymbol{|\mhat_h^j(x)|}$).}
Thanks to the regularity assumption \eqref{eq: regularity last prop}, we estimate for all $j= 2,\ldots,n$ the following quantity:
\begin{align*}
	\norm{|\mhat_h^j|-1}_{\bL^\infty(\Omega)} &= \norm{|\mhat_h^j|-|\mm_\star^j|}_{\bL^\infty(\Omega)}\\
	&\le \norm{\mhat_h^j-\mstarhhat^j}_{\bL^\infty(\Omega)} + \norm{\mstarhhat^j-\mhat_\star^j}_{\bL^\infty(\Omega)} + \norm{\mhat_\star^j-\mm_\star^j}_{\bL^\infty(\Omega)}\\
	&\mkern-22mu\stackrel{\eqref{eq: Ritz_error_2},\eqref{eq: m_hat_star}}{\le} \norm{\widehat{\ee}_h^j}_{\bL^\infty(\Omega)} + C_{\rm R} h\norm{\mhat_\star^j}_{\bW^{2,\infty}(\Omega)} + \tau\norm{\d_t\mm_\star^{j}-\d_t\mm_\star^{j-1}}_{\bL^{\infty}(\Omega)}\\
	&\le \norm{\widehat{\ee}_h^j}_{\bL^\infty(\Omega)} + 3 C_{\rm R} h\norm{\mm}_{L^\infty(\bW^{2,\infty})} + 2\tau\norm{\partial_t\mm}_{L^\infty(\bL^{\infty})}\\
	& = \norm{2\ee_h^{j-1}-\ee_h^{j-2}}_{\bL^\infty(\Omega)} + 3C_{\rm R}h\norm{\mm}_{L^\infty(\bW^{2,\infty})} + 2\tau\norm{\partial_t\mm}_{L^\infty(\bL^{\infty})}\\
	& \mkern-21mu\stackrel{\eqref{eq: induction claim},\eqref{eq: regularity last prop}}{\le} 3\rho +3C_{\rm R}M{h} + 2M{\tau}\stackrel{\eqref{eq: CFL}}{\le} 3 \rho + 3 C_{\rm R}M h + 2 M C_{\rm CFL} h^{(1+\varepsilon)/4}\\
	& \le 3 \rho + 3 C_{\rm R}M \bar{h} + 2 M C_{\rm CFL} \bar{h}^{(1+\varepsilon)/4} \stackrel{\eqref{eq: bar h 3}}{\le} 3\rho + \const{final_const}\frac{1}{4\const{final_const}}= 3 \rho + \frac{1}{4}.
\end{align*}
This, together with the choice of $\rho=1/12$, yields $\norm{|\mhat_h^j|-1}_{\bL^\infty(\Omega)}\le 1/2$ and therefore
\begin{equation}\label{eq: bound m hat (x)}
\frac{1}{2} \le |\mhat_h^j(\xx)| \le \frac{3}{2} \quad \text{for all } \xx\in \Omega \text{ and } j=2,\ldots,n.
\end{equation}
This guarantees that the quantities $\mhat_h^j$ are uniformly bounded away from zero and therefore satisfy the assumptions of all the previous results, such as Lemmas~\ref{lem: bounds r},~\ref{lem: q_h^n}, and Proposition~\ref{prop: bound s_h^n}.

\textbf{Step~3 (Bound on $\boldsymbol{\tau\sum\limits_{j=2}^n\normO{\oomega_h^{j-1}}^2}$).}
Recall $\qq_h^j$ from Lemma~\ref{lem: q_h^n}. Consider $1\le j\le n$ and test the error equation \eqref{eq: error equation} with $\pphi_h=\oomega_h^{j-1}+\qq_h^j\in \bT_h(\mhat_h^n)$ to obtain
\begin{align*}
	\alpha \scalarproductO{\oomega_h^{j-1}}{\oomega_h^{j-1}+\qq_h^j} &+ \scalarproductO{\widehat{\ee}_h^j\times\vv_{\star,h}^{j-1}}{\oomega_h^{j-1}+\qq_h^j} + \scalarproductO{\widehat{\mm}_h^{j}\times \oomega_h^{j-1}}{\oomega_h^{j-1}+\qq_h^j} \\
	&+ \lamex^2\scalarproductO{\nabla\ee_h^j}{\nabla(\oomega_h^{j-1}+\qq_h^j)} = -\scalarproductO{\rr_h^j}{\oomega_h^{j-1}+\qq_h^j}.
\end{align*}
This can be rewritten as 
\begin{equation}\label{eq: equality_1}
	\begin{aligned}
	\alpha \normO{\oomega_h^{j-1}}^2 &+ \lamex^2\scalarproductO{\nabla\ee_h^j}{\nabla\oomega_h^{j-1}} = -\alpha \scalarproductO{\oomega_h^{j-1}}{\qq_h^j} - \scalarproductO{\widehat{\ee}_h^j\times\vv_{\star,h}^{j-1}}{\oomega_h^{j-1}+\qq_h^j}\\
	&- \scalarproductO{\widehat{\mm}_h^{j}\times \oomega_h^{j-1}}{\qq_h^j}-\lamex^2\scalarproductO{\nabla\ee_h^j}{\nabla\qq_h^j} -\scalarproductO{\rr_h^j}{\oomega_h^{j-1}+\qq_h^j}.
	\end{aligned}
\end{equation}
Algebraic computations for the BDF2 time stepping (see, e.g., \cite{Dahlquist, baiocchi_crouzeix, nevanlinna_odeh}, \cite[Section V.6]{hairer}, or \cite[Lemma 9]{part1} for details) show that
\begin{align*}
\scalarproductO{\nabla\ee_h^j}{\nabla\oomega_h^{j-1}} &\stackrel{\eqref{eq: omega_h^n-1}}{=} \frac{1}{4\tau}\normO{\tau^2\nabla\d_t^2\ee_h^{j}}^2+
\frac{1}{\tau}\Big[\frac{1}{4}\normO{\nabla \ee_h^{j-1}}^2 - \scalarproductO{\nabla \ee_h^{j}}{\nabla \ee_h^{j-1}} + \frac{5}{4}\normO{\nabla \ee_h^{j}}^2\Big]\\
&\qquad-\frac{1}{\tau} \Big[ \frac{1}{4}\normO{\nabla \ee_h^{j-2}}^2 - \scalarproductO{\nabla \ee_h^{j-1}}{\nabla \ee_h^{j-2}} + \frac{5}{4}\normO{\nabla \ee_h^{j-1}}^2 \Big] + \scalarproductO{\nabla\ee_h^j}{\nabla\sss_h^j}.
\end{align*}
With the entries $(g_{\mu\nu})_{\mu,\nu=1,2}$ of the matrix 
$G =\dfrac{1}{4} 
\begin{pmatrix}	
	1 & -2 \\
	-2 & 5	
\end{pmatrix},$
this implies that
\begin{align}\label{eq: inequality_2}
	\scalarproductO{\nabla\ee_h^j}{\nabla\oomega_h^{j-1}} &\ge
	\frac{1}{\tau}\Big[g_{11}\normO{\nabla \ee_h^{j-1}}^2 + 2g_{12}\scalarproductO{\nabla \ee_h^{j}}{\nabla \ee_h^{j-1}} + g_{22}\normO{\nabla \ee_h^{j}}^2\Big]\\
	&\mkern-50mu-\frac{1}{\tau} \Big[ g_{11}\normO{\nabla \ee_h^{j-2}}^2 + 2g_{12}\scalarproductO{\nabla \ee_h^{j-1}}{\nabla \ee_h^{j-2}} + g_{22}\normO{\nabla \ee_h^{j-1}}^2 \Big] + \scalarproductO{\nabla\ee_h^j}{\nabla\sss_h^j}.\nonumber
\end{align}
Combining \eqref{eq: equality_1} with \eqref{eq: inequality_2} and using the Cauchy-Schwarz inequality together with the uniform bounds \eqref{eq: bounds C_X C_Y} on $\norm{\vstarh^{j-1}}_{\bL^\infty(\Omega)}$ from Lemma~\ref{lem: various v estimates} and on $\norm{\mhat_h^j}_{\bL^\infty(\Omega)}$ guaranteed by~\eqref{eq: bound m hat (x)}, we derive that
\begin{align*}
	\alpha \normO{\oomega_h^{j-1}}^2 &+ 
	\frac{\lamex^2}{\tau}\Big[g_{11}\normO{\nabla \ee_h^{j-1}}^2 + 2g_{12}\scalarproductO{\nabla \ee_h^{j}}{\nabla \ee_h^{j-1}} + g_{22}\normO{\nabla \ee_h^{j}}^2\Big]\\
	&\qquad-\frac{\lamex^2}{\tau} \Big[g_{11}\normO{\nabla \ee_h^{j-2}}^2 + 2g_{12}\scalarproductO{\nabla \ee_h^{j-1}}{\nabla \ee_h^{j-2}} + g_{22}\normO{\nabla \ee_h^{j-1}}^2 \Big]\\
	&\mkern-50mu\le -\alpha\scalarproductO{\oomega_h^{j-1}}{\qq_h^j} - \scalarproductO{\widehat{\ee}_h^j\times\vv_{\star,h}^{j-1}}{\oomega_h^{j-1}+\qq_h^j} - \scalarproductO{\widehat{\mm}_h^{j}\times \oomega_h^{j-1}}{\qq_h^j} \\
	&\qquad-\lamex^2\scalarproductO{\nabla\ee_h^j}{\nabla\qq_h^j} -\scalarproductO{\rr_h^j}{\oomega_h^{j-1}+\qq_h^j}-\lamex^2\scalarproductO{\nabla\ee_h^j}{\nabla\sss_h^j}\\
	& \mkern-50mu\leq \alpha\normO{\oomega_h^{j-1}}\normO{\qq_h^j} + \norm{\vstarh^{j-1}}_{\bL^\infty(\Omega)}\normO{\widehat{\ee}_h^j}(\normO{\oomega_h^{j-1}}+\normO{\qq_h^j})+ \norm{\mhat_h^j}_{\bL^\infty(\Omega)}\normO{\oomega_h^{j-1}}\normO{\qq_h^j} \\
	&\qquad + \normO{\nabla\ee_h^j}\lamex^2(\normO{\nabla\sss_h^j}+\normO{\nabla\qq_h^j}) + \normO{\boldsymbol{r}_h^j}(\normO{\oomega_h^{j-1}}+\normO{\qq_h^j})\\
	&\mkern-68mu\stackrel{\eqref{eq: bounds C_X C_Y},\eqref{eq: bound m hat (x)}}{\lesssim} \alpha\normO{\oomega_h^{j-1}}\normO{\qq_h^j} + \normO{\widehat{\ee}_h^j}(\normO{\oomega_h^{j-1}}+\normO{\qq_h^j})+\normO{\oomega_h^{j-1}}\normO{\qq_h^j} \\
	& \qquad+ \normO{\nabla\ee_h^j}(\normO{\nabla\sss_h^j}+\normO{\nabla\qq_h^j})+ \normO{\boldsymbol{r}_h^j}(\normO{\oomega_h^{j-1}}+\normO{\qq_h^j}),
\end{align*}
where the hidden constant depends only on $\gamma$, $\lamex$, $M$, $|\Omega|$, and $\bar{h}$. The definition of $\mhat_h^j$ and $\widehat{\ee}_h^j$ guarantee that, for $k\in\{0,1\}$,
\begin{align*}
\norm{\widehat{\ee}_h^j}_{\bH^k(\Omega)} &=\norm{2\ee_h^{j-1} - \ee_h^{j-2}}_{\bH^k(\Omega)} \lesssim \norm{\ee_h^{j-1}}_{\bH^k(\Omega)} + \norm{\ee_h^{j-2}}_{\bH^k(\Omega)}.
\end{align*}
This last bound, together with Young's inequality on the product terms, absorption into the term $\normO{\oomega_h^{j-1}}$ and the bounds on $\norm{\qq_h^j}_{\bH^1(\Omega)}$ from \eqref{eq: bound q_h^n} and $\normO{\rr_h^j}$ from \eqref{eq: bound r_h^n} allow us to estimate further as
\begin{align*}
	\frac{\alpha}{2} \normO{\oomega_h^{j-1}}^2 &+ \frac{\lamex^2}{\tau}\Big[g_{11}\normO{\nabla \ee_h^{j-1}}^2 + 2g_{12}\scalarproductO{\nabla \ee_h^{j}}{\nabla \ee_h^{j-1}} + g_{22}\normO{\nabla \ee_h^{j}}^2\Big] \\
	& \qquad - \frac{\lamex^2}{\tau} \Big[g_{11}\normO{\nabla \ee_h^{j-2}}^2 + 2g_{12}\scalarproductO{\nabla \ee_h^{j-1}}{\nabla \ee_h^{j-2}} + g_{22}\normO{\nabla \ee_h^{j-1}}^2 \Big] \\
	& \mkern-30mu\lesssim \norm{\ee_h^j}_{\bH^1(\Omega)}^2 + \norm{\ee_h^{j-1}}_{\bH^1(\Omega)}^2 + \norm{\ee_h^{j-2}}_{\bH^1(\Omega)}^2 + \normO{\nabla\sss_h^j}^2 + \norm{\qq_h^j}_{\bH^1(\Omega)}^2 + \normO{\boldsymbol{r}_h^j}^2\\
	& \mkern-48mu\stackrel{\eqref{eq: bound r_h^n},\eqref{eq: bound q_h^n}}{\lesssim} \sum_{k=0}^{2} \norm{\ee_h^{j-k}}_{\bH^1(\Omega)}^2 + \normO{\nabla\sss_h^j}^2 + \normO{\boldsymbol{d}_h^j}^2.
\end{align*}
Overall, the hidden constants depend only on $\gamma$, $\alpha$, $\lamex$, $M$, $|\Omega|$, and $\bar{h}$.
Multiplying by $\tau$ and summing over $j=2,\dots, n\le N$, we obtain
\begin{align}
	\frac{\alpha}{2} \tau\sum_{j=2}^{n} \normO{\oomega_h^{j-1}}^2 &+ \lamex^2\Big[g_{11}\normO{\nabla \ee_h^{n-1}}^2 + 2g_{12}\scalarproductO{\nabla \ee_h^{n}}{\nabla \ee_h^{n-1}} + g_{22}\normO{\nabla \ee_h^{n}}^2\Big]\nonumber\\
	& \mkern-30mu\lesssim \lamex^2\Big[g_{11}\normO{\nabla \ee_h^{0}}^2 + 2g_{12}\scalarproductO{\nabla \ee_h^{1}}{\nabla \ee_h^{0}} + g_{22}\normO{\nabla \ee_h^{1}}^2\Big] \label{eq: intermediate last proof}\\
	&+\tau (\norm{\ee_h^0}_{\bH^1(\Omega)}^2+\norm{\ee_h^1}_{\bH^1(\Omega)}^2)+ \tau \sum_{j=2}^{n} \norm{\ee_h^j}_{\bH^1(\Omega)}^2 + \tau \sum_{j=2}^{n} \big(\normO{\nabla\sss_h^j}^2 +\normO{\boldsymbol{d}_h^j}^2\big).\nonumber
\end{align}
Let $\mu^-$, $\mu^+>0$ be the minimum and maximum eigenvalue of the symmetric matrix $G$. With $\mu^-|\xx|^2 \le G \xx\cdot\xx\le \mu^+|\xx|^2$ for every $\xx\in\R^2$, there holds
\begin{align*}
	g_{11}\normO{\nabla \ee_h^{n-1}}^2+2 g_{12}\scalarproductO{\nabla \ee_h^n}{\nabla \ee_h^{n-1}}+g_{22}\normO{\nabla \ee_h^n}^2 &\geq \mu^{-}\big(\normO{\nabla \ee_h^n}^2+\normO{\nabla \ee_h^{n-1}}^2\big)\geq \mu^-\normO{\nabla\ee_h^n}^2\\
\intertext{and}
	g_{11}\normO{\nabla \ee_h^0}^2+2 g_{12}\scalarproductO{\nabla \ee_h^0}{\nabla \ee_h^1}+g_{22}\normO{\nabla \ee_h^1}^2 &\leq \mu^{+}\big(\normO{\nabla \ee_h^0}^2+\normO{\nabla \ee_h^1}^2\big).
\end{align*}
Considering these relations, \eqref{eq: intermediate last proof} leads to
\begin{equation}\label{eq: error estimate}
	\begin{aligned}
		\tau \sum_{j=2}^{n} \normO{\oomega_h^{j-1}}^2 &+ \normO{\nabla\ee_h^n}^2\lesssim \normO{\nabla\ee_h^0}^2 + \normO{\nabla\ee_h^1}^2+\norm{\ee_h^0}_{\bH^1(\Omega)}^2+\norm{\ee_h^1}_{\bH^1(\Omega)}^2\\
		&\mkern+120mu+ \tau \sum_{j=2}^{n} \norm{\ee_h^j}_{\bH^1(\Omega)}^2 + \tau \sum_{j=2}^{n} \big(\normO{\nabla\sss_h^j}^2 +\normO{\boldsymbol{d}_h^j}^2\big)\\
		&\mkern-75mu\lesssim \norm{\ee_h^0}_{\bH^1(\Omega)}^2+\norm{\ee_h^1}_{\bH^1(\Omega)}^2+\tau \sum_{j=2}^{n} \norm{\ee_h^j}_{\bH^1(\Omega)}^2 +\tau \sum_{j=2}^{n} \big(\normO{\nabla\sss_h^j}^2 +\normO{\boldsymbol{d}_h^j}^2\big),
	\end{aligned}
\end{equation}
where the hidden constants depend only on $\gamma$, $\alpha$, $\lamex$, $M$, $|\Omega|$, $\mu^\pm$, $T$, and $h\le \bar{h}$. 

\textbf{Step~4 (Bound on $\boldsymbol{\norm{\ee_h^n}^2_{\bH^1(\Omega)}}$).}
From the left-hand side of \eqref{eq: error estimate}, we observe that it remains to relate the term $\normO{\ee_h^n}^2$ in terms of the quantity $\tau\sum_{j=2}^{n}\normO{\oomega_h^{j-1}}^2$, so as to recover the full $\bH^1(\Omega)$-norm of the error on the left-hand side.

Recall the identity 
\begin{equation}\label{eq: omega_h^n-1 rewritten}
	\oomega_h^{n-1} - \sss_h^n\stackrel{\eqref{eq: omega_h^n-1}}{=}\frac{1}{\tau}\Big(\frac{3}{2}\ee_h^n-2\ee_h^{n-1}+\frac{1}{2} \ee_h^{n-2}\Big)\eqcolon \frac{1}{\tau}\Big(\delta_0\ee_h^n+\delta_1\ee_h^{n-1}+\delta_2 \ee_h^{n-2}\Big)=\frac{1}{\tau}\sum_{j=0}^2 \delta_j \ee_h^{n-j}
\end{equation}
and consider the equality
\begin{equation}\label{eq: omega_h^n-1 rewritten 2}
\frac{1}{\tau} \sum_{j=2}^n \delta_{n-j} \ee_h^j=\oomega_h^{n-1}-\sss_h^n-\bg_h^n, \quad \text{ with } n \ge 2 \text{ and } \delta_{\ell}=0 \text{ for }\ell>2,
\end{equation}
where $\boldsymbol{g}_h^n:=\dfrac{1}{\tau} (\delta_{n} \ee_h^0 + \delta_{n-1}\ee_h^1)$ depends only on the starting errors and satisfies $\bg_h^n=0$ for $n \ge 4$. For $n=2$, the equality \eqref{eq: omega_h^n-1 rewritten 2} reads 
$$
\frac{1}{\tau} \delta_0 \ee_h^2=\oomega_h^{1}-\sss_h^2-\bg_h^2, \quad\text{ with }\quad \bg_h^2 = \frac{1}{\tau}(\delta_2 \ee_h^0 + \delta_1 \ee_h^1).
$$
For $n=3$, the equality \eqref{eq: omega_h^n-1 rewritten 2} reads
$$
\frac{1}{\tau} \big(\delta_1 \ee_h^2 + \delta_0 \ee_h^3\big)=\oomega_h^{2}-\sss_h^3-\bg_h^3, \quad\text{ with }\quad \bg_h^3 = \frac{1}{\tau} \delta_2 \ee_h^1,
$$
while the equality \eqref{eq: omega_h^n-1 rewritten 2} coincides with \eqref{eq: omega_h^n-1 rewritten} for $n\ge 4$.

To express $\ee_h^n$ in terms of $\oomega_h^{j-1}$, $\sss_h^j$ and $\bg_h^j$, it remains to invert the discrete convolutional equation \eqref{eq: omega_h^n-1 rewritten 2}. To this end, we multiply both sides by $\zeta^{n}$ and sum over $n\ge 2$.
Notice that since $\delta_{n-j}=0$ for $n-j>2$, we have
$$
\frac{1}{\tau}\sum_{n\ge 2}\sum_{j=2}^n \delta_{n-j} \ee_h^j \zeta^n = \frac{1}{\tau}\sum_{j \ge 2} \sum_{n\ge j} \delta_{n-j} \ee_h^j \zeta^{n} = \frac{1}{\tau}\sum_{j \ge 2} \sum_{k=0}^2 \ee_h^j\delta_{k} \zeta^{k+j} = \frac{1}{\tau}\delta(\zeta)\sum_{j \ge 2} \ee_h^j \zeta^j,
$$
where
$$\delta(\zeta):=\sum_{k=0}^2 \delta_k \zeta^k = \frac{3}{2}-2\zeta+\frac{1}{2}\zeta^2 = \frac{(\zeta-3)(\zeta-1)}{2}.$$
Thus, we obtain
\begin{equation}\label{eq: convolution in zeta}
\frac{1}{\tau}\delta(\zeta)\sum_{m \ge 2} \ee_h^m \zeta^m = \sum_{m\ge 2} \big(\oomega_h^{m-1}-\sss_h^m-\bg_h^m\big) \zeta^m.
\end{equation}
We now consider the power series of $1/\delta(\zeta)$, which can be expanded for $|\zeta|<1$ using the geometric series expansion as
$$
\kappa(\zeta)=\sum_{j=0}^{\infty} \kappa_j \zeta^j:=\frac{1}{\delta(\zeta)}=\frac{2}{(\zeta-3)(\zeta-1)} =  \frac{1}{1-\zeta}-\frac{1}{3}\cdot \frac{1}{1-\zeta/3} = \sum_{j=0}^\infty \Big(1-\frac{1}{3^{j+1}}\Big) \zeta^j.
$$
In particular, it holds that $0<\kappa_j = 1-\dfrac{1}{3^{j+1}}<1$ for all $j\ge 0$. Therefore, \eqref{eq: convolution in zeta} yields
$$
\sum_{m \ge 2} \ee_h^m \zeta^m = \tau \kappa(\zeta) \sum_{m\ge 2} \big(\oomega_h^{m-1}-\sss_h^m-\bg_h^m\big) \zeta^m = \tau \sum_{j=0}^\infty\sum_{m\ge 2}  \kappa_j \big(\oomega_h^{m-1}-\sss_h^m-\bg_h^m\big) \zeta^{m+j}.
$$
Extracting the $n$-th coefficient on both sides, noticing $m = n-j$ and $j\le n-2$, we finally obtain, for $n \ge 2$,
$$
\boldsymbol{e}_h^n=\tau \sum_{j=0}^{n-2}\kappa_j\left(\oomega_h^{n-1-j}-\boldsymbol{s}_h^{n-j}-\boldsymbol{g}_h^{n-j}\right)=\tau\sum_{j=2}^n \kappa_{n-j}\left(\oomega_h^{j-1}-\boldsymbol{s}_h^j-\boldsymbol{g}_h^j\right).
$$
Therefore, we can estimate the error by
\begin{equation}\label{eq: estimate e_h^n}
\begin{aligned}
\normO{\boldsymbol{e}_h^n}^2&\le 2\tau^2 \normO[\Big]{\sum_{j=2}^n \kappa_{n-j}(\oomega_h^{j-1}-\boldsymbol{s}_h^j)}^2+2\tau^2\normO[\Big]{\sum_{j=2}^3 \kappa_{n-j}\boldsymbol{g}_h^j}^2\\
&\le 2n\tau^2  \sum_{j=2}^n\normO{\oomega_h^{j-1} - \sss_h^j}^2 + 4\tau^2 \sum_{j=2}^3\normO{\boldsymbol{g}_h^j}^2\\
&\lesssim\tau \sum_{j=2}^n \normO{\oomega_h^{j-1}}^2 + \tau \sum_{j=2}^n \normO{\sss_h^j}^2 + \normO{\ee_h^0}^2 + \normO{\ee_h^1}^2,
\end{aligned}
\end{equation}
where the hidden constant depends only on $T$. Combining \eqref{eq: estimate e_h^n} with \eqref{eq: error estimate}, we derive
\begin{equation*}
	\begin{split}
		\norm{\ee_h^n}_{\bH^1(\Omega)}^2 &= \normO{\ee_h^n}^2+\normO{\nabla\ee_h^n}^2\stackrel{\eqref{eq: estimate e_h^n}}{\lesssim} \tau \sum_{j=2}^n \normO{\oomega_h^{j-1}}^2 + \tau \sum_{j=2}^n \normO{\sss_h^j}^2 + \normO{\ee_h^0}^2 + \normO{\ee_h^1}^2 + \normO{\nabla\ee_h^n}^2\\
		& \mkern-30mu\stackrel{\eqref{eq: error estimate}}{\lesssim} \tau \sum_{j=2}^{n} \norm{\ee_h^j}^2_{\bH^1(\Omega)} +\tau \sum_{j=2}^{n} \big(\norm{\sss_h^j}_{\bH^1(\Omega)}^2 + \normO{\boldsymbol{d}_h^j}^2\big)+(\norm{\ee_h^0}_{\bH^1(\Omega)}^2+\norm{\ee_h^1}_{\bH^1(\Omega)}^2),
	\end{split}
\end{equation*}
i.e., 
\begin{align*}
	\norm{\ee_h^n}_{\bH^1(\Omega)}^2\le  \newconst{const_final_proof}\Big[\tau \sum_{j=2}^{n} \norm{\ee_h^j}^2_{\bH^1(\Omega)} +\tau \sum_{j=2}^{n} \big(\norm{\sss_h^j}_{\bH^1(\Omega)}^2 + \normO{\boldsymbol{d}_h^j}^2\big)+(\norm{\ee_h^0}_{\bH^1(\Omega)}^2+\norm{\ee_h^1}_{\bH^1(\Omega)}^2)\Big],
\end{align*}
where $\const{const_final_proof}>0$ depends only on $\gamma$, $\alpha$, $\lamex$, $M$, $|\Omega|$, $\mu^\pm$, $T$, and $\bar{h}$.

\textbf{Step~5 (Choice of $\bar{\tau}$ and proof of (\ref{eq: stability full discretization}) conditional to assumption (\ref{eq: induction claim})).}
Let us now fix $\bar{\tau}>0$ such that $\bar{\tau}\const{const_final_proof}\leq1/2$. Then, we are led to
\begin{equation*}
	\norm{\ee_h^n}_{\bH^1(\Omega)}^2 {\le}2\const{const_final_proof}\Big[ \tau \sum_{j=2}^{n-1} \norm{\ee_h^j}^2_{\bH^1(\Omega)} +\tau \sum_{j=2}^{n} \big(\norm{\sss_h^j}_{\bH^1(\Omega)}^2 + \normO{\boldsymbol{d}_h^j}^2\big)+(\norm{\ee_h^0}_{\bH^1(\Omega)}^2+\norm{\ee_h^1}_{\bH^1(\Omega)}^2)\Big].
\end{equation*}
The discrete Gronwall inequality (see Lemma~\ref{lem: gronwall} in Appendix~\ref{section: appendix}) applied with
\begin{equation}\label{eq: def a_n}
a_n \coloneqq 2 \const{const_final_proof}\Big[ \tau \sum_{j=2}^{n} \big(\norm{\sss_h^j}_{\bH^1(\Omega)}^2 + \normO{\boldsymbol{d}_h^j}^2\big)+(\norm{\ee_h^0}_{\bH^1(\Omega)}^2+\norm{\ee_h^1}_{\bH^1(\Omega)}^2)\Big], 
\end{equation}
 $b_n \coloneqq 2\const{const_final_proof}\tau$ and $w_n = \norm{\ee_h^n}_{\bH^1(\Omega)}^2$ for $n\ge 2$ as well as $w_0\coloneqq 0 \eqqcolon w_1$
implies that 
	\begin{align*}
	\norm{\ee_h^n}_{\bH^1(\Omega)}^2 \le a_n \exp{(2\const{const_final_proof} n \tau)} \le a_n \exp{(2 \const{const_final_proof}\,T\,)} \quad\text{for every}\quad n=2,\ldots,N.
	\end{align*}
The constant $\const{const_final_proof}$ depends only on $\gamma$, $\alpha$, $\lamex$, $M$, $|\Omega|$, $\mu^\pm$, $T$, as well as on $\bar{h}$ and $\bar{\tau}$
and is therefore independent of $n$, $h$, and $\tau$. This proves \eqref{eq: stability full discretization} with $\const{C_stab}\coloneqq 2\const{const_final_proof}\exp{(2 \const{const_final_proof}\,T\,)}$ provided that \eqref{eq: induction claim} holds for all $j=0,\dots,n-1$.

\textbf{Step~6 (Proof of assumption (\ref{eq: induction claim}) by induction).}
We suppose that $h\leq \bar{h}$ and $\tau \leq \bar{\tau}$.
To show that the initial values for $n\in\{0,1\}$ satisfy \eqref{eq: induction claim}, notice that
\begin{equation}\label{eq: bound e0 H1}
\begin{aligned}
\norm{\ee_h^0}_{\bH^1(\Omega)} & \stackrel{\eqref{eq: e_h^n}}{\le} \norm{\mm_h^0-\mm_\star^0}_{\bH^1(\Omega)} + \norm{\mm_\star^0 - \mstarh^0}_{\bH^1(\Omega)}\\
&\mkern-7mu\stackrel{\eqref{eq: assumption m0},\eqref{eq: Ritz_error_1}}{\le} C_0 h + C_{\rm R} h \norm{\mm_\star^0}_{\bH^2(\Omega)}\stackrel{\eqref{eq: regularity last prop}}{\le} (C_0+C_{\rm R}M)h,
\end{aligned}
\end{equation}
which implies
$$
\norm{\ee_h^0}_{\bL^{\infty}(\Omega)}\stackrel{\eqref{eq: sobolev}}{\le}C_{\infty}h^{-1/2}\norm{\ee_h^0}_{\bH^1(\Omega)} \stackrel{\eqref{eq: bound e0 H1}}{\le} C_{\infty} (C_0+C_{\rm R}M)\bar{h}^{1/2}\stackrel{\eqref{eq: bar h 1}}{\le} \rho.
$$
Analogously, the error after the first step satisfies
\begin{equation}\label{eq: bound e1 H1}
	\norm{\ee_h^1}_{\bH^1(\Omega)}\stackrel{\eqref{eq: assumption m1}}{\le} \const{feischl_tran}(h+\tau^2)\,{\le} \,\const{feischl_tran}(h+C_{\rm CFL}^2 h^{(1+\varepsilon)/2}),
\end{equation}
which implies
\begin{equation*}
\norm{\ee_h^1}_{\bL^\infty(\Omega)}\stackrel{\eqref{eq: sobolev}}{\le}C_{\infty}h^{-1/2}\norm{\ee_h^1}_{\bH^1(\Omega)} \le C_{\infty}\const{feischl_tran} (\bar{h}^{1/2}+C_{\rm CFL}^2\bar{h}^{\varepsilon/2})\stackrel{\eqref{eq: bar h 2}}{\le} \rho.
\end{equation*}
According to the induction principle, we may thus assume \eqref{eq: induction claim} for $j=0,\dots,n-1$ and must show \eqref{eq: induction claim} for $j=n$. 
To this end, recall $\normO{\dd_h^j}\le\const{consistency} (h+\tau^2)$ from Proposition~\ref{prop: consistency error}. This guarantees 
$$
\tau\sum_{j=2}^n\normO{\dd_h^j}^2\stackrel{\eqref{eq: bound d_h^n}}{\le}\const{consistency}^2\tau(n-1)(h+\tau^2)^2\le\const{consistency}^2\tau(n-1)(h+\tau^2)^2\stackrel{\eqref{eq: CFL}}{<} \const{consistency}^2T\big[h+C_{\rm CFL}^2 h^{(1+\varepsilon)/2}\big]^2.
$$ 
Similarly, $\normO{\sss_h^j}\le \const{s} (h+\tau^2)$ from Proposition~\ref{prop: bound s_h^n} yields that
$$
\tau\sum_{j=2}^n\norm{\sss_h^j}^2_{\bH^1(\Omega)}\stackrel{\eqref{eq: bound s_h^n}}{\le}\const{s}^2\tau(n-1)(h+\tau^2)^2\le\const{s}^2\tau (n-1)(h+\tau^2)^2\stackrel{\eqref{eq: CFL}}{<}\const{s}^2T\big[h+C_{\rm CFL}^2 h^{(1+\varepsilon)/2}\big]^2. 
$$
Plugging these last two estimates with \eqref{eq: bound e0 H1}--\eqref{eq: bound e1 H1} into \eqref{eq: def a_n}, we see that 
\begin{equation}\label{eq: bound a_n}
	\begin{aligned}
		a_{n}& \le 2 C_{19}\Big[(\const{s}^2T + \const{consistency}^2T + C_1^2 )(h+C_{\rm CFL}^2 h^{(1+\varepsilon)/2})^2 + (C_0+C_{\rm R}M)^2 h^2\Big]\\
&\le 2 C_{19}\Big[T(\const{s}^2 + \const{consistency}^2) + C_1^2 + (C_0+C_{\rm R}M)^2\Big] h g(\bar{h})^2,
	\end{aligned}
\end{equation}
where $g({h}) \coloneqq {h}^{1/2} + C_{\rm CFL}^2 {h}^{\varepsilon/2}$.
This implies 
\begin{equation*}
	\begin{aligned}
		\norm{\ee_h^n}_{\bL^{\infty}(\Omega)}^2 &\stackrel{\eqref{eq: sobolev}}{\le} C_{\infty}^2 h^{-1} \norm{\ee_h^n}_{\bH^1(\Omega)}^2 \\
		&\stackrel{\eqref{eq: bound a_n}}{\le} 2C_{\infty}^2 C_{19}\Big[T(\const{s}^2 + \const{consistency}^2) + C_1^2 + (C_0+C_{\rm R}M)^2\Big] \exp{(2 \const{const_final_proof}\,T\,)}g(\bar{h})^2 .
	\end{aligned}
\end{equation*}
By further reducing $\bar{h}$, we thus guarantee $\norm{\ee_h^n}_{\bL^{\infty}(\Omega)} \le \rho$ and prove \eqref{eq: induction claim} for $j=n$. This completes the proof of the proposition.
\end{proof}
\subsection{Proof of Theorem~\ref{thm: main result}.}
We can finally combine the previous propositions to prove our main result: Stability \eqref{eq: stability full discretization} from Proposition~\ref{prop: stability full discretization} yields 
\begin{align*}
	\norm{\mm_h^n-\mstarh^n}^2_{\bH^1(\Omega)}\stackrel{\eqref{eq: e_h^n}}{=}\norm{\ee_h^n}_{\bH^1(\Omega)}^2 &\lesssim h^2+\tau^4.
\end{align*}
Since $	\mm_\star^n-\mm_h^n = \mm_\star^n - \mstarh^n - (\mm_h^n - \mstarh^n) = (\mathbf{I}-\bR_h)\mm_\star^n-\ee_h^n,$
it follows by triangle inequality and the bound $\norm{(\mathbf{I}-\bR_h)\mm_\star^n}_{\bH^1(\Omega)}\le C_{\rm R} h$ that 
\begin{equation}\label{eq: almost final}
\norm{\mm(t_n)-\mm_h^n}_{\bH^1(\Omega)} \le C_{\rm conv}(h + \tau^2),
\end{equation}
where the constant $C_{\rm conv}>0$ depends only on $\gamma$, $\alpha$, $\lamex$, $M$, $|\Omega|$, $\mu^\pm$, $T$, $C_0$, $C_{\ff}$, $C_{\rm CFL}$ as well as on $\bar{h}$ and $\bar{\tau}$.\qed

\subsection{Proof of Corollary~\ref{cor: optimal order error estimate interpoland}.}
To prove \eqref{eq: optimal order error interpoland}, we define the piecewise linear time-interpoland of the exact solution 
$$
\boldsymbol{\mathcal{I}}_\tau\mm(t) \coloneqq \frac{t-t_n}{\tau}\mm_\star^{n+1}+\frac{t_{n+1}-t}{\tau}\mm_\star^n \quad \text{for every } t\in[t_n, t_{n+1}].
$$ 
Thanks to the triangular inequality, it holds that
\begin{align*}
	\norm{\mm(t)-\mm_{h\tau}(t)}_{\bH^1(\Omega)} &\le \norm{\mm(t)-\boldsymbol{\mathcal{I}}_\tau\mm(t)}_{\bH^1(\Omega)} + \norm{\boldsymbol{\mathcal{I}}_\tau\mm(t)-\mm_{h\tau}(t)}_{\bH^1(\Omega)} = I_1 + I_2.
\end{align*}
Thanks to a Taylor expansion of $\mm$ around $t\in[t_n, t_{n+1}]$, we have
\begin{align*}
	\mm_\star^n &= \mm(t) + (t_n-t)\partial_t\mm(t) + \int_{t}^{t_n} (t_n-s)\partial_{tt}\mm(s) \, \d s
\intertext{and}
	\mm_\star^{n+1} &= \mm(t) + (t_{n+1}-t)\partial_t\mm(t) + \int_{t}^{t_{n+1}} (t_{n+1}-s)\partial_{tt}\mm(s) \, \d s.
\end{align*}
Thanks to straightforward algebraic manipulations, this implies that
\begin{equation*}
\begin{aligned}
	I_1 &= \norm[\Big]{\frac{t-t_n}{\tau}\int_{t}^{t_{n+1}} (t_{n+1}-s)\partial_{tt}\mm(s) \, \d s + \frac{t_{n+1}-t}{\tau}\int_{t}^{t_n} (t_n-s)\partial_{tt}\mm(s) \, \d s}_{\bH^1(\Omega)} \\
	&\le \Big|\frac{t-t_n}{\tau}\int_t^{t_{n+1}} (t_{n+1}-s)\, \d s + \frac{t_{n+1}-t}{\tau}\int_{t_n}^{t} (s-t_n) \, \d s\Big|\norm{\partial_{tt}\mm}_{L^{\infty}(0,T; \bH^1(\Omega))}\\
	& = \frac{(t_{n+1}-t)(t-t_n)}{2\tau}(t_{n+1}-t_n)\norm{\partial_{tt}\mm}_{L^{\infty}(0,T; \bH^1(\Omega))} \le \frac{\tau^2}{8}\norm{\partial_{tt}\mm}_{L^{\infty}(0,T; \bH^1(\Omega))}.
\end{aligned}
\end{equation*}
Moreover, it holds that 
\begin{equation*}
	\begin{aligned}
		I_2 & = \norm[\Big]{\frac{t-t_n}{\tau}(\mm_\star^{n+1}-\mm_h^{n+1})+\frac{t_{n+1}-t}{\tau}(\mm_\star^n-\mm_h^n)}_{\bH^1(\Omega)} \\
		&\stackrel{\eqref{eq: almost final}}{\le}C_{\rm conv}\Big[\frac{t-t_n}{\tau}+\frac{t_{n+1}-t}{\tau} \Big](h+\tau^2) = C_{\rm conv}(h+\tau^2).
	\end{aligned}
\end{equation*}
Altogether, we have shown that
\begin{align*}
	\norm{\mm-\mm_{h\tau}}_{L^{\infty}(0,T; \bH^1(\Omega))} & = \sup_{t\in[0,T]} \norm{\mm(t)-\mm_{h\tau}(t)}_{\bH^1(\Omega)} \le \frac{\tau^2}{8}M + C_{\rm conv}(h+\tau^2),
\end{align*} which proves \eqref{eq: optimal order error estimate} with $C_{\rm conv}' \coloneqq C_{\rm conv}+ M/8$ and concludes the proof of Corollary~\ref{cor: optimal order error estimate interpoland}. \qed
\section{Numerical experiments}\label{section:numerics}
\input{llg_bdf2_a_priori_estimates_numerics.tex}

\appendix
\section{}\label{section: appendix}
\begin{proof}[{\bfseries Proof of Lemma~\ref{lem: normalization bounds}}]
	The proof is split into two steps. 

	\textbf{Step~1:}
	To prove \eqref{eq: normalization bounds 0}, note that
	$$
		\norm{\bN(\boldsymbol{u})}_{\bL^\infty(\Omega)} = \norm[\Big]{\frac{\boldsymbol{u}}{|\boldsymbol{u}|}}_{\bL^\infty(\Omega)} \le \frac{1}{c} \norm{\boldsymbol{u}}_{\bL^\infty(\Omega)}.
	$$
	Moreover, recall that $\partial_j |\boldsymbol{u}| = \dfrac{\partial_j \boldsymbol{u}\cdot \boldsymbol{u}}{|\boldsymbol{u}|},$
	and therefore
	\begin{equation}\label{eq:lem_4_I}
		\partial_j\bN(\boldsymbol{u})=\partial_j\Big(\frac{\boldsymbol{u}}{|\boldsymbol{u}|}\Big) = \frac{\partial_j \boldsymbol{u} |\boldsymbol{u}| - \boldsymbol{u} \partial_j |\boldsymbol{u}|}{|\boldsymbol{u}|^2} = \frac{\partial_j \boldsymbol{u}}{|\boldsymbol{u}|} - \frac{\boldsymbol{u} \,(\partial_j \boldsymbol{u} \cdot \boldsymbol{u})}{|\boldsymbol{u}|^3}.
	\end{equation}
	This proves \eqref{eq: normalization bounds i}, since
	\begin{align*}
		\left|\partial_j\Big(\frac{\boldsymbol{u}}{|\boldsymbol{u}|}\Big)\right| &\le \left|\frac{\partial_j \boldsymbol{u}}{|\boldsymbol{u}|} \right|+ \left|\frac{\boldsymbol{u}(\partial_j \boldsymbol{u}\cdot \boldsymbol{u})}{|\boldsymbol{u}|^3} \right| \leq  \frac{2}{c} |\partial_j \boldsymbol{u}|.
	\end{align*}

	To prove \eqref{eq: normalization bounds ii}, notice that
	\begin{align*}
		\partial_i\partial_j\Big(\frac{\boldsymbol{u}}{|\boldsymbol{u}|}\Big) & \stackrel{\eqref{eq:lem_4_I}}{=} \frac{\partial_i\partial_j \boldsymbol{u}}{|\boldsymbol{u}|} -\frac{\partial_j \boldsymbol{u} \cdot \partial_i |\boldsymbol{u}|}{|\boldsymbol{u}|^2}-\frac{\partial_i(\boldsymbol{u}(\partial_j\boldsymbol{u} \cdot \boldsymbol{u}))}{|\boldsymbol{u}|^3} - \boldsymbol{u}(\partial_j \boldsymbol{u}\cdot \boldsymbol{u})\partial_i\frac{1}{|\boldsymbol{u}|^3}\\		
		&\mkern-50mu=\frac{\partial_i\partial_j \boldsymbol{u}}{|\boldsymbol{u}|} -\frac{\partial_j \boldsymbol{u}(\partial_i \boldsymbol{u} \cdot \boldsymbol{u})}{|\boldsymbol{u}|^3}- \frac{\partial_i \boldsymbol{u}\,(\partial_j \boldsymbol{u}\cdot \boldsymbol{u})}{|\boldsymbol{u}|^3}-\boldsymbol{u}\frac{\partial_i\partial_j \boldsymbol{u} \cdot \boldsymbol{u}}{|\boldsymbol{u}|^3}- \boldsymbol{u} \frac{\partial_j \boldsymbol{u} \cdot \partial_i \boldsymbol{u}}{|\boldsymbol{u}|^3} + 3 \boldsymbol{u} \frac{(\partial_j \boldsymbol{u} \cdot \boldsymbol{u})(\partial_i \boldsymbol{u} \cdot \boldsymbol{u})}{|\boldsymbol{u}|^5}.
	\end{align*}
	This implies that
	\begin{align*}
		\Big|\partial_i\partial_j\Big(\frac{\boldsymbol{u}}{|\boldsymbol{u}|}\Big)\Big|&\le \Big|\frac{\partial_i\partial_j \boldsymbol{u}}{|\boldsymbol{u}|}\Big| + 7 \frac{|\partial_i \boldsymbol{u}||\partial_j \boldsymbol{u}|}{|\boldsymbol{u}|^2}\le \frac{1}{c} |\partial_i\partial_j \boldsymbol{u}| + \frac{7}{c^2} |\partial_i \boldsymbol{u}||\partial_j \boldsymbol{u}|.
	\end{align*}

	To prove \eqref{eq: normalization bounds iii}, first notice that $\partial_i\partial_j \boldsymbol{u}_h = 0$ on every $K\in \mathcal{T}_h$. Therefore, that the last estimate reduces to 
	\begin{align*}
		\partial_i\partial_j\Big(\frac{\boldsymbol{u}_h}{|\boldsymbol{u}_h|}\Big) \le \frac{7}{c^2}|\partial_i \boldsymbol{u}_h||\partial_j \boldsymbol{u}_h|.
	\end{align*}
	
	\textbf{Step~2:} To prove \eqref{eq: normalization bounds 2}, the triangular inequality yields that
	\begin{equation}\label{eq: in appendix}
	\left|\frac{\boldsymbol{u}}{|\boldsymbol{u}|}-\frac{\widetilde{\boldsymbol{u}}}{|\widetilde{\boldsymbol{u}}|}\right|=\left|\frac{\boldsymbol{u}(|\widetilde{\boldsymbol{u}}|-|\boldsymbol{u}|)+|\boldsymbol{u}|(\boldsymbol{u}-\widetilde{\boldsymbol{u}})}{|\boldsymbol{u}||\widetilde{\boldsymbol{u}}|}\right| \leq 2 \min \left\{|\boldsymbol{u}|^{-1},|\widetilde{\boldsymbol{u}}|^{-1}\right\}|\boldsymbol{u}-\widetilde{\boldsymbol{u}}|\le \frac{2}{c} |\boldsymbol{u}-\widetilde{\boldsymbol{u}}|.
	\end{equation}
	Integrating \eqref{eq: in appendix} over $\Omega$, we prove \eqref{eq: normalization bounds 2} for $k=0$.
	If $k=1$, we use \eqref{eq:lem_4_I} to deduce:
	\begin{align*}
		\left|\partial_j\Big(\frac{\boldsymbol{u}}{|\boldsymbol{u}|}\Big)-\partial_j\Big(\frac{\widetilde{\boldsymbol{u}}}{|\widetilde{\boldsymbol{u}}|}\Big)\right| & \stackrel{\eqref{eq:lem_4_I}}{=} \left|\frac{\partial_j \boldsymbol{u}}{|\boldsymbol{u}|} - \frac{\partial_j \widetilde{\boldsymbol{u}}}{|\widetilde{\boldsymbol{u}}|}\right| + \left|\frac{\boldsymbol{u} (\partial_j \boldsymbol{u} \cdot \boldsymbol{u})}{|\boldsymbol{u}|^3} - \frac{\widetilde{\boldsymbol{u}} (\partial_j \widetilde{\boldsymbol{u}} \cdot \widetilde{\boldsymbol{u}})}{|\widetilde{\boldsymbol{u}}|^3}\right|\eqcolon I_1 + I_2.
	\end{align*}
	It holds that 
	\begin{align*}
		I_1 & = \left|\frac{(\partial_j \boldsymbol{u}) | \widetilde{\boldsymbol{u}}| - (\partial_j \widetilde{\boldsymbol{u}} )|\boldsymbol{u}|}{|\boldsymbol{u}||\widetilde{\boldsymbol{u}}|}\right| \le \frac{|\partial_j \boldsymbol{u} - \partial_j \widetilde{\boldsymbol{u}}| |\widetilde{\boldsymbol{u}}| + |\partial_j \widetilde{\boldsymbol{u}}| \big||\widetilde{\boldsymbol{u}}| - |\boldsymbol{u}|\big|}{|\boldsymbol{u}||\widetilde{\boldsymbol{u}}|}\le \frac{1}{c} |\partial_j \boldsymbol{u} - \partial_j \widetilde{\boldsymbol{u}}| + \frac{1}{c^2} |\partial_j\widetilde{\boldsymbol{u}}||\boldsymbol{u} - \widetilde{\boldsymbol{u}}|
	\end{align*}
	and 
	\begin{align*}
		I_2 & \le \left| \frac{\boldsymbol{u}(\partial_j\boldsymbol{u}\cdot \boldsymbol{u})-\widetilde{\boldsymbol{u}}(\partial_j\widetilde{\boldsymbol{u}}\cdot\widetilde{\boldsymbol{u}})}{|\boldsymbol{u}|^3}\right| + \left|\widetilde{\boldsymbol{u}}(\partial_j\widetilde{\boldsymbol{u}} \cdot \widetilde{\boldsymbol{u}})\Big(\frac{1}{|\boldsymbol{u}|^3} - \frac{1}{|\widetilde{\boldsymbol{u}}|^3}\Big) \right|\eqqcolon I_{2,1} + I_{2,2}.
	\end{align*}
	We can estimate $I_{2,1}$ as
	\begin{align*}
		I_{2,1} &=  \left| \frac{(\boldsymbol{u}-\widetilde{\boldsymbol{u}})( \partial_j \boldsymbol{u} \cdot \boldsymbol{u}) + \widetilde{\boldsymbol{u}}((\boldsymbol{u}-\widetilde{\boldsymbol{u}}) \cdot \partial_j \boldsymbol{u}) + \widetilde{\boldsymbol{u}}(\widetilde{\boldsymbol{u}}\cdot (\partial_j \boldsymbol{u} -\partial_j \widetilde{\boldsymbol{u}}))}{|\boldsymbol{u}|^3} \right|\\
		& \le \frac{|\boldsymbol{u} - \widetilde{\boldsymbol{u}}| |\partial_j \boldsymbol{u} \cdot \boldsymbol{u}|}{|\boldsymbol{u}|^3} + \frac{|\widetilde{\boldsymbol{u}}| |\boldsymbol{u} - \widetilde{\boldsymbol{u}}| |\partial_j \boldsymbol{u}|}{|\boldsymbol{u}|^3} + \frac{|\widetilde{\boldsymbol{u}}| |\partial_j \boldsymbol{u} - \partial_j \widetilde{\boldsymbol{u}}| |\widetilde{\boldsymbol{u}}|}{|\boldsymbol{u}|^3} \\
		& \le \frac{2}{c^2} |\partial_j \boldsymbol{u}| |\boldsymbol{u} - \widetilde{\boldsymbol{u}}| + \frac{1}{c^2} |\partial_j \boldsymbol u - \partial_j \widetilde{\boldsymbol u}|.
	\end{align*}
	To estimate $I_{2,2}$ note that the mean value theorem applied to the function $f(t) = 1/t^3$ provides some $\xi \in 
	\mathbb{R}$ such that $\min\{|\boldsymbol{u}|, |\widetilde{\boldsymbol{u}}|\} \le \xi \le \max\{|\boldsymbol{u}|, |\widetilde{\boldsymbol{u}}|\}$, with
	$$
	\bigg|\frac{1}{|\boldsymbol{u}|^3} - \frac{1}{|\widetilde{\boldsymbol{u}}|^3}\bigg| = \frac{3}{\xi^4} \big||\widetilde{\boldsymbol{u}}| - |\boldsymbol{u}|\big|\le \frac{3}{\xi^4} |\boldsymbol{u} - \widetilde{\boldsymbol{u}}|.
	$$
	Since $c\le|\boldsymbol{u}|, |\widetilde{\boldsymbol{u}}|\le c^{-1}$, we have $\xi\ge c$ and therefore
	\begin{align*}
		I_{2,2}\le |\widetilde{\boldsymbol{u}}|^2|\partial_j\widetilde{\boldsymbol{u}}|\left|\frac{1}{|\boldsymbol{u}|^3} - \frac{1}{|\widetilde{\boldsymbol{u}}|^3}\right|\le \frac{3}{c^6}|\partial_j\widetilde{\boldsymbol{u}}| |\boldsymbol{u}-\widetilde{\boldsymbol{u}}|.
	\end{align*}
	Putting all the estimates together, we prove
	\begin{equation}\label{eq: in appendix 2}
		\left|\partial_j\Big(\frac{\boldsymbol{u}}{|\boldsymbol{u}|}\Big)-\partial_j\Big(\frac{\tilde{\boldsymbol{u}}}{|\tilde{\boldsymbol{u}}|}\Big)\right| \lesssim |\partial_j \boldsymbol{u} - \partial_j \tilde{\boldsymbol{u}}| + (|\partial_j \boldsymbol{u}| + |\partial_j \tilde{\boldsymbol{u}}|) |\boldsymbol{u} - \tilde{\boldsymbol{u}}|,
	\end{equation}
	where the hidden constant depends only on $c$. Considering \eqref{eq: in appendix}--\eqref{eq: in appendix 2} and integrating over $\Omega$, we conclude the proof of \eqref{eq: normalization bounds 2} for $k=1$.
\end{proof}

\begin{proof}[{\bfseries Proof of Lemma~\ref{lem: projection}}]
	To prove \eqref{eq: stability projection}, note that 
	$
	\bP(\uu)\vv = \vv - (\bN(\uu)\cdot\vv)\bN(\uu),
	$
	which implies 
	\begin{align*}
		\norm{\bP(\uu)\vv}_{\bW^{k,p}(\Omega)}\lesssim\norm{\vv}_{\bW^{k,p}(\Omega)}(1+\norm{\bN(\uu)}^2_{\bW^{k,\infty}(\Omega)}).
	\end{align*}
	Using the estimates from Lemma~\ref{lem: normalization bounds} as well as $1\lesssim c^2 \le \norm{\uu}_{\bL^\infty}^2$ by assumption, we obtain
	\begin{align*}
		\norm{\bP(\uu)\vv}_{\bW^{k,p}(\Omega)} &\stackrel{\eqref{eq: normalization bounds}}{\lesssim} \norm{\vv}_{\bW^{k,p}(\Omega)}\big(1+\norm{\uu}_{\bW^{k,\infty}(\Omega)}^2\big)\lesssim\norm{\vv}_{\bW^{k,p}(\Omega)}\norm{\uu}_{\bW^{k,\infty}(\Omega)}^2,
	\end{align*}
	where the hidden constants depend only on $c$. This proves \eqref{eq: stability projection}.
	To prove \eqref{eq: projection}, let
	\begin{equation}\label{eq: error vector}
	\boldsymbol{e}\coloneqq\bN(\uu)-\bN(\widetilde{\uu})=\dfrac{\uu}{|\uu|}-\dfrac{\widetilde{\uu}}{|\widetilde{\uu}|}
	\end{equation}
	and note that
	$$
	(\bP(\uu)-\bP(\widetilde{\uu}))\vv=-\frac{(\uu\uu^T)}{|\uu|^2}\vv+\frac{(\widetilde{\uu}\widetilde{\uu}^T)}{|\widetilde{\uu}|^2}\vv\stackrel{\eqref{eq: error vector}}{=}
	-\big(\bN(\uu)\ee^T+\ee\,\bN(\widetilde{\uu})^T\big)\vv.
	$$
	This leads to
	\begin{align*}
	&\normO{(\bP(\uu)-\bP(\widetilde{\uu}))\vv} \le \norm{\vv}_{\bL^\infty(\Omega)}\normO[\big]{\bN(\uu)\ee^T+\ee\,\bN(\widetilde{\uu})^T}\\
	&\le \norm{\vv}_{\bL^\infty(\Omega)}\normO{\ee}\big(\norm{\bN(\widetilde{\uu})}_{\bL^\infty(\Omega)}+\norm{\bN(\uu)}_{\bL^\infty(\Omega)}\big)\stackrel{\eqref{eq: error vector}}{=} 2 \norm{\vv}_{\bL^\infty(\Omega)}\normO{\bN(\uu)-\bN(\widetilde{\uu})}.
	\end{align*}
	Applying \eqref{eq: normalization bounds 2} from Lemma~\ref{lem: normalization bounds} concludes the proof of \eqref{eq: projection} for $k=0$. To prove \eqref{eq: projection L1}, we can apply the Hölder inequality and proceed similarly as
	\begin{align*}
	&\norm{(\bP(\uu)-\bP(\widetilde{\uu}))\vv}_{\bL^1(\Omega)} \le \normO{\vv}\normO[\big]{\bN(\uu)\ee^T+\ee\,\bN(\widetilde{\uu})^T}\\
	&\le \normO{\vv}\normO{\ee}\big(\norm{\bN(\widetilde{\uu})}_{\bL^\infty(\Omega)}+\norm{\bN(\uu)}_{\bL^\infty(\Omega)}\big)\stackrel{\eqref{eq: error vector}}{=} 2 \normO{\vv}\normO{\bN(\uu)-\bN(\widetilde{\uu})}.
	\end{align*} Again, Lemma~\ref{lem: normalization bounds} concludes the proof of \eqref{eq: projection L1}.
	The proof of \eqref{eq: projection} for $k=1$ follows by noticing that
	\begin{align*}
		\partial_j\big[(\bP(\uu)-\bP(\widetilde{\uu}))\vv\big] &= -\partial_j\big[(\bN(\uu)\ee^T+\ee\,\bN(\widetilde{\uu})^T)\vv\big]\\
			&\mkern-100mu= -\big[\partial_j\bN(\uu)\,\ee^T + \bN(\uu)\,\partial_j\ee^T + (\partial_j\ee) \,\bN(\widetilde{\uu})^T + \ee \,\partial_j\bN(\widetilde{\uu})^T{\big]}\vv \\
			&- \big[\bN(\uu)\ee^T+\ee\,\bN(\widetilde{\uu})^T\big]\partial_j\vv.
	\end{align*}
	Thanks to \eqref{eq: normalization bounds} from Lemma~\ref{lem: normalization bounds} and the present assumptions, we have
	$$
		\norm{\nabla\bN(\uu)}_{\bL^\infty(\Omega)}\lesssim \norm{\nabla\uu}_{\bL^\infty(\Omega)}\le M \quad \text{and} \quad \norm{\nabla\bN(\widetilde{\uu})}_{\bL^\infty(\Omega)}\lesssim \norm{\nabla\widetilde{\uu}}_{\bL^\infty(\Omega)}\le M.
	$$
	This yields
	\begin{align*}
		&\normO{\nabla\big((\bP(\uu)-\bP(\widetilde{\uu}))\vv\big)} \lesssim \big(\normO{\ee}+\normO{\nabla\ee}\big)\norm{\vv}_{\bL^\infty(\Omega)}+ \normO{\ee}\norm{\nabla\vv}_{\bL^\infty(\Omega)}\\
		& \stackrel{\eqref{eq: error vector}}{=} \big(\normO{\bN(\uu)-\bN(\widetilde{\uu})} + \normO{\nabla\bN(\uu)-\nabla\bN(\widetilde{\uu})}\big)\norm{\vv}_{\bL^\infty(\Omega)}+\normO{\bN(\uu)-\bN(\widetilde{\uu})}\norm{\nabla\vv}_{\bL^\infty(\Omega)},
	\end{align*}
	where the hidden constant depends only on $M$. Applying \eqref{eq: normalization bounds 2} from Lemma~\ref{lem: normalization bounds} concludes the proof of \eqref{eq: projection} for $k=1$ and therefore the proof of the lemma.
\end{proof}

\begin{proof}[{\bfseries Proof of Lemma~\ref{lem: stability I_h}}]
	For $k=0$ and $K\in \mathcal{T}_h$, we define the set
	$$
		\mathcal{K}=\left\{(r_h,q_h)\in (\boldsymbol{\mathcal{P}}_r(K))^2\colon \norm{r_h}_{\bL^p(K)}=1, c \le |q_h|\le c^{-1}\right\}.
	$$
	Being closed and bounded on a finite dimensional space, $\mathcal{K}$ is compact. Therefore, the continuous function
	\begin{align*}
		F\colon \mathcal{K} \to \R, \qquad
		(r_h,q_h) \mapsto \frac{\norm{\III_h(r_h/|q_h|)}_{\bL^p(K)}}{\norm{{r_h}/{|q_h|}}_{\bL^p(K)}}
	\end{align*}
	attains its maximum on $\mathcal{K}$. This implies that there exists a constant $C>0$ such that
	\begin{equation}\label{eq: stability I_h L^p}
		\norm[\Big]{\III_h\Big(\frac{r_h}{|q_h|}\Big)}_{\bL^p(K)} \le C \norm[\Big]{\frac{r_h}{|q_h|}}_{\bL^p(K)}, \quad \text{for all} \quad (r_h,q_h)\in \mathcal{K}.
	\end{equation}  
	Due to the linearity of $\III_h$, the constraint $\norm{r_h}_{\bL^p(K)}=1$ does not affect the value of $F$, since it holds that $F(\lambda r_h, q_h)=F(r_h, q_h)$ for every $\lambda\neq 0$. This proves the $\bL^p$-stability.
	
	For $k=1$, let $\alpha_K\coloneqq 1/|K|\int_K r_h/|q_h|\d x$. The inverse estimate \eqref{eq: inverse estimate L^p} proves that
	\begin{align*}
		\norm[\Big]{\nabla \III_h\Big(\frac{r_h}{|q_h|}\Big)}_{\bL^p(K)} & = \norm[\Big]{\nabla \III_h\Big(\frac{r_h}{|q_h|}-\alpha_K\Big)}_{\bL^p(K)}  \lesssim h_K^{-1}\norm[\Big]{\III_h\Big(\frac{r_h}{|q_h|}-\alpha_K\Big)}_{\bL^p(K)}\\
		& \stackrel{\eqref{eq: stability I_h L^p}}{\lesssim} h_K^{-1}\norm[\Big]{\frac{r_h}{|q_h|}-\alpha_K}_{\bL^p(K)} \lesssim \norm[\Big]{\nabla\frac{r_h}{|q_h|}}_{\bL^p(K)}.
	\end{align*}
	This proves the $\bW^{1,p}$-stability result and hence concludes the proof.
\end{proof}

\begin{proof}[{\bfseries Proof of Lemma~\ref{lem: semi-discrete approximation}}]
	Taylor expansion of $\mm\in C^2([0,T\,],\bL^2(\Omega))$ around $t_{n-2}$ shows
	\begin{align*}
		\mm_\star^{n-1}&=\mm_\star^{n-2}+\tau\partial_t\mm_\star^{n-2}+\int_{t_{n-2}}^{t_{n-1}}(t_{n-1}-s)\partial_s^2\mm(s)\d s,\\
		\mm_\star^n&=\mm_\star^{n-2}+2\tau\partial_t\mm_\star^{n-2}+\int_{t_{n-2}}^{t_{n}}(t_{n}-s)\partial_s^2\mm(s)\d s.
	\end{align*}
	Therefore, we get
	\begin{align*}
		\mm_\star^n-\widehat{\mm}_\star^{n}&\stackrel{\eqref{eq: Ritz projections}}{=} \mm_\star^n-2\mm_\star^{n-1}+\mm_\star^{n-2}\\
		&\mkern+5mu= \int_{t_{n-2}}^{t_{n}}(t_{n}-s)\partial_s^2\mm(s)\d s-2\int_{t_{n-2}}^{t_{n-1}}(t_{n-1}-s)\partial_s^2\mm(s)\d s.
	\end{align*}
	This yields 
	\begin{equation}\label{eq: extrapolation error}
		\normO{\widehat{\mm}_\star^n-\mm_\star^n} \lesssim \tau^2
	\end{equation}
	where the hidden constant depends only on $\max\limits_{t\in[0,T\,]}\normO{\partial_t^2\mm(t)}=\norm{\partial_t^2\mm}_{L^\infty(0,T;\bL^2(\Omega))}\le M <\infty$. This proves \eqref{eq: consistency semi predictor}.

	To prove \eqref{eq: consistency semi derivative}, note that $\bP(\mm_\star^n)\partial_t\mm_\star^n=\partial_t\mm_\star^n\in\bT(\mm_\star^n)$. Therefore, the time derivative error reads
	\begin{equation}\label{eq: in appendix 3}
		\begin{aligned}
			&\vv_\star^{n-1}-\partial_t\mm_\star^n\stackrel{\eqref{eq: v_star}}{=} \bP(\widehat{\mm}_\star^{n})\frac{1}{\tau}\Big(\frac{3}{2}\mm_\star^n-2\mm_\star^{n-1}+\frac{1}{2}\mm_\star^{n-2}\Big)-\bP(\mm_\star^n)\partial_t\mm_\star^n\\
			&= \bP(\widehat{\mm}_\star^{n})\bigg[\frac{1}{\tau}\Big(\frac{3}{2}\mm_\star^n-2\mm_\star^{n-1}+\frac{1}{2}\mm_\star^{n-2}\Big)-\partial_t\mm_\star^n\bigg]+\big[\bP(\widehat{\mm}_\star^{n})-\bP(\mm_\star^n)\big]\partial_t\mm_\star^n.
		\end{aligned}
	\end{equation}
	Since $\mm\in C^1([0,T\,],\bL^\infty(\Omega))$, the last term in \eqref{eq: in appendix 3} can be bounded thanks to Lemma \ref{lem: projection} as
	\begin{equation}\label{eq: in appendix 4}
		\normO{\big(\bP(\widehat{\mm}_\star^{n})-\bP(\mm_\star^n)\big)\partial_t\mm_\star^n} \stackrel{\eqref{eq: projection}}{\lesssim}\normO{\widehat{\mm}_\star^{n}-\mm_\star^n}\norm{\partial_t\mm_\star^n}_{\bL^{\infty}(\Omega)}\stackrel{\eqref{eq: extrapolation error}}{\lesssim}\tau^2,
	\end{equation}
	where the hidden constant depends only on $M$.

	The bound on the first term of the right-hand side of \eqref{eq: in appendix 3} follows from a Taylor expansion of $\mm\in C^3([0,T],\bL^2(\Omega))$ around $t_{n-2}$, exactly as in \eqref{eq: time derivative error}. Indeed,
	\begin{align*}
		\normO[\Big]{\bP(\widehat{\mm}_\star^{n})&\Big[\frac{1}{\tau}\Big(\frac{3}{2}\mm_\star^n-2\mm_\star^{n-1}+\frac{1}{2}\mm_\star^{n-2}\Big)-\partial_t\mm_\star^n\Big]}\\
		&\stackrel{\eqref{eq: stability projection}}{\lesssim}\norm{\mhat_\star^n}^2_{\bL^\infty(\Omega)}\normO[\Big]{\frac{1}{\tau}\Big(\frac{3}{2}\mm_\star^n-2\mm_\star^{n-1}+\frac{1}{2}\mm_\star^{n-2}\Big)-\partial_t\mm_\star^n}.
	\end{align*}
	Arguing exactly as in $\eqref{eq: time derivative error}$, we have that
	\begin{equation*}\label{eq: time derivative error appendix}
		\normO[\Big]{\frac{1}{\tau}\Big(\frac{3}{2}\mm_\star^n-2\mm_\star^{n-1}+\frac{1}{2}\mm_\star^{n-2}\Big)-\partial_t\mm_\star^n}\lesssim \tau^2.
	\end{equation*}
	Plugging this and \eqref{eq: in appendix 4} into \eqref{eq: in appendix 3}, we derive
	\begin{equation*}\label{eq: time derivative error 2}
		\normO{\vv_\star^{n-1}-\partial_t\mm_\star^n} \lesssim \tau^2,
	\end{equation*}
	where the hidden constant depends only on $M$. This concludes the proof.
\end{proof}
\begin{lemma}[Discrete Gronwall inequality \texorpdfstring{\cite[Lemma~10.5]{thomee2013galerkin}})]\label{lem: gronwall}
Let $\left\{a_n\right\}_{n \geq 0},$\\$\left\{b_n\right\}_{n \geq 0}$, and $\left\{w_n\right\}_{n \geq 0}$ be sequences of real numbers such that
$$
a_n \leq a_{n+1}, \quad b_n \geq 0, \quad \text { and } \quad w_n \leq a_n+\sum_{j=0}^{n-1} b_j w_j \quad \text { for all } n \geq 0 .
$$
Then, it holds that
\begin{equation*}
	w_n \leq a_n \exp \sum_{j=0}^{n-1} b_j \quad \text { for all } n \geq 0.\tag*{\qed}
\end{equation*}
\end{lemma}
\renewcommand{\thetheorem}{A.\arabic{theorem}}
\setcounter{theorem}{0} 

\printbibliography
\end{document}

%% file: llg_bdf2_a_priori_estimates_numerics.tex
To support our theoretical findings, we present some numerical experiments that demonstrate the convergence 
rates $\OO(h+\tau^2)$ stated in Theorem~\ref{thm: main result}. Throughout, we restrict to $\Omega\subset \R^2$ and report 
on the empirical results obtained by our Matlab implementation of Algorithm~\ref{alg: full discr}. We stress that $d=2$ relates to the thin-film limit of LLG and thus keeps relevant physics; see \cite{davoli2020}.
\subsection{Some experiments from \cite{part1}}\label{sec: recall_part_1}
Some experiments on experimental convergence rates were already shown in our previous work \cite{part1} to anticipate 
the (formally) second-order convergence in time and the results of the present work. 
We refer to \cite[Example 4.1]{part1} for a test case with a smooth constant initial condition $\mm^0\equiv(0,1,0)$ 
evolving under the influence of a time-constant external field $\ff(x,y) = (x,y,0)/|(x,y,0)|$, 
while we refer to \cite[Example 4.2]{part1} for an experiment proposed in \cite{feischl} with a time-dependent applied field and exact solution
$$
\mm(\xx,t) \coloneqq \begin{pmatrix}
    -(x_1^3-3x_1^2/2+1/4)\sin(3\pi t/T) \\
    \sqrt{1-(x_1^3-3x_1^2/2+1/4)^2} \\
    -(x_1^3-3x_1^2/2+1/4)\cos(3\pi t/T)
\end{pmatrix}.
$$ 
Numerical experiments in \cite{part1} confirm the expected convergence rates, i.e., the first-order convergence in space and the second-order convergence in time of the $\ell^\infty(\bH^1)$-error
$$
\max_{j=1,\dots,N} \norm{\mm_h^j - \mm(t_j) }_{\bH^{1}{(\Omega)}} = \OO(h+\tau^2);
$$
see also Figure~\ref{fig:recall}.
\begin{figure}[t]
\centering

\begin{subfigure}{0.9\linewidth}
  \centering
  \begin{subfigure}[b]{0.48\linewidth}
    \includegraphics[width=\linewidth]{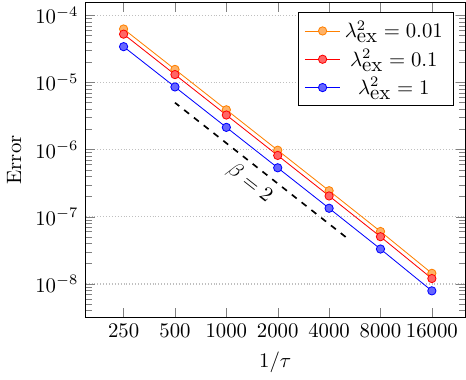}
  \end{subfigure}
  \begin{subfigure}[b]{0.48\linewidth}
    \includegraphics[width=\linewidth]{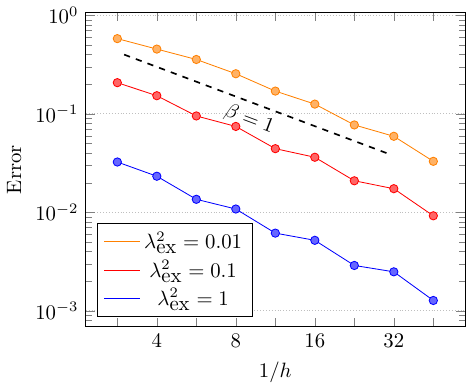}
  \end{subfigure}
  \caption{Empirical convergence rates of the $\bH^1$-error $\norm{\mm_{\rm ref}(T)-\mm_{h\tau}(T)}_{\bH^1(\Omega)}$ at final time $T$ for the first example of Section~\ref{sec: recall_part_1} (time-constant external field $\ff$).}
\end{subfigure}

\vspace{0.5em}

\begin{subfigure}{0.9\linewidth}
  \centering
  \begin{subfigure}[b]{0.48\linewidth}
    \includegraphics[width=\linewidth]{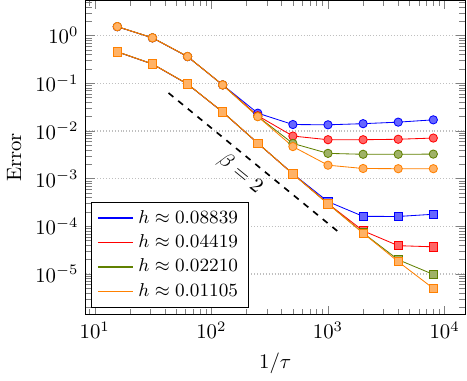}
  \end{subfigure}
  \begin{subfigure}[b]{0.48\linewidth}
    \includegraphics[width=\linewidth]{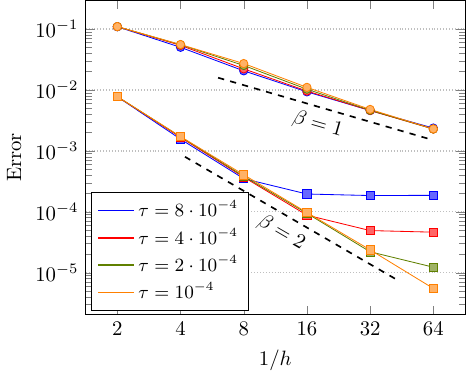}
  \end{subfigure}
  \caption{Empirical convergence rates of $\max\limits_{j=1,\dots,N} \norm{\mm(t_j) - \mm_h^j}_{\bL^2(\Omega)}$ (\protect\tikz[baseline=-0.5ex]{
  \protect\draw[line width=1pt] (-0.4,0) -- (-0.15,0); 
  \protect\draw[line width=1pt] (-0.15,-0.15) rectangle (0.15,0.15); 
  \protect\draw[line width=1pt] (0.15,0) -- (0.4,0); 
 })
 and $\max\limits_{j=1,\dots,N} \norm{\mm(t_j) - \mm_h^j}_{\bH^1(\Omega)}$
 (\protect\tikz[baseline=-0.5ex]{
  \protect\draw[line width=1pt] (-0.4,0) -- (-0.15,0); 
  \protect\draw[line width=1pt] (0,0) circle [radius=0.15]; 
  \protect\draw[line width=1pt] (0.15,0) -- (0.4,0); 
 }) for the second example of Section~\ref{sec: recall_part_1} (time dependent external field $\ff$).}
\end{subfigure}
\captionsetup{width=0.9\linewidth}
\caption{Empirical convergence rates in the experiments from Section~\ref{sec: recall_part_1}.}
\label{fig:recall}

\end{figure}
In particular, these numerical experiments show that for fixed refinement, either spatial or temporal error starts to dominate and the overall error does not decrease anymore. To prevent the spatial error to dominate over the temporal error and to equibalance the error components, one can impose $h=C\tau^2$, for some constant $C>0$. In Figure~\ref{fig:recall_bis}, we compute the convergence rates corresponding to uniform meshes $\mathcal{T}_\ell$, with mesh-sizes $h_\ell = 2^{-(2\ell+1)}$ and time-step sizes $\tau_\ell = 2^{-\ell} / 100 $, i.e., $C = 5000$.
This choice leads to optimal decay $\OO(\tau^2)$ for the error.
\begin{figure}[t]
\centering
\includegraphics[width=0.6\linewidth]{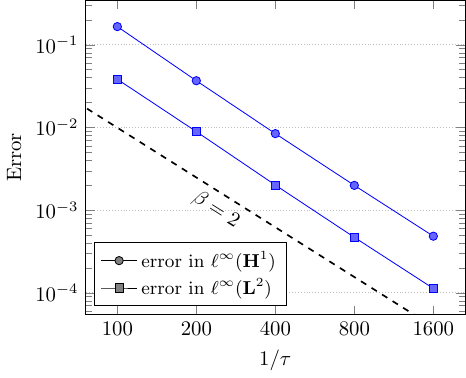}
\captionsetup{width=\textwidth}
\caption{Empirical convergence rates for $h = C \tau^2$ with $C = $5,000 in the second experiment from Section~\ref{sec: recall_part_1}, with time-dependent $\ff$.} 
\label{fig:recall_bis}
\end{figure}
\subsection{Experimental convergence rates for time-dependent applied field}\label{sec:example_2}
We consider another experiment from \cite{feischl}. We define 
$g(t)\coloneqq(T+0.1)/(T+0.1-t)$ and $d(\xx)\coloneqq~(x_1-1/2)^2+(x_2-1/2)^2$. For some $D>0$, the exact solution reads
$$
\mm(\xx,t) \coloneqq \begin{pmatrix}
    D\text{e}^{-\frac{g(t)}{1/4-d(\xx)}}(x_1-1/2)\\
    D\text{e}^{-\frac{g(t)}{1/4-d(\xx)}}(x_2-1/2)\\
    \sqrt{1-D^2 \text{e}^{-2\frac{g(t)}{1/4-d(\xx)}}d(\xx)}
\end{pmatrix}
\text{ if } d(\xx)<\dfrac{1}{4}\, \text{ and }\, \mm(\xx,t)\coloneqq\begin{pmatrix} 0\\0\\1 \end{pmatrix}\,\text{ else.}
$$ We refer to \cite{feischl} for a visualization of the evolution of $\mm$ in time. Notice that the solution is smooth in both space and time. 

We solve \eqref{eq: LLG system} for $\Omega = (0,1)^2$ and choose the parameters as in \cite{feischl}, i.e., $D\coloneqq 400$, $T=0.2$, $\alpha = 0.2$, and $\lamex^2=1.$ Moreover, we compute the external field $\ff$ by inserting the exact prescribed solution into \eqref{eq: alternative LLG}, i.e.,
$$
\ff(\xx,t) =  \alpha \partial_t\mm(\xx,t) + \mm(\xx,t)\times \partial_t\mm(\xx,t)-\lamex^2 \Delta \mm(\xx,t).
$$
To test the spatial convergence rate, we fix the time-step size $\tau = 10^{-3}$ and consider a sequence of uniform triangulations of $\Omega$ corresponding to a sequence of mesh sizes $h_{\ell} = 2^{-(\ell+1)}$ for $\ell = 0,\dots,7$, ranging from $h_0 = 1/2$ with $\#\mathcal{T}_0=16$ to $h_7 = 1/256$ with $\#\mathcal{T}_7=$262,144. In Figure~\ref{fig:example_2_conv} (right), we plot the convergence of the errors 
$$
\max\limits_{j=1,\dots,N} \norm{\mm(t_j) - \mm_h^j}_{\bL^2(\Omega)} \text{ and } \max\limits_{j=1,\dots,N} \norm{\mm(t_j) - \mm_h^j}_{\bH^1(\Omega)}
$$ over $1/h$. To test the temporal convergence rate, we fix the uniform triangulation of $\Omega$ consisting of 262,144 triangles and consider a sequence of uniform partitions of the time interval, i.e., $\tau_\ell = 2^{-\ell} /10$ for $\ell=0,\dots,7$. To avoid that the spatial error dominates over the temporal one, we require the mesh to be fine. We notice that this phenomenon is starting to appear for the lowest values of $\tau$ for the $\bL^2$-norm. In Figure~\ref{fig:example_2_conv} (left), we plot the convergence of the $\ell^\infty(\bL^2)$- and $\ell^\infty(\bH^1)$-errors over $1/\tau$.
\begin{figure}
\centering
\begin{subfigure}[b]{0.49\linewidth} 
   \centering
   \includegraphics[width=1\linewidth]{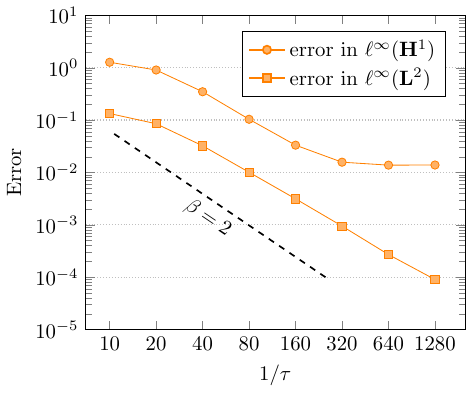}
\end{subfigure}
\begin{subfigure}[b]{0.49\linewidth}
  \centering
  \includegraphics[width=1\linewidth]{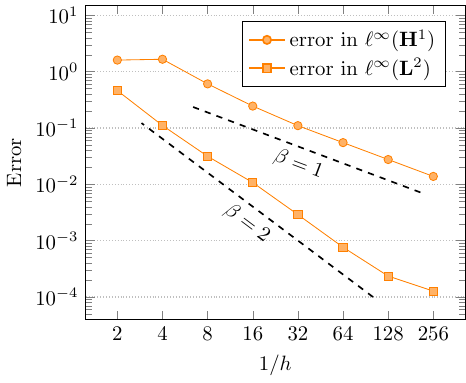}
\end{subfigure}
\captionsetup{width=\textwidth}
\caption{Empirical convergence rates of $\max\limits_{j=1,\dots,N} \norm{\mm(t_j) - \mm_h^j}_{\bL^2(\Omega)}$ (\protect\tikz[baseline=-0.5ex]{
  \protect\draw[line width=1pt] (-0.4,0) -- (-0.15,0); 
  \protect\draw[line width=1pt] (-0.15,-0.15) rectangle (0.15,0.15); 
  \protect\draw[line width=1pt] (0.15,0) -- (0.4,0); 
 })
 and $\max\limits_{j=1,\dots,N} \norm{\mm(t_j) - \mm_h^j}_{\bH^1(\Omega)}$   
 (\protect\tikz[baseline=-0.5ex]{
  \protect\draw[line width=1pt] (-0.4,0) -- (-0.15,0); 
  \protect\draw[line width=1pt] (0,0) circle [radius=0.15]; 
  \protect\draw[line width=1pt] (0.15,0) -- (0.4,0); 
 }) in the experiment from Section~\ref{sec:example_2}. Left: Convergence in time for a fixed spatial mesh with $h = 1/256$. Right: Convergence in space for a fixed time-step size $\tau = 10^{-3}$.}
\label{fig:example_2_conv}
\end{figure}
We observe that the temporal convergence rates are slightly below the expected second order in time. We conjecture that this might be due to the fact that the solution, despite being smooth in both space and time, has a term $d(t)$ which is close to be singular at $t = T$. 
To investigate this behavior further, Figure~\ref{fig:example_2_bis} shows the temporal convergence rates for an exact solution of the same form and the same parameters, but with $\tilde{g}(t) = (T + \chi)/ (T + \chi - t)$, where $\chi\in\{0.01, 0.1, 1\}$, instead of $g(t)$. 
\begin{figure}[ht!]
\centering
\includegraphics[width=0.6\linewidth]{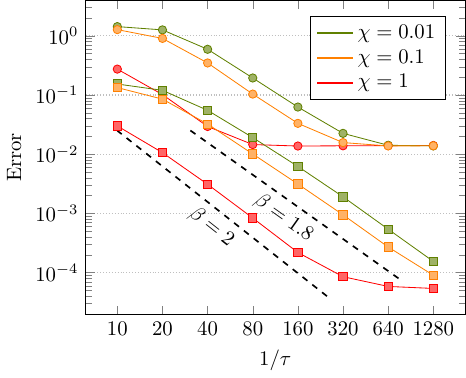}
\captionsetup{width=\textwidth}
\caption{Empirical convergence rates of $\max\limits_{j=1,\dots,N} \norm{\mm(t_j) - \mm_h^j}_{\bL^2(\Omega)}$ (\protect\tikz[baseline=-0.5ex]{
  \protect\draw[line width=1pt] (-0.4,0) -- (-0.15,0); 
  \protect\draw[line width=1pt] (-0.15,-0.15) rectangle (0.15,0.15); 
  \protect\draw[line width=1pt] (0.15,0) -- (0.4,0); 
 })
 and $\max\limits_{j=1,\dots,N} \norm{\mm(t_j) - \mm_h^j}_{\bH^1(\Omega)}$   
 (\protect\tikz[baseline=-0.5ex]{
  \protect\draw[line width=1pt] (-0.4,0) -- (-0.15,0); 
  \protect\draw[line width=1pt] (0,0) circle [radius=0.15]; 
  \protect\draw[line width=1pt] (0.15,0) -- (0.4,0); 
 }) 
 in the experiment from Section~\ref{sec:example_2} with $\tilde{g}(t) = (T + \chi)/(T + \chi - t)$.}
\label{fig:example_2_bis}
\end{figure}

We observe that the case $\chi = 1 $ leads to a better convergence rate, which is closer to the expected second-order rate, together with a better accuracy of the error, even though the stagnation of the $\ell^\infty(\bL^2)$-error appears before. On the contrary, the case $\chi = 0.01$ behaves similarly to the case of the original example $\chi = 0.1$ for the convergence rates, but with a slightly worse accuracy of the error. This seems to confirm our conjecture that the term $g(t)$, which is close to be singular at $t = T$, is responsible for the suboptimal convergence rates in time.

\subsection{Pulse effective field and high-frequency initial condition}\label{sec:example_3}

To go beyond the verification of the implementation via prescribed exact solutions, we investigate the convergence properties of Algorithm~\ref{alg: full discr} when the initial condition $\mm^0$ is chosen to be a high-frequency oscillation
\[
    \mm^0(x,y) = \frac{(0.2, \sin(4x+4y), \cos(4x+4y))}{|(0.2, \sin(4x+4y), \cos(4x+4y))|}
\]
and evolves under the influence of an external field which is constant in space and acts as a rapid pulse in time, namely
\[
    \ff(t) = (0, 0, \operatorname{sech}(10(t-0.5))).
\] 
The use of the $\operatorname{sech}$ function provides a sharp peak at $t=0.5$, forcing the scheme to adapt to a rapidly changing effective field.
For damping parameter $\alpha = 0.25$ and exchange constant $\lambda_{\text{ex}}^2 = 1$, we solve LLG~\eqref{eq: LLG system} on $\Omega = (0,1)^2$ up to the final time $T = 1$. 

To test second-order convergence in time, we fix a uniform triangulation of $\Omega$ consisting of $512$ triangles, corresponding to a mesh size $h = 2^{-7/2}\approx 0.0884$. 
We employ a uniform partition of the time interval with time-step sizes $\tau_\ell = 2^{-\ell} \tau_0$ for $\ell = 0, \ldots, 9$ with $\tau_0 = 5 \cdot 10^{-3}$. The error is calculated comparing the numerical solution to a reference solution $\mm_{\text{ref}}$ computed with $\tau_{\text{ref}} = 2^{-13} \tau_0$.

To test first-order convergence in space, we fix the time-step size $\tau = 10^{-1}$ and consider a sequence of uniform triangulations. We start from a coarse mesh with $h_0 = 2^{-3/2}\approx 0.3536$ and $\#\mathcal{T}_{0} = 32$ that is successively refined by longest edge bisection ten times, leading to $h_{10}=2^{-13/2}\approx 0.011$ and $\#\mathcal{T}_{10} =$32,768. The reference solution is obtained from a mesh refined twice further, i.e., $h_{12} =2^{-15/2}\approx 0.0055$ and $\#\mathcal{T}_{{12}} =$ 131,072. Figure~\ref{fig:example_3_conv} displays the $\bH^1$- and $\bL^2$-errors at final $T = 1$ over $1/\tau$ (left) and $1/h$ (right). We show that the scheme achieves optimal first-order convergence in space and second-order convergence in time.
\begin{figure}[ht!]
     \centering
     \begin{subfigure}[b]{0.48\linewidth}
         \centering
         \includegraphics[width=1\linewidth]{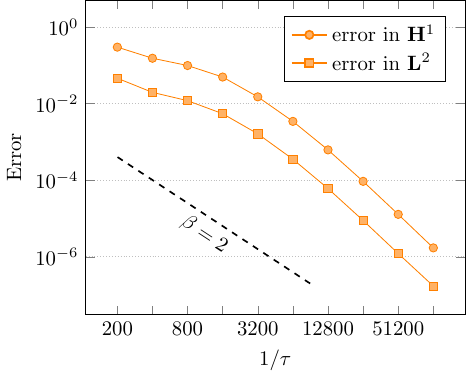}
     \end{subfigure}
     \begin{subfigure}[b]{0.49\linewidth}  
         \centering
         \includegraphics[width=1\linewidth]{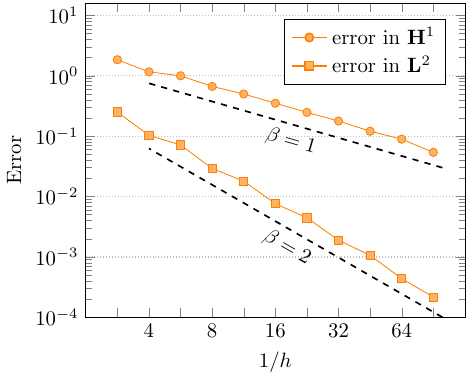}
     \end{subfigure}
     \captionsetup{width=\textwidth}
     \caption{Empirical convergence rates of the final-time errors $\norm{\mm_{\rm ref}(T) - \mm_{h\tau}(T)}_{\bL^2(\Omega)}$ (\protect\tikz[baseline=-0.5ex]{
  \protect\draw[line width=1pt] (-0.4,0) -- (-0.15,0); 
  \protect\draw[line width=1pt] (-0.15,-0.15) rectangle (0.15,0.15); 
  \protect\draw[line width=1pt] (0.15,0) -- (0.4,0); 
 })
 and $\norm{\mm_{\rm ref}(T) - \mm_{h\tau}(T)}_{\bH^1(\Omega)}$   
 (\protect\tikz[baseline=-0.5ex]{
  \protect\draw[line width=1pt] (-0.4,0) -- (-0.15,0); 
  \protect\draw[line width=1pt] (0,0) circle [radius=0.15]; 
  \protect\draw[line width=1pt] (0.15,0) -- (0.4,0); 
 }) for the experiment from Section~\ref{sec:example_3}. Left: Convergence in time for a fixed spatial mesh with $h = 2^{-7/2}\approx 0.0884$. Right: Convergence in space for a fixed time-step size $\tau = 10^{-1}$.}
     \label{fig:example_3_conv} 
\end{figure}

%% file: literature.bib
@Book{BrownB1962,
  author    = {Brown, W. F.},
  publisher = {North-Holland Publishing Company, New York},
  title     = {{Magnetostatic principles in ferromagnetism}},
  year      = {1962},
}

@book{BrownB1963,
	Author = {Brown, W F},
	Publisher = {Interscience Publishers, London},
	Title = {{Micromagnetics}},
	Year = {1963}}

@book{hubert1998magnetic,
  author    = {Alex Hubert and Rudolf Schäfer},
  title     = {Magnetic Domains: The Analysis of Magnetic Microstructures},
  publisher = {Springer},
  address   = {Berlin, Heidelberg},
  year      = {1998},
  isbn      = {978-3-540-64108-7},
  doi       = {10.1007/978-3-540-85054-0}
}

@article{LandauA1935,
  author  = {L. D. Landau and E. M. Lifshitz},
  title   = {On the theory of the dispersion of magnetic permeability in ferromagnetic bodies},
  journal = {Physikalische Zeitschrift der Sowjetunion},
  volume  = {8},
  pages   = {153--169},
  year    = {1935}
}

@article {MR4617914,
    AUTHOR = {Liu, Mingru and Huang, Pengzhan and He, Yinnian},
     TITLE = {A linearized {C}rank-{N}icolson/leapfrog scheme for the
              {L}andau-{L}ifshitz equation},
   JOURNAL = {Rocky Mountain J. Math.},
  FJOURNAL = {The Rocky Mountain Journal of Mathematics},
    VOLUME = {53},
      YEAR = {2023},
    NUMBER = {3},
     PAGES = {821--837},
      ISSN = {0035-7596,1945-3795},
   MRCLASS = {65M60},
  MRNUMBER = {4617914},
       DOI = {10.1216/rmj.2023.53.821},
       URL = {https://doi.org/10.1216/rmj.2023.53.821},
}

@article {MR3273326,
    AUTHOR = {Gao, Huadong},
     TITLE = {Optimal error estimates of a linearized backward {E}uler {FEM}
              for the {L}andau-{L}ifshitz equation},
   JOURNAL = {SIAM J. Numer. Anal.},
  FJOURNAL = {SIAM Journal on Numerical Analysis},
    VOLUME = {52},
      YEAR = {2014},
    NUMBER = {5},
     PAGES = {2574--2593},
      ISSN = {0036-1429,1095-7170},
   MRCLASS = {65M60 (35K51 35Q60 65M12 78M12)},
  MRNUMBER = {3273326},
MRREVIEWER = {Sarangam\ Majumdar},
       DOI = {10.1137/130936476},
       URL = {https://doi.org/10.1137/130936476},
}

@article {bkw2024,
    AUTHOR = {Bartels, S\"oren and Kov\'acs, Bal\'azs and Wang, Zhangxian},
     TITLE = {Error analysis for the numerical approximation of the harmonic
              map heat flow with nodal constraints},
   JOURNAL = {IMA J. Numer. Anal.},
  FJOURNAL = {IMA Journal of Numerical Analysis},
    VOLUME = {44},
      YEAR = {2024},
    NUMBER = {2},
     PAGES = {633--653},
      ISSN = {0272-4979,1464-3642},
   MRCLASS = {65M12 (65M60 80A19)},
  MRNUMBER = {4727106},
       DOI = {10.1093/imanum/drad037},
       URL = {https://doi.org/10.1093/imanum/drad037},
}

@article {MR3542789,
    AUTHOR = {An, Rong},
     TITLE = {Optimal error estimates of linearized {C}rank-{N}icolson
              {G}alerkin method for {L}andau-{L}ifshitz equation},
   JOURNAL = {J. Sci. Comput.},
  FJOURNAL = {Journal of Scientific Computing},
    VOLUME = {69},
      YEAR = {2016},
    NUMBER = {1},
     PAGES = {1--27},
      ISSN = {0885-7474,1573-7691},
   MRCLASS = {65M60 (35K55 35Q60 65M15)},
  MRNUMBER = {3542789},
MRREVIEWER = {Dmitriy\ Leykekhman},
       DOI = {10.1007/s10915-016-0181-1},
       URL = {https://doi.org/10.1007/s10915-016-0181-1},
}

@article {MR4362832,
    AUTHOR = {Cai, Yongyong and Chen, Jingrun and Wang, Cheng and Xie,
              Changjian},
     TITLE = {A second-order numerical method for
              {L}andau-{L}ifshitz-{G}ilbert equation with large damping
              parameters},
   JOURNAL = {J. Comput. Phys.},
  FJOURNAL = {Journal of Computational Physics},
    VOLUME = {451},
      YEAR = {2022},
     PAGES = {Paper No. 110831, 12},
      ISSN = {0021-9991,1090-2716},
   MRCLASS = {78M20},
  MRNUMBER = {4362832},
       DOI = {10.1016/j.jcp.2021.110831},
       URL = {https://doi.org/10.1016/j.jcp.2021.110831},
}

@article {MR4454924,
    AUTHOR = {An, Rong and Sun, Weiwei},
     TITLE = {Analysis of backward {E}uler projection {FEM} for the
              {L}andau-{L}ifshitz equation},
   JOURNAL = {IMA J. Numer. Anal.},
  FJOURNAL = {IMA Journal of Numerical Analysis},
    VOLUME = {42},
      YEAR = {2022},
    NUMBER = {3},
     PAGES = {2336--2360},
      ISSN = {0272-4979,1464-3642},
   MRCLASS = {65M60 (65M12 65M15)},
  MRNUMBER = {4454924},
       DOI = {10.1093/imanum/drab038},
       URL = {https://doi.org/10.1093/imanum/drab038},
}

@article {MR4646547,
    AUTHOR = {Guo, Qi and Xiao, Yamin},
     TITLE = {The initial-boundary value problem for the {L}andau-{L}ifshitz
              equation with {G}ilbert damping term},
   JOURNAL = {Commun. Math. Sci.},
  FJOURNAL = {Communications in Mathematical Sciences},
    VOLUME = {21},
      YEAR = {2023},
    NUMBER = {6},
     PAGES = {1727--1742},
      ISSN = {1539-6746,1945-0796},
   MRCLASS = {35Q60 (35B65)},
  MRNUMBER = {4646547},
MRREVIEWER = {Sharad\ Dwivedi},
}

@article {MR4673464,
    AUTHOR = {Cai, Yongyong and Chen, Jingrun and Wang, Cheng and Xie,
              Changjian},
     TITLE = {Error analysis of a linear numerical scheme for the
              {L}andau-{L}ifshitz equation with large damping parameters},
   JOURNAL = {Math. Methods Appl. Sci.},
  FJOURNAL = {Mathematical Methods in the Applied Sciences},
    VOLUME = {46},
      YEAR = {2023},
    NUMBER = {18},
     PAGES = {18952--18974},
      ISSN = {0170-4214,1099-1476},
   MRCLASS = {65M06 (35K61 35Q60 82D40)},
  MRNUMBER = {4673464},
       DOI = {10.1002/mma.9601},
}

@article {bartels2016,
    AUTHOR = {Bartels, S\"oren},
     TITLE = {Projection-free approximation of geometrically constrained
              partial differential equations},
   JOURNAL = {Math. Comp.},
  FJOURNAL = {Mathematics of Computation},
    VOLUME = {85},
      YEAR = {2016},
    NUMBER = {299},
     PAGES = {1033--1049},
      ISSN = {0025-5718,1088-6842},
   MRCLASS = {65J05 (65M12 65M60 65N12 65N30)},
  MRNUMBER = {3454357},
MRREVIEWER = {Waleed\ M.\ Abd-Elhameed},
       DOI = {10.1090/mcom/3008},
       URL = {https://doi.org/10.1090/mcom/3008},
}

@article {bs2006,
    AUTHOR = {Ba\v nas, \v Lubom\'ir and Slodi\v cka, Mari\'an},
     TITLE = {Error estimates for {L}andau-{L}ifshitz-{G}ilbert equation
              with magnetostriction},
   JOURNAL = {Appl. Numer. Math.},
  FJOURNAL = {Applied Numerical Mathematics. An IMACS Journal},
    VOLUME = {56},
      YEAR = {2006},
    NUMBER = {8},
     PAGES = {1019--1039},
      ISSN = {0168-9274,1873-5460},
   MRCLASS = {78A25 (65M60 74F15)},
  MRNUMBER = {2234837},
MRREVIEWER = {Nicolae\ Pop},
       DOI = {10.1016/j.apnum.2005.09.003},
}

@article {aj2006,
    AUTHOR = {Alouges, Fran\c cois and Jaisson, Pascal},
     TITLE = {Convergence of a finite element discretization for the
              {L}andau-{L}ifshitz equations in micromagnetism},
   JOURNAL = {Math. Models Methods Appl. Sci.},
  FJOURNAL = {Mathematical Models and Methods in Applied Sciences},
    VOLUME = {16},
      YEAR = {2006},
    NUMBER = {2},
     PAGES = {299--316},
      ISSN = {0218-2025,1793-6314},
   MRCLASS = {65M60 (35K55 35Q60 82D40)},
  MRNUMBER = {2210092},
MRREVIEWER = {B\'eatrice\ M.\ Rivi\`ere},
       DOI = {10.1142/S0218202506001169},
}

@article {bkp2008,
    AUTHOR = {Bartels, S\"oren and Ko, Joy and Prohl, Andreas},
     TITLE = {Numerical analysis of an explicit approximation scheme for the
              {L}andau-{L}ifshitz-{G}ilbert equation},
   JOURNAL = {Math. Comp.},
  FJOURNAL = {Mathematics of Computation},
    VOLUME = {77},
      YEAR = {2008},
    NUMBER = {262},
     PAGES = {773--788},
      ISSN = {0025-5718,1088-6842},
   MRCLASS = {82D40 (35K65 35Q60 65M12)},
  MRNUMBER = {2373179},
MRREVIEWER = {Yaniv\ Almog},
       DOI = {10.1090/S0025-5718-07-02079-0},
}

@article {akt2012,
    AUTHOR = {Alouges, F. and Kritsikis, E. and Toussaint, J.-C.},
     TITLE = {A convergent finite element approximation for Landau-Lifshitz-Gilbert equation},
   JOURNAL = {Physica B},
    VOLUME = {407},
      YEAR = {2012},
    NUMBER = {5},
     PAGES = {1345--1349},
      ISSN = {0898-1221,1873-7668},
   MRCLASS = {65M60 (82D40)},
  MRNUMBER = {3766546},
       DOI = {10.1016/j.camwa.2017.11.028},
}

@article {bsffgppr2014,
    AUTHOR = {Bruckner, F. and Suess, D. and Feischl, M. and F\"uhrer, T.
              and Goldenits, P. and Page, M. and Praetorius, D. and Ruggeri,
              M.},
     TITLE = {Multiscale modeling in micromagnetics: existence of solutions
              and numerical integration},
   JOURNAL = {Math. Models Methods Appl. Sci.},
  FJOURNAL = {Mathematical Models and Methods in Applied Sciences},
    VOLUME = {24},
      YEAR = {2014},
    NUMBER = {13},
     PAGES = {2627--2662},
      ISSN = {0218-2025,1793-6314},
   MRCLASS = {78A25 (35Q60 65N30)},
  MRNUMBER = {3260281},
       DOI = {10.1142/S0218202514500328},
}

@article {mrs2018,
    AUTHOR = {Praetorius, Dirk and Ruggeri, Michele and Stiftner, Bernhard},
     TITLE = {Convergence of an implicit-explicit midpoint scheme for
              computational micromagnetics},
   JOURNAL = {Comput. Math. Appl.},
  FJOURNAL = {Computers \& Mathematics with Applications. An International
              Journal},
    VOLUME = {75},
      YEAR = {2018},
    NUMBER = {5},
     PAGES = {1719--1738},
      ISSN = {0898-1221,1873-7668},
   MRCLASS = {65M60 (82D40)},
  MRNUMBER = {3766546},
       DOI = {10.1016/j.camwa.2017.11.028},
}

@article {ruggeri2022,
    AUTHOR = {Ruggeri, Michele},
     TITLE = {Numerical analysis of the {L}andau-{L}ifshitz-{G}ilbert
              equation with inertial effects},
   JOURNAL = {ESAIM Math. Model. Numer. Anal.},
  FJOURNAL = {ESAIM. Mathematical Modelling and Numerical Analysis},
    VOLUME = {56},
      YEAR = {2022},
    NUMBER = {4},
     PAGES = {1199--1222},
      ISSN = {2822-7840,2804-7214},
   MRCLASS = {65M60 (35K61 65M12 78M10)},
  MRNUMBER = {4444535},
       DOI = {10.1051/m2an/2022043},
}

@article {ds2014,
    AUTHOR = {Dumas, Eric and Sueur, Franck},
     TITLE = {On the weak solutions to the {M}axwell-{L}andau-{L}ifshitz
              equations and to the {H}all-magneto-hydrodynamic equations},
   JOURNAL = {Comm. Math. Phys.},
  FJOURNAL = {Communications in Mathematical Physics},
    VOLUME = {330},
      YEAR = {2014},
    NUMBER = {3},
     PAGES = {1179--1225},
      ISSN = {0010-3616,1432-0916},
   MRCLASS = {35Q60 (35D30 82D10)},
  MRNUMBER = {3227510},
MRREVIEWER = {Shintaro\ Kondo},
       DOI = {10.1007/s00220-014-1924-1},
}

@book {prohl2001,
    AUTHOR = {Prohl, Andreas},
     TITLE = {Computational micromagnetism},
    SERIES = {Advances in Numerical Mathematics},
 PUBLISHER = {Teubner, Stuttgart},
      YEAR = {2001},
     PAGES = {xviii+304},
      ISBN = {3-519-00358-9},
   MRCLASS = {82D40 (65N30 74F15 82-02 82C80 82D30)},
  MRNUMBER = {1885923},
MRREVIEWER = {Thomas\ P.\ Svobodny},
       DOI = {10.1007/978-3-663-09498-2},
}

@article {cimrak2005,
    AUTHOR = {Cimr\'ak, Ivan},
     TITLE = {Error estimates for a semi-implicit numerical scheme solving
              the {L}andau-{L}ifshitz equation with an exchange field},
   JOURNAL = {IMA J. Numer. Anal.},
  FJOURNAL = {IMA Journal of Numerical Analysis},
    VOLUME = {25},
      YEAR = {2005},
    NUMBER = {3},
     PAGES = {611--634},
      ISSN = {0272-4979,1464-3642},
   MRCLASS = {82D40 (78M10 82-04)},
  MRNUMBER = {2153750},
       DOI = {10.1093/imanum/dri011},
       URL = {https://doi.org/10.1093/imanum/dri011},
}

@article {cimrak2007,
    AUTHOR = {Cimr\'ak, Ivan},
     TITLE = {Existence, regularity and local uniqueness of the solutions to
              the {M}axwell-{L}andau-{L}ifshitz system in three dimensions},
   JOURNAL = {J. Math. Anal. Appl.},
  FJOURNAL = {Journal of Mathematical Analysis and Applications},
    VOLUME = {329},
      YEAR = {2007},
    NUMBER = {2},
     PAGES = {1080--1093},
      ISSN = {0022-247X,1096-0813},
   MRCLASS = {35Q60 (35B10 35B65 78A25 82D40)},
  MRNUMBER = {2296907},
MRREVIEWER = {Franck\ Sueur},
       DOI = {10.1016/j.jmaa.2006.06.080},
}

@article {csg1998,
    AUTHOR = {Chen, Yunmei and Ding, Shijin and Guo, Boling},
     TITLE = {Partial regularity for two-dimensional {L}andau-{L}ifshitz
              equations},
   JOURNAL = {Acta Math. Sinica (N.S.)},
  FJOURNAL = {Acta Mathematica Sinica. New Series},
    VOLUME = {14},
      YEAR = {1998},
    NUMBER = {3},
     PAGES = {423--432},
      ISSN = {1000-9574},
   MRCLASS = {35Q55 (58E50)},
  MRNUMBER = {1693128},
MRREVIEWER = {Yisong\ Yang},
       DOI = {10.1007/BF02580447},
}

@article {cf2001,
    AUTHOR = {Carbou, Gilles and Fabrie, Pierre},
     TITLE = {Regular solutions for {L}andau-{L}ifschitz equation in a
              bounded domain},
   JOURNAL = {Differ. Integral Equ.},
  FJOURNAL = {Differential and Integral Equations. An International Journal
              for Theory \& Applications},
    VOLUME = {14},
      YEAR = {2001},
    NUMBER = {2},
     PAGES = {213--229},
      ISSN = {0893-4983},
   MRCLASS = {35Q60 (35A07 35K20 82D40)},
  MRNUMBER = {1797387},
MRREVIEWER = {Bruno\ Scheurer},
}

@article {ft2017b,
    AUTHOR = {Feischl, Michael and Tran, Thanh},
     TITLE = {Existence of regular solutions of the
              {L}andau-{L}ifshitz-{G}ilbert equation in 3{D} with natural
              boundary conditions},
   JOURNAL = {SIAM J. Math. Anal.},
  FJOURNAL = {SIAM Journal on Mathematical Analysis},
    VOLUME = {49},
      YEAR = {2017},
    NUMBER = {6},
     PAGES = {4470--4490},
      ISSN = {0036-1410,1095-7154},
   MRCLASS = {35Q60 (35K55 35R60 60H15 65M12 82D40)},
  MRNUMBER = {3723324},
MRREVIEWER = {Peter\ Bernard\ Weichman},
       DOI = {10.1137/16M1103427},
}

@article {ft2017,
    AUTHOR = {Feischl, Michael and Tran, Thanh},
     TITLE = {The eddy current--{LLG} equations: {FEM}-{BEM} coupling and a
              priori error estimates},
   JOURNAL = {SIAM J. Numer. Anal.},
  FJOURNAL = {SIAM Journal on Numerical Analysis},
    VOLUME = {55},
      YEAR = {2017},
    NUMBER = {4},
     PAGES = {1786--1819},
      ISSN = {0036-1429,1095-7170},
   MRCLASS = {65M60 (35Q41 35Q60 35R60 60H15 65M12 65M15 82D45)},
  MRNUMBER = {3679317},
       DOI = {10.1137/16M1065161},
}

@article {part1,
    AUTHOR = {Ald{\'e}, Michele and Feischl, Michael and Praetorius, Dirk},
     TITLE = {BDF2-type integrator for Landau-Lifshitz-Gilbert equation in micromagnetics: unconditional weak convergence to weak solutions},
   JOURNAL = {Math. Comp., in print},
      YEAR = {2026},
}

@article {part3,
    AUTHOR = {Ald{\'e}, Michele and Feischl, Michael and Praetorius, Dirk},
     TITLE = {BDF2-type integrator for Landau-Lifshitz-Gilbert equation in micromagnetics: implicit-explicit time-stepping and full effective field},
   JOURNAL = {work in progress},
      YEAR = {2026},
}

@article {bp2006,
    AUTHOR = {Bartels, S\"oren and Prohl, Andreas},
     TITLE = {Convergence of an implicit finite element method for the
              {L}andau-{L}ifshitz-{G}ilbert equation},
   JOURNAL = {SIAM J. Numer. Anal.},
  FJOURNAL = {SIAM Journal on Numerical Analysis},
    VOLUME = {44},
      YEAR = {2006},
    NUMBER = {4},
     PAGES = {1405--1419},
      ISSN = {0036-1429,1095-7170},
   MRCLASS = {65M60 (82D40)},
  MRNUMBER = {2257110},
MRREVIEWER = {Anne\ Nouri},
       DOI = {10.1137/050631070},
}

@article {dppr2023,
    AUTHOR = {Di Fratta, Giovanni and Pfeiler, Carl-Martin and Praetorius,
              Dirk and Ruggeri, Michele},
     TITLE = {The mass-lumped midpoint scheme for computational
              micromagnetics: {N}ewton linearization and application to
              magnetic skyrmion dynamics},
   JOURNAL = {Comput. Methods Appl. Math.},
  FJOURNAL = {Computational Methods in Applied Mathematics},
    VOLUME = {23},
      YEAR = {2023},
    NUMBER = {1},
     PAGES = {145--175},
      ISSN = {1609-4840,1609-9389},
   MRCLASS = {65M60 (35K55 65M12 65M22 65Z05)},
  MRNUMBER = {4529401},
       DOI = {10.1515/cmam-2022-0060},
}

@article {hpprss2019,
    AUTHOR = {Hrkac, Gino and Pfeiler, Carl-Martin and Praetorius, Dirk and
              Ruggeri, Michele and Segatti, Antonio and Stiftner, Bernhard},
     TITLE = {Convergent tangent plane integrators for the simulation of
              chiral magnetic skyrmion dynamics},
   JOURNAL = {Adv. Comput. Math.},
  FJOURNAL = {Advances in Computational Mathematics},
    VOLUME = {45},
      YEAR = {2019},
    NUMBER = {3},
     PAGES = {1329--1368},
      ISSN = {1019-7168,1572-9044},
   MRCLASS = {65M60 (35K55 35Q60 65M12 78M10)},
  MRNUMBER = {3955721},
       DOI = {10.1007/s10444-019-09667-z},
}

@article {akst2014,
    AUTHOR = {Alouges, Fran\c cois and Kritsikis, Evaggelos and Steiner,
              Jutta and Toussaint, Jean-Christophe},
     TITLE = {A convergent and precise finite element scheme for
              {L}andau-{L}ifschitz-{G}ilbert equation},
   JOURNAL = {Numer. Math.},
  FJOURNAL = {Numerische Mathematik},
    VOLUME = {128},
      YEAR = {2014},
    NUMBER = {3},
     PAGES = {407--430},
      ISSN = {0029-599X,0945-3245},
   MRCLASS = {65M60 (35K55 35Q60 65M12)},
  MRNUMBER = {3268842},
MRREVIEWER = {Nicolae\ Pop},
       DOI = {10.1007/s00211-014-0615-3},
}

@article {dpprs2020,
    AUTHOR = {Di Fratta, Giovanni and Pfeiler, Carl-Martin and Praetorius,
              Dirk and Ruggeri, Michele and Stiftner, Bernhard},
     TITLE = {Linear second-order {IMEX}-type integrator for the (eddy
              current) {L}andau-{L}ifshitz-{G}ilbert equation},
   JOURNAL = {IMA J. Numer. Anal.},
  FJOURNAL = {IMA Journal of Numerical Analysis},
    VOLUME = {40},
      YEAR = {2020},
    NUMBER = {4},
     PAGES = {2802--2838},
      ISSN = {0272-4979,1464-3642},
   MRCLASS = {65M60 (65M12 78M10)},
  MRNUMBER = {4167063},
       DOI = {10.1093/imanum/drz046},
}

@article {bpp2015,
    AUTHOR = {Ba\v nas, L. and Page, M. and Praetorius, D.},
     TITLE = {A convergent linear finite element scheme for the
              {M}axwell-{L}andau-{L}ifshitz-{G}ilbert equations},
   JOURNAL = {Electron. Trans. Numer. Anal.},
  FJOURNAL = {Electronic Transactions on Numerical Analysis},
    VOLUME = {44},
      YEAR = {2015},
     PAGES = {250--270},
      ISSN = {1068-9613},
   MRCLASS = {65M60 (65M12)},
  MRNUMBER = {3345799},
MRREVIEWER = {Pedro\ Gonz\'alez-Casanova},
}

@article {bppr2014,
    AUTHOR = {Ba\v nas, Lubom\'ir and Page, Marcus and Praetorius, Dirk and
              Rochat, Jonathan},
     TITLE = {A decoupled and unconditionally convergent linear {FEM}
              integrator for the {L}andau-{L}ifshitz-{G}ilbert equation with
              magnetostriction},
   JOURNAL = {IMA J. Numer. Anal.},
  FJOURNAL = {IMA Journal of Numerical Analysis},
    VOLUME = {34},
      YEAR = {2014},
    NUMBER = {4},
     PAGES = {1361--1385},
      ISSN = {0272-4979,1464-3642},
   MRCLASS = {65M60 (65M12 65M15 78A30)},
  MRNUMBER = {3269429},
MRREVIEWER = {Saulo\ Pomponet\ Oliveira},
       DOI = {10.1093/imanum/drt050},
       URL = {https://doi.org/10.1093/imanum/drt050},
}

@article {dip2020,
    AUTHOR = {Di Fratta, Giovanni and Innerberger, Michael and Praetorius,
              Dirk},
     TITLE = {Weak-strong uniqueness for the {L}andau-{L}ifshitz-{G}ilbert
              equation in micromagnetics},
   JOURNAL = {Nonlinear Anal. Real World Appl.},
  FJOURNAL = {Nonlinear Analysis. Real World Applications. An International
              Multidisciplinary Journal},
    VOLUME = {55},
      YEAR = {2020},
     PAGES = {103122, 13},
      ISSN = {1468-1218,1878-5719},
   MRCLASS = {35Q60 (35A01 35Q82 82D40)},
  MRNUMBER = {4077401},
MRREVIEWER = {Gaetano\ Siciliano},
       DOI = {10.1016/j.nonrwa.2020.103122},
       URL = {https://doi.org/10.1016/j.nonrwa.2020.103122},
}

@article {feischl,
    AUTHOR = {Akrivis, Georgios and Feischl, Michael and Kov\'{a}cs,
              Bal\'{a}zs and Lubich, Christian},
     TITLE = {Higher-order linearly implicit full discretization of the
              {L}andau-{L}ifshitz-{G}ilbert equation},
   JOURNAL = {Math. Comp.},
  FJOURNAL = {Mathematics of Computation},
    VOLUME = {90},
      YEAR = {2021},
    NUMBER = {329},
     PAGES = {995--1038},
      ISSN = {0025-5718,1088-6842},
   MRCLASS = {65M60 (35Q60 65L06 65M12 65M15)},
  MRNUMBER = {4232216},
MRREVIEWER = {Hamdullah\ Y\"{u}cel},
       DOI = {10.1090/mcom/3597},
       URL = {https://doi.org/10.1090/mcom/3597},
}

@book {hairer,
    AUTHOR = {Hairer, E. and Wanner, G.},
     TITLE = {Solving ordinary differential equations. {II}},
    SERIES = {Springer Series in Computational Mathematics},
    VOLUME = {14},
   EDITION = {revised},
      NOTE = {Stiff and differential-algebraic problems},
 PUBLISHER = {Springer, Berlin},
      YEAR = {2010},
     PAGES = {xvi+614},
      ISBN = {978-3-642-05220-0},
   MRCLASS = {65-02 (65Lxx)},
  MRNUMBER = {2657217},
       DOI = {10.1007/978-3-642-05221-7},
}

@article {Alouges_Soyeur,
    AUTHOR = {Alouges, Fran\c{c}ois and Soyeur, Alain},
     TITLE = {On global weak solutions for {L}andau-{L}ifshitz equations:
              existence and nonuniqueness},
   JOURNAL = {Nonlinear Anal.},
  FJOURNAL = {Nonlinear Analysis. Theory, Methods \& Applications. An
              International Multidisciplinary Journal},
    VOLUME = {18},
      YEAR = {1992},
    NUMBER = {11},
     PAGES = {1071--1084},
      ISSN = {0362-546X,1873-5215},
   MRCLASS = {35Q99 (35D05 58E20 82D40)},
  MRNUMBER = {1167422},
MRREVIEWER = {Kotik\ K.\ Lee},
       DOI = {10.1016/0362-546X(92)90196-L},
       URL = {https://doi.org/10.1016/0362-546X(92)90196-L},
}

@article {Alouges_2008,
    AUTHOR = {Alouges, Fran\c{c}ois},
     TITLE = {A new finite element scheme for {L}andau-{L}ifchitz equations},
   JOURNAL = {Discrete Contin. Dyn. Syst. Ser. S},
  FJOURNAL = {Discrete and Continuous Dynamical Systems. Series S},
    VOLUME = {1},
      YEAR = {2008},
    NUMBER = {2},
     PAGES = {187--196},
      ISSN = {1937-1632,1937-1179},
   MRCLASS = {65M60 (35K55 82B80 82D40)},
  MRNUMBER = {2379897},
MRREVIEWER = {Etienne\ Emmrich},
       DOI = {10.3934/dcdss.2008.1.187},
       URL = {https://doi.org/10.3934/dcdss.2008.1.187},
}

@article {Spin-polarized,
    AUTHOR = {Abert, Claas and Hrkac, Gino and Page, Marcus and Praetorius,
              Dirk and Ruggeri, Michele and Suess, Dieter},
     TITLE = {Spin-polarized transport in ferromagnetic multilayers: an
              unconditionally convergent {FEM} integrator},
   JOURNAL = {Comput. Math. Appl.},
  FJOURNAL = {Computers \& Mathematics with Applications. An International
              Journal},
    VOLUME = {68},
      YEAR = {2014},
    NUMBER = {6},
     PAGES = {639--654},
      ISSN = {0898-1221,1873-7668},
   MRCLASS = {65M60 (35B45 35D30 35K51 35Q60 65M12 82D40)},
  MRNUMBER = {3253344},
MRREVIEWER = {Marco\ Picasso},
       DOI = {10.1016/j.camwa.2014.07.010},
       URL = {https://doi.org/10.1016/j.camwa.2014.07.010},
}

@phdthesis{goldenits2012konvergente,
  title={Konvergente numerische Integration der Landau-Lifshitz-Gilbert Gleichung},
  author={Goldenits, Petra},
  year={2012},
  school={TU Wien, Institute of Analysis and Scientific Computing}
}

@book {brennerscott,
    AUTHOR = {Brenner, Susanne C. and Scott, L. Ridgway},
     TITLE = {The mathematical theory of finite element methods},
    SERIES = {Texts in Applied Mathematics},
    VOLUME = {15},
   EDITION = {Third},
 PUBLISHER = {Springer, New York},
      YEAR = {2008},
     PAGES = {xviii+397},
      ISBN = {978-0-387-75933-3},
   MRCLASS = {65-01 (65-02)},
  MRNUMBER = {2373954},
       DOI = {10.1007/978-0-387-75934-0},
}

@article {davoli2020,
    AUTHOR = {Davoli, Elisa and Di Fratta, Giovanni and Praetorius, Dirk and
              Ruggeri, Michele},
     TITLE = {Micromagnetics of thin films in the presence of
              {D}zyaloshinskii-{M}oriya interaction},
   JOURNAL = {Math. Models Methods Appl. Sci.},
  FJOURNAL = {Mathematical Models and Methods in Applied Sciences},
    VOLUME = {32},
      YEAR = {2022},
    NUMBER = {5},
     PAGES = {911--939},
      ISSN = {0218-2025,1793-6314},
   MRCLASS = {78A30 (35C20 35Q51 49J45 49S05 65M12 65M60 82D40)},
  MRNUMBER = {4430360},
       DOI = {10.1142/S0218202522500208},
}

@article {AnLiSun2024,
    AUTHOR = {An, Rong and Li, Yonglin and Sun, Weiwei},
     TITLE = {Optimal error analysis of the normalized tangent plane {FEM}
              for {L}andau-{L}ifshitz-{G}ilbert equation},
   JOURNAL = {IMA J. Numer. Anal.},
  FJOURNAL = {IMA Journal of Numerical Analysis},
    VOLUME = {45},
      YEAR = {2025},
    NUMBER = {5},
     PAGES = {3109--3137},
      ISSN = {0272-4979,1464-3642},
   MRCLASS = {65M12 (65M60)},
  MRNUMBER = {4966454},
       DOI = {10.1093/imanum/drae084},
       URL = {https://doi.org/10.1093/imanum/drae084},
}

@article {Dahlquist,
    AUTHOR = {Dahlquist, Germund},
     TITLE = {{$G$}-stability is equivalent to {$A$}-stability},
   JOURNAL = {BIT},
  FJOURNAL = {BIT. Nordisk Tidskrift for Informationsbehandling (BIT)},
    VOLUME = {18},
      YEAR = {1978},
    NUMBER = {4},
     PAGES = {384--401},
      ISSN = {0006-3835},
   MRCLASS = {65L05 (65J10)},
  MRNUMBER = {520750},
MRREVIEWER = {Ramon\ E.\ Moore},
       DOI = {10.1007/BF01932018},
       URL = {https://doi.org/10.1007/BF01932018},
}

@article{baiocchi_crouzeix,
author = {Baiocchi, Claudio and Crouzeix, Michel},
title = {On the equivalence of A-stability and G-stability},
year = {1989},
issue_date = {February 1989},
publisher = {Elsevier Science Publishers B. V.},
address = {NLD},
volume = {5},
number = {1–2},
issn = {0168-9274},
journal = {Appl. Numer. Math.},
month = feb,
pages = {10–22},
numpages = {13},
DOI = {10.1016/0168-9274(89)90020-2},
}

@article {nevanlinna_odeh,
    AUTHOR = {Nevanlinna, Olavi and Odeh, F.},
     TITLE = {Multiplier techniques for linear multistep methods},
   JOURNAL = {Numer. Funct. Anal. Optim.},
  FJOURNAL = {Numerical Functional Analysis and Optimization. An
              International Journal},
    VOLUME = {3},
      YEAR = {1981},
    NUMBER = {4},
     PAGES = {377--423},
      ISSN = {0163-0563,1532-2467},
   MRCLASS = {65L05},
  MRNUMBER = {636736},
MRREVIEWER = {Peter\ Alfeld},
       DOI = {10.1080/01630568108816097},
       URL = {https://doi.org/10.1080/01630568108816097},
}

@book{ciarlet,
author = {Ciarlet, Philippe G.},
title = {The finite element method for elliptic problems},
publisher = {Society for Industrial and Applied Mathematics},
year = {2002},
doi = {10.1137/1.9780898719208},
address = {},
edition   = {},
%URL = {https://epubs.siam.org/doi/abs/10.1137/1.9780898719208},
%eprint = {https://epubs.siam.org/doi/pdf/10.1137/1.9780898719208}
}

@book{thomee2013galerkin,
  author    = {Vidar Thom{\'e}e},
  title     = {Galerkin finite element methods for parabolic problems},
  edition   = {2},
  series    = {Springer Series in Computational Mathematics},
  volume    = {25},
  publisher = {Springer},
  address   = {Berlin, Heidelberg},
  year      = {2006},
  isbn      = {978-3-540-33122-3},
  doi       = {10.1007/3-540-33122-0}
}

@misc{guo2026NSLLG,
      title={Compressible Navier-Stokes-Landau-Lifshitz-Gilbert system: derivations and well-posedness}, 
      author={Boling Guo and Ning Jiang and Hui Liu and Yi-Long Luo and Teng-Fei Zhang},
      year={2026},
      eprint={2604.20265},
      archivePrefix={arXiv},
      primaryClass={math.AP},
      %url={https://arxiv.org/abs/2604.20265}, 
}

@misc{Li2026,
      title={Stability and error analysis of fully discrete original energy-dissipative and length-preserving scheme for the Landau-Lifshitz-Gilbert equation}, 
      author={Binghong Li and Xiaoli Li and Cheng Wang and Jiang Yang},
      year={2026},
      eprint={2602.07571},
      archivePrefix={arXiv},
      primaryClass={math.NA},
      %url={https://arxiv.org/abs/2602.07571}, 
}

@misc{XieWang2025,
      title={Error analysis of third-Order in time and fourth-order linear finite difference scheme for Landau-Lifshitz-Gilbert equation under large damping parameters}, 
      author={Changjian Xie and Cheng Wang},
      year={2025},
      eprint={2510.25172},
      archivePrefix={arXiv},
      primaryClass={math.NA},
      %url={https://arxiv.org/abs/2510.25172}, 
}

@misc{Xie2025,
      title={A novel third-order accurate and stable scheme for micromagnetic simulations}, 
      author={Changjian Xie},
      year={2025},
      eprint={2511.21047},
      archivePrefix={arXiv},
      primaryClass={math-ph},
      %url={https://arxiv.org/abs/2511.21047}, 
}

@article {abp24,
    AUTHOR = {Akrivis, Georgios and Bartels, S\"oren and Palus, Christian},
     TITLE = {Quadratic constraint consistency in the projection-free
              approximation of harmonic maps and bending isometries},
   JOURNAL = {Math. Comp.},
  FJOURNAL = {Mathematics of Computation},
    VOLUME = {94},
      YEAR = {2025},
    NUMBER = {355},
     PAGES = {2251--2269},
      ISSN = {0025-5718,1088-6842},
   MRCLASS = {65M22 (35J50 35J57 35J62 65M60)},
  MRNUMBER = {4919561},
       DOI = {10.1090/mcom/4035},
       URL = {https://doi.org/10.1090/mcom/4035},
}
